\documentclass[12pt]{amsart}
\usepackage{amsmath, amssymb, amsthm,latexsym,overpic}
\usepackage{graphicx}
\usepackage{enumerate,mathtools}
\usepackage{enumitem}
\usepackage[pdftex]{hyperref}

\title[Uniformly branching  trees]{Uniformly branching  trees}%
\author{Mario Bonk}\thanks{M.B.\ was partially supported by NSF grant DMS-1808856.}
\author{Daniel Meyer}
\address{Department of Mathematics, University of California, 
Los Angeles, CA 90095, USA}
\email{mbonk@math.ucla.edu}

\address{Department of Mathematical Sciences, University of Liverpool,
Mathematical Sciences Building,  Liverpool L69 7ZL,  United Kingdom}
\email{dmeyermail@gmail.com}

\date{\today}

\setenumerate{itemsep=3pt,topsep=3pt}
\setenumerate[1]{label=\upshape(\roman*)}

\newcounter{mylistnum}

\newcommand{\mybar}[1]{\makebox[0pt]{$\phantom{#1}\overline{\phantom{#1}}$}#1}

\newcommand\C{{\mathbb C}}

\newcommand\N{{\mathbb N}}

\newcommand\dist{\operatorname{dist}}
\newcommand\diam{\operatorname{diam}}
\newcommand\id{\operatorname{id}}
\newcommand\inte{\operatorname{int}}

\renewcommand\:{\colon}
\newcommand\sub {\subset}
\newcommand\ra {\rightarrow}

\def\length{\mathop{\mathrm{length}}}

\newcommand\la{\lambda}

\newcommand\no{\noindent} 
 
\newcommand{\be}{e}
\providecommand{\abs}[1]{\lvert#1\rvert}

\newcommand{\X} {\mathbf{X}}
\newcommand{\Y} {\mathbf{Y}}

\newcommand{\V} {\mathbf{V}}

\newcommand{\qT} {\mathbf{T}}

\newcommand{\cT} {\mathbb{T}}

\newcommand{\cX} {\mathbb{X}}

\newcommand{\cV} {\mathbb{V}}

\newcommand{\A}{\mathcal{A}}

\newtheorem{theorem}{Theorem}[section]
\newtheorem{problem}[theorem]{Problem}

\newtheorem{proposition}[theorem]{Proposition}
\newtheorem{cor}[theorem]{Corollary}

\newtheorem{lemma}[theorem]{Lemma}

\theoremstyle{definition}

\newtheorem{definition}[theorem]{Definition}

\theoremstyle{remark}

\numberwithin{equation}{section}

\begin{document}

\abstract{A  quasiconformal tree $T$ is a (compact) metric tree that is
  doubling and of bounded turning. We call $T$ trivalent
  if every branch point of  $T$ has exactly three branches. If
  the set of branch points is uniformly relatively separated and
  uniformly relatively dense, we say that  $T$ is uniformly
    branching. We prove that a metric space $T$ is
  quasisymmetrically equivalent to the continuum
    self-similar tree   if and only if it is a trivalent
 quasiconformal tree that is uniformly branching. In particular, 
 any  two
   trees of this type are
  quasisymmetrically equivalent.}
%  
%   In order to prove these results, 
%   we  introduce the concept of a \emph{quasi-visual subdivision}
%   of a metric space. We obtain
%  a criterion when two metric spaces are quasisymmetrically equivalent  in 
%  terms of subdivisions of this type.}
  
\endabstract

\maketitle

\tableofcontents

\section{Introduction}\label{s:Intro}

\no 
This paper can be seen as part of  a program promoted by the present authors and other researchers to 
understand the geometry of low-dimen\-sio\-nal fractals, in particular fractals that arise in some natural dynamical settings such as 
limits sets of Kleinian groups, Julia sets of rational maps, or attractors of iterated function systems. One wants to characterize these spaces up to a natural  equivalence given by homeomorphisms with good geometric control. 

 A relevant class
 in this context are homeomorphisms that distort relative distances in a controlled way, namely {\em quasisymmetric} homeomorphisms (for a review of the relevant 
 definition see Section~\ref{sec:appr-subd}; general background on quasisymmetries and related concepts can be found in \cite{He}). We call two metric spaces $S$ and $T$ {\em quasisymmetrically equivalent} if there exists a quasisymmetric homeomorphisms  
 from  $S$ onto $T$. 
In this case, the spaces $S$ and $T$ are obviously homeomorphic.  Quasisymmetric equivalence gives a stronger notion of equivalence of metric spaces that respects not only topological properties
of the spaces, but their {\em quasiconformal geometry}. This term 
\ refers to geometric properties of a space that are of a robust scale-invariant nature. Mostly, these conditions  do not involve 
{\em absolute} distances, but rather {\em relative} distances, i.e., 
 ratios of distances.  Typically, such conditions are invariant under quasisymmetries.

The problem of quasisymmetric equivalence   has been studied for various types of spaces (see \cite{Bo06} for a general overview). For example, David and Semmes \cite{DS}  gave a  characterization of the standard $1/3$-Cantor set
up to quasisymmetric equivalence. A similar characterzation for the unit interval $[0,1]$ or the unit circle is due to Tukia and V\"ais\"al\"a \cite{TV} (see 
 also \cite{Ro, Me11, HM}). 
The problem of characterizing the standard $2$-sphere up to quasisymmetric equivalence is particularly interesting because of the connection of this problem to questions in complex dynamics and geometric group theory (see \cite{BM}  for more discussion on  the background  and references to the literature).

The goal of the present paper is to investigate the quasiconformal geometry of a particular 
 tree-like space with a self-similar structure and  very regular branching behavior, name\-ly 
the {\em continuum self-similar tree} $\cT$ (abbreviated as CSST). 
Under this name,   it was introduced   in \cite{BT}
and defined    as the  attractor of an iterated function system
in the complex plane (see Figure~\ref{fig:CSST} for an illustration). 
Sets homeomorphic to $\cT$ were considered  in the literature  before in various contexts (see  \cite{BT} for more discussion and references).

The major novelty of this paper is the insight that in  order to give a characterization of the CSST up to quasisymmetric equivalence, a new condition is relevant that gives quantitative control for how  branch points in a (metric) tree are  
 spread out. We call such trees {\em uniformly branching} (see Definition~\ref{def:unif_branching} below, which is based on Definitions~\ref{def:unif_sepa} and~\ref{def:unif_dense}).
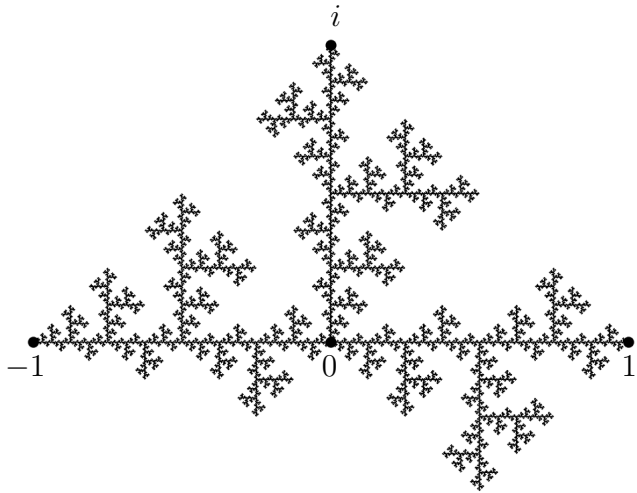
\begin{figure}
%\vspace{-0.5cm}
 \begin{overpic}[ scale=0.7
 , clip=true, trim=0mm 15mm 0mm 0mm
 % ,width=10cm, tics=10,
  %   , grid
    ]{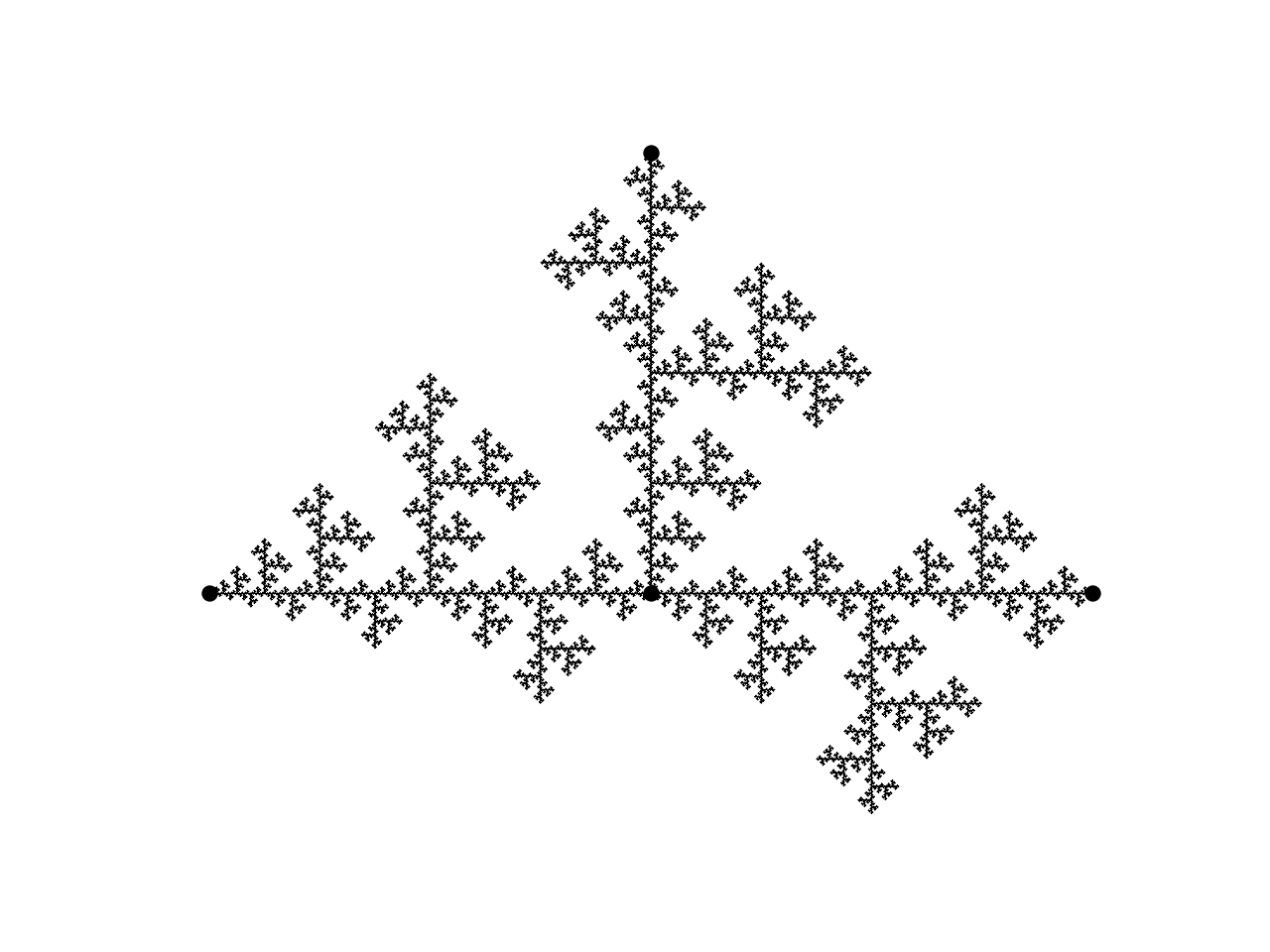}
    \put(50,56){ $i$}
    \put(12,15){ $-1$}
     \put(84,15){ $1$}
       \put(49,15){ $0$}
           \end{overpic}
 % \vspace{-1cm}
\caption{The continuum self-similar tree $\cT$.}
\label{fig:CSST}
\end{figure}
In order to formulate our main result, we will  now discuss the relevant concepts
in more detail.  We start with the CSST $\cT$. 

Up to bi-Lipschitz equivalence, one can define an abstract 
 version of $\cT$ as follows. We start with a  line segment $J_0$ of length $2$. Its  midpoint $c$ subdivides  $J_0$ into two line segments of length $1$. We glue 
to $c$ one of the endpoints of another line segment $s$ of the same length. Then we obtain a tripod-like set  $J_1$ consisting of three line segments of length $1$. The set $J_1$ carries the natural  path metric. We now repeat this procedure inductively. At the $n$-th step we obtain a simplicial tree $J_n$ consisting of 
$3^n$ line segments of length $2^{1-n}$. To pass to $J_{n+1}$, each of these line segments $s$ is subdivided by its midpoint $c_s$ into two line segment of length $2^{-n}$ and we glue to $c_s$ one endpoint of another line segment of length
 $2^{-n}$.  One can actually realize $J_n$ as a subset of the complex plane (equipped with the induced path metric).
 See \cite{BT} for more discussion and Figure~\ref{fig:J5} for an illustration of $J_5$.

In this way, we obtain  an ascending sequence $J_0\sub J_1\sub \ldots $  of (simplicial)  trees equipped with  a geodesic metric
(i.e., a metric satisfying \eqref{eq:geodmetr} below). 
The union $J=\bigcup_{n\in \N_0}J_n $ carries a natural path metric $\varrho$ that agrees with the metric on $J_n$ for each $n\in \N_0$. As an abstract space  one can now define $\cT$  as the completion of the metric 
space $(J,\varrho)$. For the equivalent definition of $\cT$ as the attractor of an iterated function system see Section~\ref{sec:cont-self-simil}.

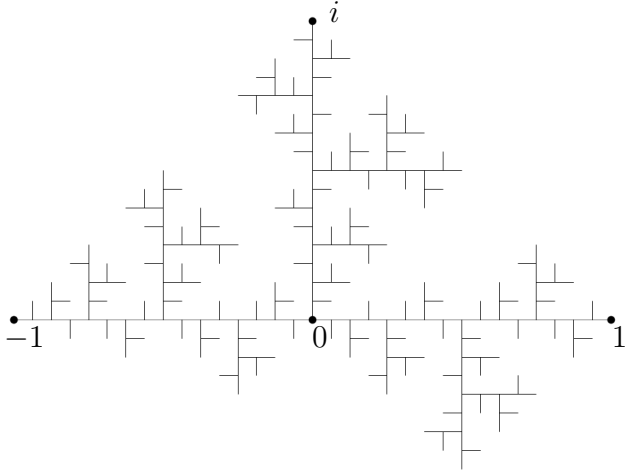
\begin{figure}
%\vspace{-0.7cm}
 \begin{overpic}[ scale=0.7
%,clip=true, trim=0mm 0mm 0mm 10mm
  %,width=10cm, tics=10,
     %, grid
    ]{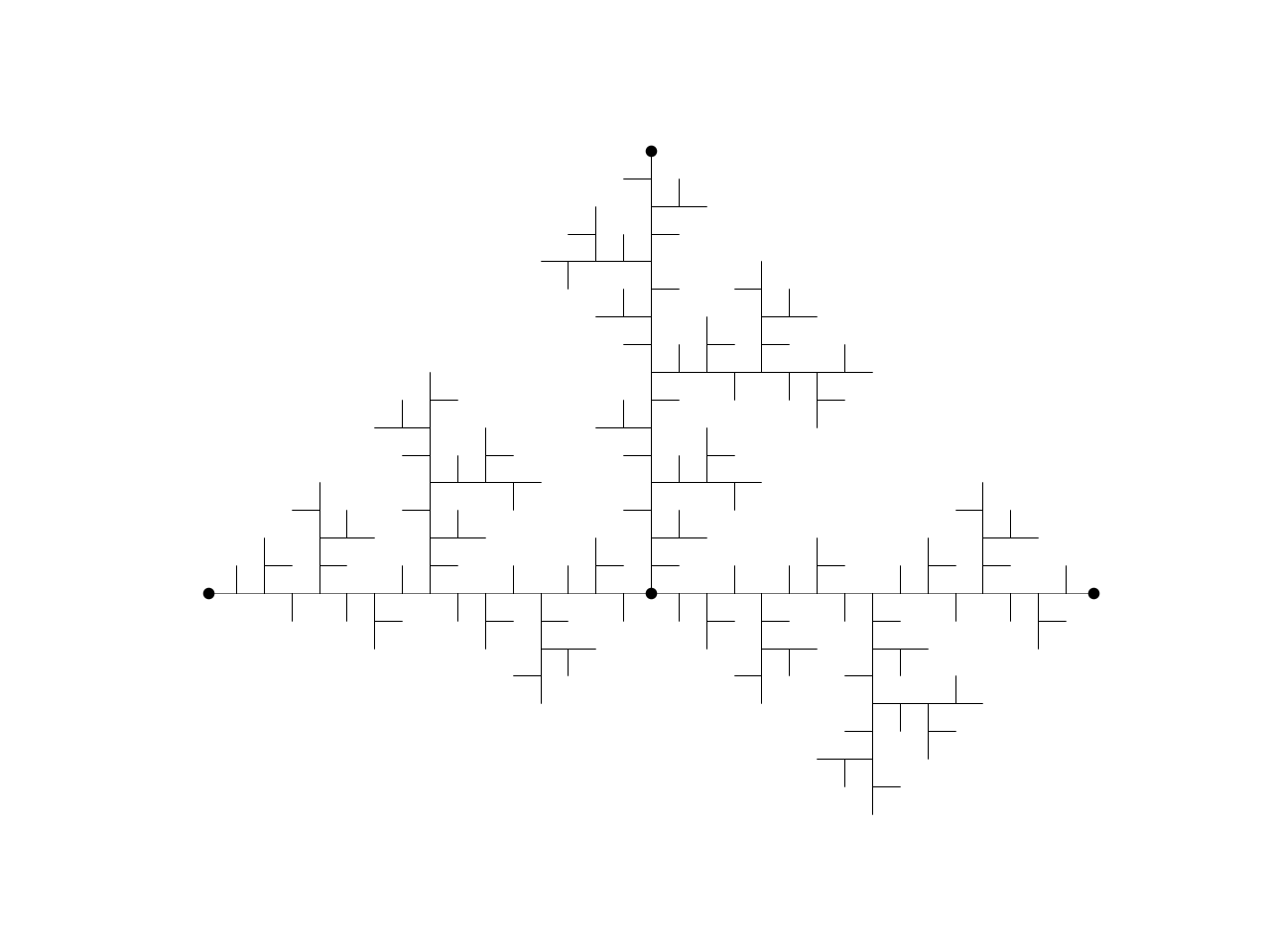}
    \put(52,63){ $i$}
    \put(14,25){ $-1$}
     \put(85,25){ $1$}
       \put(50,25){ $0$}
           \end{overpic}
  \vspace{-1cm}
\caption{The set $J_5$.}
\label{fig:J5}
\end{figure}

In order to describe the topological properties of $\cT$, we have to recall some terminology.
 By definition a {\em metric tree} is a compact, connected, and locally
connected metric  space $T$ that contains  at least two
distinct points and has  the following property: if $x,y\in T$, then
there exists a unique (possibly degenerate) arc in $T$ with endpoints $x$ and
$y$. We denote this unique arc in $T$ by $[x,y]$.

 Let $T$ be a metric tree,  and  $p\in T$. The closure $B$ of a connected component 
 $U$ of $T\setminus\{p\}$ has the form $B=\overline U=U\cup \{p\}$ and is called a {\em branch} of $p$ in $T$. A point $p\in T$ can have at most countably many branches. If there are at least three branches of $p$ in $T$, then $p$ is called a 
 {\em branch point} and a {\em triple point} if there exactly three branches. 
 We say that the metric tree $T$ is {\em trivalent} if every branch point of $T$ is a triple point. 
 
 In \cite{BT} it was shown that the CSST is a trivalent metric tree with a dense set of triple points.  This seems obvious from the abstract definition of $\cT$ outlined above, but is  harder to show if one introduces $\cT$ as the attractor of an iterated function system. The CSST $\cT$ is characterized  by these properties up to homeomorphism \cite[Theorem~1.7]{BT}: {\em a metric space $T$ is homeomorphic to $\cT$ if and only if $T$ is a trivalent metric tree with a dense set of triple points}. 
 
 In a sense, this gives an essentially complete understanding of the {\em topological}
 properties of the CSST. In contrast, in the present paper we are interested in the  {\em geometric} properties of $\cT$, in particular its {\em quasiconformal geometry}. More specifically, 
 the goal of the present paper is to give a characterization of the CSST up to quasisymmetric equivalence. Here it does not matter whether we use the abstract 
version of the CSST (whose construction was outlined above) or its representation 
$\cT\sub \C$ as
an  attractor of an  iterated function system, because both versions of the CSST are bi-Lipschitz and hence quasisymmetrically equivalent.  For concreteness we will actually use the latter representation of the  CSST (see Section~\ref{sec:cont-self-simil}). 

It is clear that the requirements  relevant for this characterization of the CSST have to go beyond mere topology. One would expect scale-invariant conditions involving relative distances. Some of the  necessary conditions will not come as a surprise to the informed reader. We recall some relevant terminology. 
    
We say that a metric tree $T$ is of {\em bounded turning} if
there exists a constant $K\ge 1$ such that
$$ \diam  [x,y]\le K |x-y|$$ 
for all $x,y\in T$. Here and elsewhere in this paper we use  Polish notation $|x-y|$ to denote the distance of two points $x$ and $y$ in  a space with a given underlying metric. 

A metric space $T$ is {\em doubling} if there
exists a constant $N\in \N$ such that every ball in $T$ of radius
$r>0$ can be covered by $N$ or fewer balls of radius $r/2$.

The CSST has these properties (see Section~\ref{sec:cont-self-simil}), and they are invariant under quasisymmetries.
Accordingly, they are necessary conditions for  a tree $T$ to be  quasisymmetrically equivalent to the CSST. It is useful to  introduce a name for such trees. 
We call a metric tree $T$ a {\em  quasiconformal tree} if
it is doubling and  of bounded turning. 
%For brevity, we  often call a 
%quasiconformal tree simply a {\em qc-tree}. 

Quasiconformal trees were considered before in our paper \cite{BM19}. There we showed
that every quasiconformal tree is quasisymmetrically equivalent to a tree with a geodesic metric. 

%This can be seen as a generalization of a well-known result due Tukia and V\"ais\"al\"a \cite{TV}: {\em  a metric arc $\alpha$ is quasisymmetrically equivalent to the unit interval $[0,1]$ if and only if $\alpha$ is doubling and of bounded turning}.  

We have seen above that a  necessary condition  for a metric space $T$ to be quasisymmetrically equivalent to the CSST is that $T$ is a trivalent quasiconformal tree  with a dense set of branch points. It is easy to see that  
this is only necessary, but in general not sufficient.  For a sufficient condition one would expect a more quantitative  version of the density of branch points.  As we will see momentarily,  one also has to stipulate that branch points are quantitatively separated in a suitable sense. 
 
To formulate this precisely, we need a quantitative measure of
the size of a branch point $p$ in a (not
necessarily trivalent) metric tree $T$.  Only finitely many of the
branches of $p$ can have a diameter exceeding a given positive
number (see \cite[Section~3]{BT} for more details). This implies
that we can label the branches $B_n$ of $p$ by numbers
$n=1,2,3,\dots$ so that
\begin{equation*}
  \diam(B_1)
  \ge
  \diam(B_2)\ge \diam(B_3)
  \ge \dots\, . 
\end{equation*}
Then we  define the \emph{height} $H_T(p)$ of $p$ in $T$ as
\begin{equation}\label{eq:HT}
  H_{T}(p)=\diam(B_3). 
\end{equation} 
So $H_T(p)$ is the diameter of the third largest branch of $p$.

Note that when $T$ is trivalent, we have
\begin{equation}
  \label{eq:def_weight}
  H_T(p)
  =
  \min\{\diam(B) : B \text{ is a branch of $p$ in $T$}\}.  
\end{equation}

% trees $T$ this agrees with the definition
% \eqref{eq:def_weight} given earlier. 

% In this case,
% $p$ has precisely three branches $B_1$, $B_2$, $B_3$. We may assume that they are labeled so that $\diam(B_1)\ge \diam(B_2)\ge \diam(B_3)$. We then define the  \emph{height}  $H_T(p)$ of $p$ in $T$ 
%  as 
% \begin{equation}
%   \label{eq:def_weight}
%   H_T(p)\coloneqq \diam(B_3). \end{equation}
% So  in a trivalent tree $T$  the height of a branch point $p$ is the smallest diameter of any of the three branches of $p$ in $T$.

 We can now define the relevant concepts. 
 \begin{definition}[Uniform relative  separation of branch points]
  \label{def:unif_sepa}
  A %tri\-valent
  tree $T$ is said to have \emph{uniformly relatively separated}
  branch points if
  \begin{equation*}
    \label{eq:unif_sepa}
    \abs{p-q}\gtrsim \min\{H_T(p), H_T(q)\}
  \end{equation*}
  for all distinct branch points $p,q\in T$ with $C(\gtrsim)$
  independent of $p$ and $q$.    
\end{definition}
The reference  $C(\gtrsim)$ to the implicit multiplicative constant here is explained in the subsection on  notation on page \pageref{sec:notation}. 

\begin{definition}[Uniform relative density of branch points]
  \label{def:unif_dense}
  A
  %trivalent
  tree $T$ is said to have \emph{uniformly relatively 
    dense} branch points if  for all $x,y\in T$, $x\ne y$,  there exists  a branch point $p\in
  [x,y]$ with 
  \begin{equation*}
    \label{eq:unif_dense}
    H_T(p) \gtrsim \abs{x-y},
  \end{equation*}
  where  $C(\gtrsim)$ independent of $x$ and $y$. 
\end{definition}

These definitions lead to the most  important concept of this paper. 
\begin{definition}[Uniformly branching trees]
  \label{def:unif_branching}
A 
tree  is called \emph{uniformly branching} if its
branch points are uniformly relatively separated and uniformly relatively dense.
\end{definition}  
Now our main result can be formulated as follows. 

\begin{theorem}
  \label{thm:CSST_qs}
  A metric space $T$ is quasisymmetrically equivalent 
  to the continuum  self-similar tree $\cT$ if and only if
  $T$ is a trivalent quasiconformal tree that is uniformly branching. 
\end{theorem}

An immediate consequence is the following fact. 

\begin{cor}
  \label{cor:unif_qs}
  Let $S$ and $T$ be two trivalent quasiconformal trees that are
  uniformly branching. Then $S$ and $T$  are quasisymmetrically equivalent. 
\end{cor}
\begin{proof}
  Indeed, by  Theorem~\ref{thm:CSST_qs} there exist quasisymmetries 
  $\varphi\: S\ra \cT$ and $\psi\: T\ra \cT$. Then $\psi^{-1}\circ \varphi$ is a quasisymmetry of $S$ onto $T$, and so $S$ and $T$ are quasisymmetrically equivalent.   \end{proof} 
  
  One can outline the proof of Theorem~\ref{thm:CSST_qs} as follows. The ``only if'' part  of the statement is not hard to show, because the CSST is 
  a trivalent  quasiconformal tree that is uniformly branching (see Proposition~\ref{prop:CSST_unif}).  Moreover, the relevant  conditions are  invariant under quasisymmetries.
  This is well-known for the doubling and the bounded turning conditions which are the basis for the definition of a quasiconformal tree. It is also true 
   that uniform relative separation and uniform relative density
   of branch points are invariant  under quasisymmetries (see
   Lemmas~\ref{lem:unif_sepa_qs}
   and~\ref{lem:unif_dense_qs}). This is not  surprising, because 
%   This  easy to believe, because
   these are conditions in terms of relative distances: 
  heights of branch points versus distances of points, both of which are measured in  ``units''  of length. A quasisymmetry gives good control
  for the distortion of such relative distances. As we will see in 
  Section~\ref{sec:unif-separ-dens}, it is completely straightforward to implement these ideas rigorously.   
  
 It is much more difficult to prove the ``if'' part of  Theorem~\ref{thm:CSST_qs}.
 The basic strategy is very similar to ideas in the recent papers
 \cite{BT} and \cite{BM19}. For the given tree $T$ one wants to create a ``subdivision'' by finer and finer decompositions of $T$. This subdivision is represented  by a sequence 
  $\{\X^n\}_{n\in \N_0}$ of decompositions 
  of $T$ given  by finite covers    $\X^n$, $n\in \N_0$,  of $T$   by compact subsets. We call the sets $X\in \X^n$  the {\em tiles} of level $n$ of the decomposition. We will obtain these tiles by cutting $T$ 
  at the set $\V^n$ of all branch points  with heights $\ge \delta^n$, where $\delta\in (0,1)$ is a small parameter. More precisely, the tiles  in  $ \X^n$ are the closures of the components of $T\setminus \V^n$.  Since $\V^n\sub \V^{n+1}$, the tiles of level $n$ are subdivided into smaller  tiles of level  $n+1$. 
  
  If we assume that $T$ is a uniformly branching trivalent tree,
  then the tiles in our subdivision have very good geometric 
  properties. For example, if $X$ is a tile of level $n$, then $\diam(X)\asymp \delta^n$.  Moreover, $\X^n$ for $n\in \N_0$ is an {\em edge-like} decomposition of $T$ in the sense that each tile $X\in \X^n$ has at most  two boundary points (see Proposition~\ref{prop:decomp}). This implies 
  that the incidence relations  of the tiles $X\in \X^n$ are  the same as   the incidence relations  of the edges in a suitable finite trivalent simplicial tree.

One wants to realize a similar subdivision with the same combinatorics for the CSST 
$\cT$. Once this is achieved, then under some mild extra
conditions there exists a homeomorphism $F\: T\ra \cT$ that maps
the tiles   in the subdivision of $T$ to corresponding tiles in
the subdivision  of $\cT$. One can show that this homeomorphism
is a quasisymmetry if the tiles in the corresponding subdivisions
satisfy suitable geometric conditions. We condense the relevant conditions into the concept of a {\em quasi-visual approximation} and {\em quasi-visual subdivision} of a space (see Section~\ref{sec:appr-subd} for precise definitions). 

Roughly speaking, if we have quasi-visual approximations for two spaces that correspond under a homeomorphism, then this homeomorphism is a quasisymmetry. 
Related ideas have been used by us and other authors  before to prove that a homeomorphism
is a quasisymmetry. We formulate this in a general setting in 
 Section~\ref{sec:appr-subd}. We hope that it clarifies  some of the earlier approaches and may be of independent interest. 

 In view of this, one would like our subdivision $\{\X^n\}$ of
 $T$ to be quasi-visual. This is always the case under the given
 hypotheses (see
 Proposition~\ref{prop:decomp}~\ref{item:decomp5}).  In order to
 find a quasi-visual subdivision of $\cT$ corresponding to  the subdivision $\{\X^n\}$ of $T$, one uses an
 inductive process.  The decomposition of $\cT$ on level $n$ is
 given by the image sets $F^n(X)$, $X\in \X^n$, of an auxiliary
 homeomorphism $F^n\: T\ra \cT$. Note that we use the  superscript
 $n$ here to indicate  the level (and not the $n$-th iterate of some map $F$). For the inductive step, we
 subdivide the tile $F^n(X)$ of $\cT$ for each $X\in \X^n$  in the same way as $X$ is
 subdivided by tiles in $\X^{n+1}$. 

 In this way, one obtains a   subdivision  $\{F^n(\X^n)\}$ of
 $\cT$ with the same combinatorics as the subdivision $\{\X^n\}$
 of $T$. As we will see, one can realize the isomorphism between
 $\{\X^n\}$ and $\{F^n(\X^n)\}$ by a single homeomorphism
 $F\: T \ra\cT$ such that $F(X)=F^n(X)$ for all $X\in \X^n$.
 Then $\{\X^n\}$ and $\{F^n(\X^n)\}=\{F(\X^n)\}$ are in a suitable sense isomorphic subdivisions
 of $T$ and $\cT$, respectively. The
 subdivision $\{\X^n\}$ is quasi-visual, and one can show that
 with careful choices in the inductive process $\{F(\X^n)\}$ is a
 quasi-visual subdivision as well (for more details see
 Proposition~\ref{prop:exqs} and its proof). Hence the
 homeomorphism $F$ is a quasisymmetry, and $T$ and $\cT$ are
 quasisymmetrically equivalent.

%  The ultimate reason why  it is  is possible to realize 
%  any given combinatorics of a decomposition $\X^n$ of $T$ also by a decomposition of  $\cT$ 
%  hinges   on a universality property   of the 
% CSST. Lemma~\ref{lem:decomp}  gives a precise though technical formulation of this that we will use in the proof 
% of Theorem~\ref{thm:CSST_qs}.  We will comment more on the underlying ideas as we discuss the   details.

\subsection*{Organization of this paper}
\label{sec:organ-this-paper}

We first  introduce
\emph{quasi-visual approximations} and \emph{quasi-visual subdivisions} in Section~\ref{sec:appr-subd}. These are
sequences of decompositions of a space that approximate it in a
geometrically controlled way. They are closely connected to
quasisymmetries (see Proposition~\ref{prop:qv-qs} and  Proposition~\ref{prop:qv_f_qs}). 
In Section~\ref{sec:topology-trees} we collect some facts about the  topology of
trees and their decompositions. 
 In Section~\ref{sec:unif-separ-dens} we show that quasisymmetries
preserve the property of a trivalent tree to be uniformly branching.

The CSST and some of its properties are reviewed in Section~\ref{sec:cont-self-simil}.
There we show that the CSST is uniformly branching (Proposition~\ref{prop:CSST_unif}). 
In Section~\ref{sec:subdivisions-trees} we consider specific
subdivisions of trees and introduce
\emph{edge-like subdivisions}. We will show that 
each uniformly branching quasiconformal tree admits edge-like subdivisions with good geometric control (see Proposition~\ref{prop:decomp}).  

 Section~\ref{sec:decomp} is the technical core of the paper. There we show that if a uniformly branching quasiconformal tree $T$ has  
 an edge-like subdivision as in Proposition~\ref{prop:decomp},
 then   there exists a subdivision of the CSST with the same combinatorics (see Proposition~\ref{prop:exqs}).
 For the proof we use an inductive process based on an existence result for homeomorphisms $F\: T\ra \cT$ with good properties (see 
 Lemma~\ref{lem:decomp}). 
 
 Our argument is  wrapped up in Section~\ref{sec:proof}, where we assemble the considerations of the earlier sections 
for the proof of   Theorem~\ref{thm:CSST_qs}.

\subsection*{Notation}
\label{sec:notation}

We denote by $\N=\{1,2,\dots\}$ the set of natural numbers and set
$\N_0=\N\cup \{0\}$. We write $\#X\in \N_0$ for the cardinality of a finite set $X$.

%The \emph{ceiling} of a real number $x$, denoted by $\ceil{x}$,
%is the smallest integer $m\in \Z$ with $x\leq m$.  

Given two non-negative quantities $a$ and $b$, we write $a\asymp
b$ if there is a constant $C>0$ depending on some ambient parameters such that $a/C \leq b \leq C
a$. In this case, we refer to the constant as
$C=C(\asymp)$. Similarly, we write $a\lesssim b$ or $b\gtrsim a$
if there is a constant $C>0$ such that $a\leq Cb$. We refer to the
constant as $C=C(\lesssim)=C(\gtrsim)$.

Let  $(X,d)$ be a metric space.    If $A\subset X$, we denote by $\overline{A}$ the
closure, by $\inte(A)$ the interior, and by $\partial A$ the boundary
of $A$. Moreover, 
$$\diam(A)=\sup\{d(x,y): x,y\in X\}$$
is the diameter of $A$. If $B\sub X$ is another set, then we use the notation 
$$ \dist(A,B)=\inf\{d(x,y): x\in A,\, y\in B\}$$ 
for the distance of $A$ and $B$. 

Let $T$ be a metric tree, $x,y\in T$, and $[x,y]$ be the unique arc in $T$ joining $x$ and $y$.  Then we write $\diam [x,y]$ instead of $\diam([x,y])$
for the diameter of the arc $[x,y]$; so we omit the 
parentheses in our notation  for better readability.  Similarly, we denote by $\length  [x,y]$ the length of $[x,y]$ with respect to the given metric on $T$.

%The disjoint union of two set $A$ and $B$ is denoted by
%$A\sqcup B$. 

If $F\colon X\to Y$ is a map between sets $X$, $Y$, and $A\subset
X$ is a subset of $X$, then we write $F|A$ for the restriction of $F$ to $A$. The identity map on $X$ is denoted by $\id_X$, or simply by $\id$ if $X$ is understood.

\subsection*{Acknowledgment} We thank Angela Wu for a remark that led to 
 Lemma~\ref{lem:qv_qs} in its present form.

\section{Quasi-visual approximations and subdivisions}
\label{sec:appr-subd}

In this section, we provide a general framework to describe a
space by {\em approximations} and \emph{subdivisions}. Our main  result here gives 
 a method of how to prove quasisymmetric equivalence  of two metric spaces based on these concepts (see
Proposition~\ref{prop:qv_f_qs}). Very similar ideas have recently been formulated by Kigami \cite{Ki}. 
We will first recall the definition 
of a  quasisymmetric homeomorphism and then consider  \emph{quasi-visual approximations} of metric spaces.

Let $(S,d)$ and $(T,\rho)$ be metric spaces. A homeomorphism
$F\colon S\to T$ is called a {\em quasisymmetric homeomorphism}
or a {\em quasisymmetry}  if
there  exists a homeomorphism  $\eta\: [0,\infty)\ra [0,\infty)$ such that for all $t>0$ and all 
$x,y,z\in S$ the following implication holds:
\begin{equation}\label{eq:defqs}
  d(x,y)\leq td(x,z) 
  \Rightarrow
  \rho(F(x),F(y))\leq \eta(t) \rho(F(x),F(z)).  
\end{equation}
If we want to emphasize the distortion function $\eta$ here, then we call $F$ an $\eta$-{\em quasisymmetry}. 

%Here it is enough to actually find a function $\eta\: (0,\infty)\ra (0,\infty)$ with $\lim_{t\to 0}\eta(t)=0$ so that implication 
%\eqref{eq:defqs} holds, because for each such function $\eta$ there exists a homeomorphism $\tilde \eta\: [0,\infty)\ra [0,\infty)$ such that $\eta(t)\le \tilde \eta(t)$ for $t>0$. So then we  can replace 
%$\eta$ with $\tilde \eta$ in \eqref{eq:defqs}.

If a bijection $F\colon S\to T$ satisfies the implication
\eqref{eq:defqs}, then $F$ is continuous, and $F^{-1}$ satisfies a similar implication. Hence $F^{-1}$ is also continuous, and $F$ is a homeomorphism. So once we have \eqref{eq:defqs} for a bijection
$F\:S\ra T$,  it is a quasisymmetric homeomorphism.

As we will see, the following concept is relevant in connection with quasisymmetries. 

\subsection*{Quasi-visual approximations}
\label{sec:quasi-visu-appr}

\begin{definition}[Quasi-visual approximations]
  \label{def:qv_approx}
  Let $S$ be a bounded metric space.  A \emph{quasi-visual
    approximation} of $S$ is a sequence
  $\{\X^n\}_{n\in\N_0}$ of finite covers $\X^n$, $n\in \N_0$,  of $S$ by some of its subsets. Here we assume $\X^0=\{S\}$ and make the following requirements 
  for all $n\in \N_0$ with implicit constants independent of $n$, $X$, $Y$: 
    \begin{enumerate}
  \item 
    \label{item:qv_approx1}
       $\diam(X) \asymp \diam(Y)$
  for all  $X,Y\in \X^n$ with $X\cap Y\ne \emptyset$.
  \item 
    \label{item:qv_approx2}
      $\dist(X,Y) \gtrsim \diam(X)$
      for all  $X,Y\in \X^n$ with 
      $X\cap Y=\emptyset$.
     \item 
    \label{item:qv_approx3}
     $\diam(X)\asymp \diam(Y)$
for all $X\in \X^n$, $Y\in \X^{n+1}$ with 
      $X\cap Y\ne \emptyset$.
  \item 
    \label{item:qv_approx4}
    For some constants $k_0\in \N_0$ and $\lambda \in (0,1)$ independent of $n$ we have $ \diam(Y) \leq \lambda \diam(X)$
    for all $X\in \X^n$ and $Y\in \X^{n+k_0}$ 
     with $X\cap Y\neq \emptyset$.   \end{enumerate}
\end{definition}

%Condition \ref{item:qv_approx3} says the  ``child'' $X'$ cannot be much
%smaller than its ``parent'' $X$; condition \ref{item:qv_approx4} says
%that the size of a $k_0$-th generation ``descendant'' is smaller
%than the size  of the ``ancestor'' by the definite factor
%$\lambda<1$.

For simplicity, we will usually just write $\{\X^n\}$ instead of 
$\{\X^n\}_{n\in\N_0}$  with the index set $\N_0$ for $n$ understood.
We call the elements of $\X^n$ the {\em tiles of level $n$}, or simply the {\em $n$-tiles} of  the sequence  $\{\X^n\}$.

Condition~\ref{item:qv_approx4} in the previous definition can be expressed in a slightly
different form.

\begin{lemma}
  \label{lem:sub_shrink}
  Let $\{\X^n\}$ be a sequence of finite covers  of a bounded metric space  $S$ satisfying condition~\ref{item:qv_approx3} in Definition~\ref{def:qv_approx}. Then
  $\{\X^n\}$ satisfies \ref{item:qv_approx4} if and only 
  if there
  exist  constants $C>0$ and $\rho \in (0,1)$ such that
  \begin{gather}
 % \begin{equation}
    \label{eq:qv_approx4p}
    \tag{iv'} 
  \diam(X^{n+k}) \leq C \rho^k \diam(X^n).
%\end{equation}
  \end{gather}
  for all $k,n\in \N_0$,  $X^n\in \X^n$, $X^{n+k}\in\X^{n+k}$ with
  $X^n\cap X^{n+k}\neq \emptyset$. 
  \end{lemma}

\begin{proof}
  Assume first that condition \ref{item:qv_approx4} in Definition~\ref{def:qv_approx} holds with  
   $k_0\in \N$ and $\lambda\in (0,1)$. Let  $k,n\in \N_0$,  
    $X^n\in \X^n$, and  $X^{n+k}\in\X^{n+k}$ with
  $X^n\cap X^{n+k}\neq \emptyset$ be arbitrary. 
  We pick a point $x\in X^n\cap X^{n+k}$. For  $i=n+1, \dots, n+k-1$, we choose 
  $X^i\in \X^i$  with $x\in X^i$. 
  We write $k$  in the form $k=jk_0 + \ell$ with  $ j\in \N_0$ and
  $\ell\in \{0,\dots,k_0-1\}$. Finally,  let $C_1=C(\asymp)\ge 1$ be  the constant in condition \ref{item:qv_approx3} of
  Definition~\ref{def:qv_approx}. Then by 
 repeated applications of  \ref{item:qv_approx3} and  \ref{item:qv_approx4} 
   we  see 
  that
  \begin{align*}
    \diam(X^{n+k})
    &=
    \diam(X^{n+jk_0 +\ell}) 
    \leq 
    C_1^{\ell}\diam(X^{n+jk_0})
    \\
    &\leq
    C_1^{\ell}\lambda^j\diam(X^n)     
    \le  
    C_1^{k_0-1}(\lambda^{1/k_0})^{jk_0}\diam(X^n). 
   \end{align*}  
If we define $\rho \coloneqq\lambda^{1/k_0}\in (0,1)$ 
      and $C\coloneqq C_1^{k_0-1}{\lambda}^{-(k_0-1)/k_0}\ge C_1^{k_0-1}\rho^{-\ell} $,
      we can write this as 
       \begin{align*}
    \diam(X^{n+k})
  &\le  C_1^{k_0-1} \rho^{\,jk_0}\diam(X^n)
    \le C\rho^{\,jk_0 + \ell} \diam(X^n)
    \\
    &= C\rho^k \diam(X^n),
   \end{align*}  
  as desired.

  To shown the reverse implication, assume that
  \eqref{eq:qv_approx4p} holds for constants $C>0$ and 
  $\rho\in (0,1)$. We can then choose $k_0\in \N$ sufficiently large
  such that $\lambda\coloneqq C\rho^{k_0}<1$. With
  these choices,  condition~\ref{item:qv_approx4} 
  in Definition~\ref{def:qv_approx} is clearly satisfied.
\end{proof}

A very similar argument shows that  condition~\ref{item:qv_approx3} in Definition~\ref{def:qv_approx} implies an inequality  opposite to   \eqref{eq:qv_approx4p}. Namely, there exists a constant $\tau\in (0,1) $
such that
  \begin{equation}
    \label{eq:qv_approx8}
    \diam(X^{n+k})\ge   \tau^k \diam(X^n)
    \end{equation}
  for all $k,n\in \N_0$,  $X^n\in \X^n$, $X^{n+k}\in\X^{n+k}$ with
  $X^n\cap X^{n+k}\neq \emptyset$.

The previous lemma implies that in a quasi-visual approximation
the diameter of tiles tend to $0$ uniformly with their level.

\begin{cor}
  \label{cor:qv_tiles_shrink}
  Let $\{\X^n\}$ be a quasi-visual approximation of a bounded
  metric space $S$. Then
  \begin{equation*}
    \max\{\diam(X) : X\in \X^n\} \to 0
    \text{ as } n\to \infty. 
  \end{equation*}
\end{cor}

\begin{proof}
  Indeed, by Lemma~\ref{lem:sub_shrink} we know that there exist
  constants $C>0$ and $\rho\in (0,1)$ with
  \begin{equation*}
    \diam(X^n) \leq C \rho^n \diam(X^0)
  \end{equation*}
  for all  $n\in \N_0$ and $X^n\in \X^n$. Here $X^0=S$ is the
  only $0$-tile. The statement follows. 
\end{proof}

We record a special situation when a
  sequence of coverings forms a quasi-visual approximation.
 
 \begin{lemma}\label{lem:visual} 
 Let $S$ be a bounded metric space, and let  
  $\{\X^n\}_{n\in\N_0}$  be a sequence of finite covers $\X^n$,
  $n\in\N_0$, of $S$ by some of its subsets. 
  Suppose  that $\X^0=\{S\}$, and that there exists a constant $\delta\in (0,1)$ such that the following conditions are true: 
  \begin{enumerate}
  \item 
    \label{item:visual1}
       $\diam(X) \asymp \delta^n$
  for all $n\in \N_0$ and  $X \in \X^n$.
  
  \item 
    \label{item:visual2}
      $\dist(X,Y) \gtrsim \delta^n$
      for all  $n\in \N_0$ and $X,Y\in \X^n$ with 
      $X\cap Y=\emptyset$.
      \end{enumerate}  
      Here we require  that the implicit constants are independent of $n$, $X$, and $Y$. 
      Then  $\{\X^n\}_{n\in\N_0}$ is a quasi-visual approximation of $S$. 
   \end{lemma}
   It is natural to call  $\{\X^n\}$ as in this lemma  a {\em visual approximation} of $S$, because its tiles have properties that are analogous to  properties of tiles of an expanding Thurston map with respect to a {\em visual metric} (see 
   \cite[Proposition 8.4]{BM}).  We chose the term 
   {\em quasi-visual approximation} in 
   Definition~\ref{def:qv_approx}, because  there the conditions   are more relaxed in comparison to the more stringent ones  for  a visual approximation. 

\begin{proof}[Proof of Lemma~\ref{lem:visual}]
We have to verify conditions 
\ref{item:qv_approx1}--\ref{item:qv_approx4} in Definition~\ref{def:qv_approx}.  In the following, all implicit constants are independent of the tiles under consideration and their levels.

Fix $n\in \N_0$,  
 and  consider  arbitrary tiles $X,Y\in \X^n$. Then
  \begin{equation*}
    \diam(X)
    \asymp
    \delta^n
    \asymp
    \diam(Y),
  \end{equation*}
  by \ref{item:visual1}. Thus condition \ref{item:qv_approx1} in 
  Definition~\ref{def:qv_approx} is satisfied.

  If $X\cap Y=\emptyset$, then by   \ref{item:visual2} we have 
  $$ \dist (X,Y)\gtrsim \delta^n \asymp \diam (X),$$ 
 and so condition \ref{item:qv_approx2} in Definition~\ref{def:qv_approx}
  is satisfied.

  Let $Z\in \X^{n+1}$ be
  arbitrary. Then
  \begin{equation*}
    \diam(Z)
    \asymp
    \delta^{n+1}
    \asymp
    \delta \diam(X),
  \end{equation*}
  by \ref{item:visual1}. Condition ~\ref{item:qv_approx3}
  in 
  Definition~\ref{def:qv_approx} immediately follows (with a constant depending on $\delta$ in addition to the other ambient parameters).

  Similarly, if $Z\in \X^{n+k}$ for some $k\in \N_0$, then
  \begin{equation*}
    \diam(Z)
    \asymp
    \delta^{n+k}
    \asymp
    \delta^k \diam(X)
  \end{equation*}
  by  \ref{item:visual1}.
 This implies that condition~\eqref{eq:qv_approx4p} in Lemma~\ref{lem:sub_shrink} is true, which is equivalent to \ref{item:qv_approx4} in 
  Definition~\ref{def:qv_approx}. 

The statement follows.
\end{proof}

The covers $\X^n$ in a quasi-visual approximation provide a
discrete approximation of $S$ that improves with larger $n$. We
introduce a quantity that records
 at which level this approximation allows us to distinguish
two given points for the first time.

\begin{definition}
  \label{def:mxy} 
  Let $\{\X^n\}_{n\in\N_0}$ be a quasi-visual approximation of 
  the bounded metric space $S$. For two distinct points $x,y\in S$ we define
  \begin{align*} 
    m(x,y)
    \coloneqq 
    \max\{n\in \N_0: {}&\text{there exist sets $X,Y\in \X^n$}
    \\ 
                       &\text{with $x\in X$, $y \in
                         Y$, and $X\cap Y\ne \emptyset$}\}. 
  \end{align*} 
 \end{definition}

Note that by Corollary~\ref{cor:qv_tiles_shrink}
we have 
$$ \lim_{n\to\infty} \max_{X\in \X^n} \diam(X)=0,$$ which implies that  $m(x,y)\in \N_0$ is well defined for $x,y\in S$ with $x\ne y$.

\begin{lemma}
  \label{lem:qv_metric}
  Let $\{\X^n\}_{n\in\N_0}$ be a quasi-visual approximation of the
 bounded  metric space $(S,d)$. Then there is a constant $C(\asymp)$ such
  that for all distinct $x,y\in S$ we have
  \begin{equation*}
    d(x,y) \asymp \diam(X^m),
  \end{equation*}
  where $m=m(x,y)$ and $X^m\in \X^m$ is an arbitrary $m$-tile
  that contains $x$.
\end{lemma}
 
\begin{proof}
  Let $x,y\in S$ with $x\ne y$ be  arbitrary,  and
  $m=m(x,y)\in \N_0$ be given as in Definition~\ref{def:mxy}. This
  means that there are $m$-tiles $X,Y\in \X^m$ with $x\in X$, $y\in Y$,
  and $X\cap Y\neq \emptyset$. Let $X^m\in \X^m$ be an arbitrary 
  $m$-tile that contains $x$. Then 
  \begin{equation*}
    d(x,y) \leq \diam(X) + \diam(Y) \asymp \diam(X)\asymp
    \diam (X^m),
  \end{equation*}
  where we used condition~\ref{item:qv_approx1} in 
  Definition~\ref{def:qv_approx}.

  On the other hand, we can find $X',Y'\in \X^{m+1}$ with $x\in X'$ and
  $y\in Y'$. The definition of $m(x,y)$ now implies that $X'$ and
  $Y'$ must be disjoint. Thus
  \begin{equation*}
    d(x,y)\geq \dist(X',Y') \gtrsim \diam(X') \asymp \diam(X^m).
  \end{equation*}
 Here we used   \ref{item:qv_approx2} and \ref{item:qv_approx3} in 
  Definition~\ref{def:qv_approx}. The statement follows.
\end{proof}

%Note that in the previous proof property \ref{item:qv_approx4}
%  was not used. Thus the statement remains true if we just assume
%  that the sequence $\{\X^n\}$ satisfies
%  properties~\ref{item:qv_approx1},\ref{item:qv_approx2}, and
%  \ref{item:qv_approx3} of Definition~\ref{def:qv_approx}.
%  
%  WE DO NEED \ref{item:qv_approx4} to conclude that the $\diam(X^n)\to 0$ as $n\to \infty$. Otherwise, $m(x,y)$ may not be 
%
%

\subsection*{Quasi-visual approximations and quasisymmetries}
\label{sec:quasi-visu-quas}

The following proposition  relates the concept of a quasi-visual approximation with quasisymmetries.   This is the main reason we consider such
approximations. 

\begin{proposition}
  \label{prop:qv-qs}
  Let $(S,d)$ and $(T, \rho)$ be bounded metric spaces, the map $F\: S\ra T$ be a
  bijection, and $\{\X^n\}_{n\in \N_0}$ be a quasi-visual approximation of
  $(S,d)$. Then $F\: S\ra T$ is a quasisymmetry if and only if
  $\{F(\X^n)\}_{n\in \N_0}$ is a quasi-visual approximation of
  $(T,\rho)$. \end{proposition} Here we use the obvious notation
$F(\X^n)\coloneqq \{F(X): X\in \X^n\}$.

%
%We will break up the proof of this proposition in smaller
%statements. In addition to being (hopefully) easier to read, this
%makes it clearer what properties are used in each implication. 

The proof of Proposition~\ref{prop:qv-qs} requires some
preparation. We first remind the reader of the following
well-known fact. 

\begin{lemma}
  \label{lem:qs_diam}
  Let $(S,d)$ and $(T,\rho)$ be bounded metric spaces, and the map
  $F\colon S\to T$ be an $\eta$-quasisymmetry. Let
  $A\subset S$ be a subset of $S$, and $x,y\in A$ be points satisfying
  $d(x,y) \asymp \diam(A)$ with  a constant
  $C_1=C(\asymp)$. Then
  \begin{equation*}
    \diam(F(A)) \asymp \rho (F(x),F(y)),
  \end{equation*}
  where $C(\asymp)=C(C_1,\eta)$. 
\end{lemma}

\begin{proof} Under the given assumptions, let $z\in A$ be arbitrary. Then
  \begin{equation*}
    d(x,z) \leq \diam(A) \leq C_1d(x,y).
  \end{equation*}
  We write $x'=F(x)$,  $y'=F(y)$, $z'=F(z)$, and $A'=F(A)$. Since
  $F$ is an  $\eta$-quasisymmetry, we have
  $    \rho(x',z') \leq \eta(C_1)\rho (x',y')$.
  It follows that 
  \begin{equation*}
    \diam(A') \leq 2\sup_{z'\in A'}\rho(x',z') \leq
    2\eta(C_1)\rho(x',y'). 
  \end{equation*}
  On the other hand, $\rho(x',y') \leq \diam(A')$. The statement follows  with $C_2=C(\asymp)=\max\{1, 2\eta(C_1)\}$.
\end{proof}

 Quasi-visual approximations are 
invariant under ``quasisymmetric cha\-nges'' of the metric.

\begin{lemma}
  \label{lem:qv_qs_inv}
  Let $d_1$ and $d_2$ be bounded metrics on a set $S$. Suppose
  that a sequence of finite covers $\{\X^n\}$ of $S$ is a
  quasi-visual approximation of the metric space $(S,d_1)$ and
  that the identity map $\id\colon(S,d_1) \to (S,d_2)$ is a
  quasisymmetric homeomorphism. Then $\{\X^n\}$ is also a
  quasi-visual approximation of $(S,d_2)$.
  \end{lemma} 

\begin{proof}
 Under the given assumptions,  we need to show that
  $\{\X^n\}$ is a quasi-visual approximation of $(S,d_2)$, meaning that it satisfies
  the conditions \ref{item:qv_approx1}--\ref{item:qv_approx4} in
  Definition~\ref{def:qv_approx}.
For $i=1,2$ we denote the diameter of a set $M\sub S$ with respect to the metric $d_i$   by
  $\diam_i(M)$. 
  
  \smallskip
  \ref{item:qv_approx1}
  Let $n\in \N_0$ and $X,Y\in \X^n$ with $X\cap Y\ne \emptyset $ be arbitrary. Then there are points $z\in X\cap Y$, 
  $x\in X$,  and $y\in Y$   with $d_1(x,z)\asymp \diam_1(X)$
  and $d_1(y,z) \asymp \diam_1(Y)$, where $C(\asymp)=3$. Since
  condition  \ref{item:qv_approx1} of
  Definition~\ref{def:qv_approx} is satisfied for $d_1$, we know
  that
  \begin{equation*}
    d_1(x,z) \asymp \diam_1(X) \asymp \diam_1(Y)\asymp d_1(y,z).
  \end{equation*}
  The quasisymmetric equivalence of $d_1$ and $d_2$ together with
  Lemma~\ref{lem:qs_diam} now implies  that
  \begin{equation*}
    \diam_2(X)\asymp d_2(x,z) \asymp d_2(y,z) \asymp \diam_2(Y).
  \end{equation*}
  So $\{\X^n\}$ satisfies \ref{item:qv_approx1} in 
  Definition~\ref{def:qv_approx}
  for $(S,d_2)$. 

  \smallskip
  \ref{item:qv_approx3}
  The verification of this condition  is very similar to the one  for
  \ref{item:qv_approx1} and so we will skip the details.  

  \smallskip 
  \ref{item:qv_approx2}
  Let $n\in \N_0$ and $X,Y\in \X^n$  with $X\cap Y= \emptyset$, as well as 
  $x\in X$ and $y\in Y$ be arbitrary. We denote  the distance  of $X$ and $Y$   with respect to the metric $d_i$   by
  $\dist_i(X,Y)$ for $i=1,2$.  
 We can choose  $x'\in X$ with
  $d_1(x,x') \asymp \diam_1(X)$, where $C(\asymp)=3$. Since
  \ref{item:qv_approx2} holds for $d_1$, we know that
  \begin{equation*}
    d_1(x,x') 
    \leq 
    \diam_1(X) 
    \lesssim
    \dist_1(X,Y) 
    \leq d_1(x,y).
  \end{equation*}
  Quasisymmetric equivalence of $d_1$ and $d_2$ together with Lemma~\ref{lem:qs_diam} now implies
  \begin{equation*}
    \diam_2(X) \asymp d_2(x,x') \lesssim d_2(x,y).
  \end{equation*}
  Taking the infimum over $x\in X$ and $y\in Y$ yields
  $$\diam_2(X) \lesssim \dist_2(X,Y),$$ which is condition 
  \ref{item:qv_approx2} for $d_2$.  

  \smallskip
  \ref{item:qv_approx4}
  Let $\lambda\in (0,1)$ and $k_0\in \N$ be  constants as in 
  condition~\ref{item:qv_approx4} of
  Definition~\ref{def:qv_approx}
    for the metric space $(S,d_1)$.  Let $n,\ell\in \N_0$, $X^n\in \X^n$ and $X^{n+\ell k_0}\in \X^{n+\ell k_0}$ with $X^n \cap X^{n+\ell k_0}\ne \emptyset$ be arbitrary.
    Then we can choose  $x\in  X^n \cap X^{n+\ell k_0} $ and 
    $x'\in X^n$ with $\diam (X^n)\le 3 d_1(x,x')$. 
  Finally,  let $y\in X^{n+\ell k_0}$ be arbitrary.
    
 By repeated application of   \ref{item:qv_approx4} in 
  Definition~\ref{def:qv_approx}  for $(S,d_1)$, we see that 
  \begin{align*}
    d_1(x,y) &\leq \diam_1(X^{n+\ell k_0})\\
    &\le \la^\ell \diam_1(X^{n})\le 3 \la^\ell d_1(x,x').
    \end{align*}
If we assume that  $\id\colon (S,d_1) \to (S,d_2)$ is an $\eta$-quasisymmetry, then it follows that 
  \begin{align*}
    \diam_2 (X^{n+\ell k_0})&\le 2 \sup\{d_2(x,y): y\in  
    X^{n+\ell k_0}\}\\
    & \le 2 \eta (3\la^\ell) d_2(x,x')\le  2 \eta (3\la^\ell)\diam_2(X^n).
    \end{align*}
  Since $\eta\colon [0,\infty) \to [0,\infty)$ is a
  homeomorphism, there exist $\ell_0\in \N$ such that
  $\widetilde{\lambda}\coloneqq 2 \eta (3\la^{\ell_0}) <1$. It follows 
  that with this number $\widetilde{\lambda}$  and  $\widetilde{k}_0 \coloneqq \ell_0 k_0$  
 condition \ref{item:qv_approx4} in 
 Definition~\ref{def:qv_approx} holds for
  $(S,d_2)$. 
\end{proof}

The following result is a converse of the previous lemma. 

\begin{lemma}
  \label{lem:qv_qs}
  Let $d_1$ and $d_2$ be bounded metrics on a set $S$. Suppose that a
  sequence of finite covers $\{\X^n\}$ of $S$ is a quasi-visual
  approximation of  both metric spaces $(S,d_1)$ and
  $(S,d_2)$. Then the identity map $\id\colon(S,d_1) \to (S,d_2)$
  is a quasisymmetric homeomorphism.
\end{lemma}

\begin{proof}  We have to find  a  homeomorphism  
 $\eta\: [0,\infty) \ra [0,\infty)$ such that
  for all  $t>0$ and $x,y,z\in S$ we have the implication
  \begin{equation} \label{eq:qsimpl}
     d_1(x,y) \leq t d_1(x,z)
    \Rightarrow
    d_2(x,y) \leq \eta(t) d_2(x,z).
  \end{equation}

 So let $t>0$ and $x,y,z\in S$ with $d_1(x,y) \leq t d_1(x,z)$
  be arbitrary. We may assume that $x\neq y$. Then also $x\ne z$.
  
  In the following, we denote by $X^k\in \X^k$ $k$-tiles that
  contain the point $x$ for various $k\in\N_0$.  We also use the
  notation $\diam_1$ and $\diam_2$ for the diameters of sets for
  the metrics $d_1$ and $d_2$, respectively.

By condition \eqref{eq:qv_approx4p} in
  Lemma~\ref{lem:sub_shrink} and by \eqref{eq:qv_approx8} 
  there are constants $\rho,\tau \in (0,1)$ and $C>0$ independent
  of $t,x,y,z$ such that
  \begin{equation}
    \label{eq:quantver}
    \diam_i(X^{k+\ell}) 
    \le
    C \rho^\ell \diam_i(X^{k})
 \end{equation}
 and 
   \begin{equation}
    \label{eq:quantvertau}
    \diam_i(X^{k+\ell}) 
    \ge
     \tau^\ell \diam_i(X^{k})
 \end{equation} for all $k,\ell\in \N_0$ and $i\in \{1,2\}$.

Now let $m=m(x,y)\in \N_0$ and
  $n=m(x,z)\in \N_0$ be as in Definition~\ref{def:mxy}.
    By  Lemma~\ref{lem:qv_metric} there exists a constant 
  $C(\asymp)>0$ independent of $x,y,z$ such that
  \begin{equation}\label{eq:Kimpl}
d_i(x,y) \asymp   \diam_i(X^m)  \text { and } 
d_i(x,z) \asymp   \diam_i(X^n) 
  \end{equation}
  for $i=1,2$.  This means that by  enlarging the original constant $C$ in \eqref{eq:quantver} if necessary (and thus avoiding introducing new constants) we may assume that 
   \begin{equation}\label{eq:Kimpl1}
\frac 1Cd_i(x,y) \le \diam_i(X^m)\le C  d_i(x,y) 
\end{equation}
and
\begin{equation}\label{eq:Kimpl2}
  \frac 1Cd_i(x,z)\le   \diam_i(X^n)\le C  d_i(x,z)
\end{equation}
for $i=1,2$. Together with our assumption $d_1(x,y) \leq t d_1(x,z)$,  this implies in
particular that
\begin{equation}
  \label{eq:diamXmXn}
  \diam_1(X^m) \leq t C^2  \diam_1(X^n).
\end{equation}

   The idea of the proof is now to translate the previous inequalities
  for the metric $d_1$ to  a relation between  $m$, $n$, and  $t$. Using  this for the metric $d_2$, we will obtain  a bound for  $\diam_2(X^m)$ in terms of  $\diam_2(X^n)$. This will lead to the desired estimate by \eqref{eq:Kimpl1} and   \eqref{eq:Kimpl2}. We consider two cases.

  % Here $X^m$ and $X^n$ may be chosen as any $m$-tile,
  % respectively $n$-tile, that contains $x$. Thus, we may assume
  % by \eqref{eq:subdivXnXnk} that one is the ancestor of the
  % other, meaning that $X^m\subset X^n$ or $X^n\subset X^m$.

  \smallskip
  \emph{Case 1:} $m\leq n$. Then by
  \eqref{eq:diamXmXn} and \eqref{eq:quantver} we
  have \begin{align*} \diam_1(X^m) &\leq tC^2 \diam_1(X^n) = tC^2
    \diam_1(X^{m+(n-m)})
         \\
                                   &\leq tC^3\rho^{n-m}
                                     \diam_1(X^m),
  \end{align*}
 and so  $t C^3\rho^{n-m}\geq 1$. This implies
  \begin{equation}\label{eq:s1}
    m-n
    \ge 
    r(t)
    \coloneqq
   -\frac{\log(tC^3)}{\log(1/\rho)}.
  \end{equation}
  Then it follows that
    \begin{align*}
   d_2(x,y)& \le C \diam_2(X^m) &&    \text{by \eqref{eq:Kimpl1}}\\
    &\leq C \tau^{-(n-m)}\diam_2(X^{m+(n-m)}) &&\text{by \eqref{eq:quantvertau}}\\
    &=C \tau^{m-n}\diam_2(X^{n}) && \\
    &\le C \tau^{r(t)}\diam_2(X^n) &&\text{by \eqref{eq:s1}}\\
       &\le C^2\tau^{r(t)} d_2(x,z) &&\text{by \eqref{eq:Kimpl2}}\\
    &=\eta_1(t) d_2(x,z), &&
     \end{align*}
where we set $\eta_1(t)= C^2\tau^{r(t)}$. If we define $\eta_1(0)=0$, then $\eta_1\:[0,\infty)\ra [0,\infty)$ is a homeomorphism as follows from the expression 
for  $r(t)$ in \eqref{eq:s1}.

   \smallskip
  \emph{Case 2:} $m> n$. 
 Then  by  \eqref{eq:diamXmXn} and  \eqref{eq:quantvertau} we have 
   \begin{align*}
   tC^2 \diam_1(X^n)  &\geq 
    \diam_1(X^m) = \diam_1(X^{n+(m-n)}) \\
    &\geq 
   \tau^{m-n} \diam_1(X^n), 
  \end{align*}
and so $tC^2\ge \tau^{m-n}$. This implies
\begin{equation}\label{eq:s0}
   m-n
    \ge 
    s(t)
    \coloneqq
  - \frac{\log(tC^2)}{\log(1/\tau)}.
\end{equation}

Then  it follows  that 
 \begin{align*}
   d_2(x,y)& \le C \diam_2(X^m) =C  \diam_2(X^{n+(m-n)})
   &&\text{by \eqref{eq:Kimpl1}}
    \\
    &\leq  C^2\rho^{m-n}\diam_2(X^{n}) &&\text{by  \eqref{eq:quantver}}
    \\
    &\le  C^2 \rho^{s(t)} \diam_2(X^n)
    && \text{by \eqref{eq:s0}}\\
    &\le C^3 \rho^{s(t)} d_2(x,z)
    && \text{by  \eqref{eq:Kimpl2}}\\
    &=\eta_2(t) d_2(x,z), &&
     \end{align*}
where we set $\eta_2(t)=C^3 \rho^{s(t)}$. If we define 
$\eta_2(0)=0$, then again $\eta_2\:[0,\infty) \ra [0,\infty)$ is a homeomorphism.

If we now set $\eta(t)=\max\{\eta_1(t), \eta_2(t)\}$ for $t\ge 0$, then the previous considerations show that  
 $ \eta\:[0,\infty)\ra [0, \infty)$
is a homeomorphism for which implication \eqref{eq:qsimpl} is true.
The statement follows. 
\end{proof}

Parts of the  previous  argument follow ideas from the proof of 
\cite[Lemma~18.10]{BM} which go back to
\cite[Theorem~4.2]{Me02}. In an earlier version of this paper, we only claimed {\em weak quasisymmetry} of the identity map  (corresponding to Case 1 in the previous proof).  Angela Wu observed  that one can actually show that it is a quasisymmetry. 

If one carefully traces through the proof, then one can see 
that the distortion function $\eta$ can be chosen to have 
the form $$\eta(t)=K\max\{t^\alpha, t^{1/\alpha}\}$$ with a constant $K\ge 1$ and 
$\alpha=\log(\rho)/\log(\tau) \in (0,1]$ (note that $0<\tau\le \rho$ as follows from 
  \eqref{eq:quantver} and 
  \eqref{eq:quantvertau}).
 Quasisymmetries  with such a distortion function are called
\emph{power quasisymmetries}. It is no surprise that our map is of this type; essentially, this is  implied by
\cite[Theorem~11.3]{He}.

Combining Lemma~\ref{lem:qv_qs_inv} and Lemma~\ref{lem:qv_qs}, we
can now establish the main result of this subsection.

\begin{proof}[Proof of Proposition~\ref{prop:qv-qs}] We define a  distance function  $\tilde d$ on $S$ by setting 
$$ \tilde d(x,y)=\rho(F(x), F(y))$$ 
for $x,y\in S$. Since $F\: S\ra T$ is a bijection and $\rho$ is a bounded metric on $T$, it is clear that $\tilde d$ is a bounded metric on $S$. Moreover, the map $F\: (S, \tilde d)\ra (T,\rho)$ is an isometry.

Now suppose   $F$ is a quasisymmetry from $(S,d)$ onto $(T,\rho)$.
If we postcompose this with the isometry $F^{-1}\:   (T,\rho)
\ra  (S,\tilde d)$, then we obtain the identity map
$\id\:  (S,d)\ra  (S,\tilde d)$ and see that it is a quasisymmetry. 
 Lemma~\ref{lem:qv_qs_inv}  implies that $\{\X^n\}$ is a quasi-visual approximation of 
 $(S, \tilde d)$. If we map $\{\X^n\}$
 to $T$ by the isometry $F\: (S, \tilde d)\ra (T,\rho)$,
 then it follows that $\{F(\X^n)\}$ is a quasi-visual approximation
 of  $(T, \rho)$.

 Conversely, suppose $\{F(\X^n)\}$ is a quasi-visual
 approximation of $(T, \rho)$. Applying the isometry
 $F^{-1}\: (T,\rho) \ra (S,\tilde d)$, we see that
 $\{\X^n\}=\{F^{-1}(F(\X^n))\}$ is a quasi-visual approximation
 of both metric spaces $(S, \tilde d)$ and $(S,d)$.  By
 Lemma~\ref{lem:qv_qs} the identity map
 $\id \: (S,d) \ra (S, \tilde d)$ is a quasisymmetry. If we
 compose this with the isometry $F\: (S, \tilde d)\ra (T,\rho)$,
 then we obtain the map $F\: (S, d)\ra (T,\rho)$ and it follows
 that this map is a quasisymmetry.
\end{proof}

\subsection*{Subdivisions}
\label{sec:subdivisions}
We now consider quasi-visual approximations of different metric spaces $S$ and $T$ that are ``combinatorially isomorphic'' in a suitable sense and hope to construct a quasisymmetry $F\: S\ra T$ that establishes the  correspondence between these approximations.  In order to obtain a positive result in this direction, we  have  to impose somewhat stronger assumptions on the  approximations. This is the motivation for the following concept.

\begin{definition}[Subdivisions]
  \label{def:subdiv}
  A \emph{subdivision} of a compact metric space $S$ is a sequence
  $\{\X^n\}_{n\in \N_0}$ with the following properties:
  \begin{enumerate}
  \item 
    \label{item:subdiv1}
  $\X^0=\{S\}$,  and    $\X^n$ is a finite collection of   
    compact subsets of $S$   for each $n\in \N$.
   \item 
    \label{item:subdiv2}
    For each  $n\in \N_0$ and $Y\in \X^{n+1}$, there exists $X\in
    \X^{n}$ with $Y\subset X$. 
  \item 
    \label{item:subdiv3}
    For each  $n\in \N_0$ and $X\in \X^{n}$, we have
    \begin{equation*}
      X= \bigcup\{Y\in \X^{n+1} : Y\subset X\}.
    \end{equation*}
  \end{enumerate}
\end{definition}

Let $\{\X^n\}$ be a subdivision of $S$. Since $\X^0=\{S\}$, it follows from induction based on 
 \ref{item:subdiv3} that
\begin{equation*}
  S= \bigcup_{X\in \X^n}X,
\end{equation*}
for each $n\in \N$.  In particular, each $\X^n$ is a finite cover of $S$ by some of its compact subsets. An element $X$ of  any of the collections $\X^n$, $n\in \N_0$, is called a \emph{tile}
of the subdivision,  and a {\em tile of level $n$} or an $n$-{\em tile} if $X\in \X^n$.

If  $k,n\in \N_0$ and $X\in \X^n$, then \ref{item:subdiv3} also implies that 
\begin{equation}
  \label{eq:subdivXnXnk}
  X = \bigcup\{Y\in \X^{n+k} : Y\subset X\}.
\end{equation}
So each $n$-tile $X$ is ``subdivided''
by tiles  of a higher level $n+k$. 

The following concept is of key importance  for this paper. 
\begin{definition}[Quasi-visual subdivisions]
  \label{def:qvsub}
A subdivision $\{\X^n\}$ of a compact metric space $S$ is called a
\emph{quasi-visual subdivision} of $S$ if it is a quasi-visual approximation
according to Definition~\ref{def:qv_approx}.   
\end{definition} 

%We denote by $\X\coloneqq \bigcup_n \X^n$ the disjoint union of the collections $\{\X^n\}_{n\in \N_0}$. This means we distinguish tiles $X$ and $Y$ if their levels are different even if the underlying sets are the same.

%Let $X$ be an $n$-tile. An $(n+1)$-tile $X'\subset X$ is called a
%\emph{child of $X$}, and $X$ is called the \emph{parent of
%  $X'$}. Given a $k\in \N_0$, we call an $(n+k)$-tile $X^{n+k}
%\subset X$ a \emph{descendant} of $X$, and $X$ an \emph{ancestor}
%of $X^{n+k}$. 

A sequence $\{X^n\}_{n\in\N_0}$ of tiles in a subdivision $\{\X^n\}$ is called {\em descending} if $X^n\in \X^n$ 
for  $n\in \N_0$ and 
\begin{equation}
  \label{eq:desc_ntiles}
  X^0\supset X^1 \supset X^2\supset \dots \,.
\end{equation}
It easily follows from \eqref{eq:subdivXnXnk} that for each $x\in S$ there is a descending sequence
$\{X^n\}_{n\in \N_0}$ of tiles such that  $x\in \bigcap_n
X^n$. Note that this  sequence is not unique in general.

Let  $\{\X^n\}$ and $\{\Y^n\}$ be  subdivisions of compact 
  metric spaces $S$ and $T$, respectively. We say 
  $\{\X^n\}$ and $\{\Y^n\}$ are
\emph{isomorphic subdivisions} if there exist bijections 
 $F^n\colon\X^n\to \Y^n$,  $n\in \N_0$,  such that for all $n\in \N_0$, $X,Y\in
\X^n$, and $X'\in \X^{n+1}$ we have
\begin{align}
  \label{eq:isosubdiv_cap}
  X\cap Y \neq \emptyset
  \quad\text{if and only if}\quad
  F^n(X) \cap F^n(Y) \neq \emptyset
  \intertext{and}
  \label{eq:isosubdiv_incl}
    X'\subset X 
  \quad\text{if and only if}\quad
  F^{n+1}(X') \subset F^n(X).
\end{align}
We say that the
isomorphism between $\{\X^n\}$ and $\{\Y^n\}$ is
 \emph{given
  by} the family $\{F^n\}$. We say that an isomorphism $\{F^n\}$ is {\em induced} by a 
 homeomorphism $F\: S\ra T$ if $F^n(X)=F(X)$ for all $n\in \N_0$ and $X\in \X^n$. 

\begin{proposition}
  \label{prop:qv_f_qs}
  Let $S$ and $T$ be compact metric spaces with 
  quasi-visual subdivisions $\{\X^n\}$ and $\{\Y^n\}$, respectively. 
If  $\{\X^n\}$ and $\{\Y^n\}$ are isomorphic, then there exists  a 
unique quasisymmetric homeomorphism 
  $F\colon S\ra  T$ that induces the isomorphism
  between $\{\X^n\}$ and $\{\Y^n\}$.
\end{proposition}

\begin{proof} 
Since $\{\X^n\}$ and $\{\Y^n\}$ are quasi-visual subdivisions of $S$ and $T$, respectively, by  Corollary~\ref{cor:qv_tiles_shrink} we know that 
\begin{equation}\label{eq:shr}
\lim_{n\to \infty} \max_{X\in \X^n} \diam(X)=0 \text{ and } \lim_{n\to \infty} \max_{Y\in \Y^n} \diam(Y)=0.
\end{equation} 
Now if the isomorphism between $\{\X^n\}$ and $\{\Y^n\}$ is given by the sequence $\{F^n\}$ of bijections $F^n\: \X^n\ra \Y^n$, then
there exists a unique homeomorphism 
$F\:S\ra T$ such that $F(X)=F^n(X)$ for all $n\in \N_0$ and $X\in \X^n$. 
This follows from 
  \cite[Proposition~2.1]{BT}. The idea for the proof is straightforward: for each $x\in S$ one finds a descending sequence $\{X^n\}$ with $X^n\in \X^n$ for $n\in \N_0$ and $\{x\}= \bigcap_n X^n$. Then $\{F^n(X^n)\}$ is a descending sequence
  so that $\bigcap_n F^n(X^n)$ contains a single point $y\in T$. Here \eqref{eq:shr} is important, because it guarantees that the intersection of a descending sequence of 
tiles is a singleton set. 
  One now sets $F(x)=y$ and shows that $F$ is a well-defined homeomorphism that  induces the isomorphism given by $\{F^n\}$.  
   
We then have $\Y^n=F^n(\X^n)=F(\X^n)$ for $n\in \N_0$, and so 
$\{\Y^n\}=\{F(\X^n)\}$. Our hypotheses and       
 Proposition ~\ref{prop:qv-qs} now imply that $F$ is a quasisymmetry. 
\end{proof}

\section{Topological facts about trees}
\label{sec:topology-trees}

In this section we collect some general facts about metric trees that will be useful later on.   There is a rich literature on the underlying topological spaces, usually called dendrites (a metric space is a tree if and only if it  is a non-degenerate dendrite; see \cite[Proposition 2.2] {BM19}). We refer to \cite[Chapter V]{Wh}, \cite[Section \S 51 VI]{Ku68},  \cite[Chapter~X]{Na}, and the references in these sources for more on the subject. 

Let $T$ be a (metric) tree (as defined in the introduction), and  $x,y\in T$. Then $[x,y]$ denotes the unique arc in $T$ with
endpoints $x$ and $y$. If $x=y$, then this arc is degenerate and
$[x,y]=\{x\}$.  We also consider the half-open and open arcs
joining $x$ and $y$ in $T$ defined as
\begin{equation*}
%  \label{eq:defarcs}
  (x,y] \coloneqq [x,y]\setminus\{x\}, 
  \quad 
  [x,y) \coloneqq [x,y]\setminus\{y\},
  \quad
  (x,y) \coloneqq [x,y]\setminus\{x,y\}.  
\end{equation*}

The given metric on the tree $T$ (or $T$ itself)  is called {\em geodesic} if 
\begin{equation}\label {eq:geodmetr} 
 |x-y|= \length[x,y]
 \end{equation}
 for all $x,y\in T$. This is a strong assumption that is in general {\em not} true for the trees we consider. 
\subsection*{Subtrees}
\label{sec:subtrees}

A subset $X$ of a  tree $T$ is called a {\em subtree} of $T$ if $X$ equipped with the restriction of the metric on $T$ is also a tree.
 One can show that  $X\sub T$  is a subtree of $T$ if and
 only if $X$ contains at least two points and is  closed and
 connected (see \cite[Lemma~3.3]{BT}).  If $X$ is a subtree of $T$, then $[x,y]\sub X$ for all $x,y\in X$.

 The following statement is \cite[Lemma 2.3]{BM19}. 

\begin{lemma}
  \label{lem:top_T}
  Let $T$ be a tree and  $\V\subset T$ be a finite set. Then
  the following statements are true:

  \begin{enumerate}
    \item 
    \label{item:top_T1}
    Two points $x,y\in T\setminus \V$ lie in the same
    component of $T\setminus \V$ if and only if $[x,y]\cap
    \V=\emptyset$.  
    \item 
    \label{item:top_T2} If $U$ is a component of $T\setminus \V$, then $U$ is an open set and 
    $\overline{U}$ is a subtree of $T$ with  $\partial \overline U\sub \partial U\sub \V$. 
  \item
    \label{item:top_T3}
    If $U$ and $W$  are  two distinct components of $T\setminus
    \V$, then $\overline{U}$ and $\overline{W}$ 
    have at most one point in common. Such a common point belongs to $\V$, and is a boundary point of both  $\overline{U}$ and $\overline{W}$. 
  \end{enumerate}
\end{lemma}

Let $T$ be a tree and $p\in T$. If $U$ is a component of $T\setminus\{p\}$, then
$$B\coloneqq\overline U=U\cup \{p\}$$
is called a {\em branch} of $p$ (in $T$).
This is a subtree of
$T$ (see \cite[Lemma~3.2~(ii) and Lemma~3.4]{BT}). 
 In particular, only the point $p$ is added as we pass from $U$ to $\overline U$. 

%  Let $T$ be a tree and $p\in T$. If  $U$ is   a component of $T\setminus\{p\}$, then
% $B\coloneqq\overline U=U\cup \{p\}$, and  
%  $B$ is a subtree of
% $T$, called a {\em branch} of $p$ (in $T$) (see \cite[Lemma 3.2 (ii) and Lemma 3.4]{BT}). 
%  In particular, only the point $p$ is added as we pass from $U$ to $\overline U$. 

The set 
$T\setminus \{p\}$ has at least one and at most
countably many distinct complementary components
 (see
\cite[Lemma~3.8 and the discussion after its proof]{BT}). We denote the number of these components by $\deg_T(p)\in \N\cup\{\infty\}$, where we set 
$\deg_T(p)=\infty$ if $T\setminus \{p\}$ has infinitely many components. Note  that $\deg_T(p)$ is equal to the number of branches of $p$ in $T$. 

If $\deg_T(p)=1$, then we call $p$ a \emph{leaf} of $T$; so $p$
is a leaf of $T$ precisely when $T\setminus\{p\}$ has one
component or, equivalently, if $T\setminus\{p\}$ is connected.
  
If $\deg_T(p)\ge 3$, or equivalently, if $T\setminus\{p\}$ has at
least three components, we call $p$ a \emph{branch point} of
$T$. If $\deg_T(p) =3$, or equivalently, if $T\setminus\{p\}$ has
exactly three components, $p$ is called a \emph{triple point} of
$T$.
Finally, we call $T$ \emph{trivalent}, if each branch point
of $T$ is a triple point, or equivalently, if $\deg_T(p)\leq 3$
for all $p\in T$. 

% We call $T$ \emph{trivalent}, if each branch point $p$ of $T$ is
% a {\em triple point}, i.e., the set $T\setminus \{p\}$ has
% exactly three components.

\begin{lemma}  \label{lem:subt}
Let $T$ be a tree, $S\sub T$ be a subtree of $T$, and $p\in
  S$. Then the following statements are true:
  \begin{enumerate} 
   \item 
    \label{item:subt0}  Every branch $B$ of $p$ in $S$ is contained in a unique branch $B'$ of $p$ in $T$. The  
 assignment $B\mapsto B'$ is an injective map between the  sets of branches of $p$  in $S$ and in $T$. If $p$ is an interior point of $S$, then this map is a bijection. 
 
 In particular, 
    $\deg_S(p)\le \deg_T(p)$ with equality if $p\in \inte(S)$.

  \item 
    \label{item:subt1} If $p$ is a leaf of $T$, then $p$ is a leaf of $S$. Conversely, 
if  $p$ is a leaf of $S$ with $p\not \in  \partial S$, then 
$p$ is a leaf of $T$. 

\item \label{item:subt2}    
 If $B\sub S$ is a branch of $p$ in $S$ with $B\cap \partial S=\emptyset$, then $B$ is a branch of $p$ in $T$.       \end{enumerate} 
\end{lemma} 

\begin{proof}   \ref{item:subt0} This is \cite[Lemma 3.5]{BT}. 

 \smallskip
  \ref{item:subt1} Suppose $p$ is a leaf of $T$. Then by  \ref{item:subt0} we have $1\le \deg_S(p)\le \deg_T(p)=1$. Hence $\deg_S(p)=1$ and $p$ is a leaf of  $S$.

Conversely, suppose $p$ is a leaf of $S$ with $p\not \in \partial S$.
Then $p\in S \setminus \partial S=\inte(S)$, and so
$\deg_T(p)=\deg_S(p)=1$ by  \ref{item:subt0}. Hence $p$ is a leaf of $T$. 

% \smallskip
%  \ref{item:subt1}  Suppose $p$ is a leaf of $T$. Let $x,y\in S\setminus\{p\}$ be arbitrary. Then $[x,y]\sub S$, because $S$ is a subtree of $T$. Since $p$ is a leaf of $T$, the set $T\setminus\{p\}$ is connected and has only one component. This implies  $p\not\in [x,y]$ by Lemma~\ref{lem:top_T}~\ref{item:top_T1}, and so 
%  $[x,y]\sub S\setminus\{p\}$. It follows that  $S\setminus\{p\}$ is connected, and so $p$ is a leaf of $S$.
%
%  For the converse, we  argue by contradiction and assume that $p$ is a leaf of $S$ with $p\not\in \partial S$, but that $p$ is not a leaf of $T$. Then there exist at least
%  two components of $T\setminus \{p\}$, and so two distinct
%  points $x,y\in T\setminus \{p\}$ such that $p\in (x,y)$ as follows from Lemma~\ref{lem:top_T}~\ref{item:top_T1}. 
% Since
%  $p$ lies in $S$ and not in $\partial S$, it is an interior
%  point of $S$. This implies that there exist points
%  $x'\in (x,p)$ and $y'\in (p, y)$ close to $p$ that belong to
%  $S$. Since $S$ is a subtree of $T$, we then have
%  $p\in (x',y')\sub (x,y) \cap S$. In particular, $x'$ and $y'$
%  lie in different components of $S\setminus \{p\}$  by  
%  if $p$ is a leaf of $S$ with $p\not \in  \partial S$. This is
%  impossible, because $p$ is a leaf of $S$ and so
%  $S\setminus \{p\}$ is connected.

  \smallskip 
  \ref{item:subt2}
  The set $U=B\setminus\{p\}$ is a connected subset of
  $S\setminus\{p\}\sub T\setminus\{p\}$. This set is relatively
  open in $S$, and hence an open subset of $T$, because
  $U\cap \partial S=\emptyset$.  
  
  The set
  $U=\overline U\cap S\setminus\{p\}$ is also relatively closed
  in $S\setminus\{p\}$.  Since $S$ is closed, every limit point
  of $U$ in $T\setminus\{p\}$ belongs to $S\setminus \{p\}$, and
  hence to $U$. This shows that $U$ is relatively closed
  in $T\setminus\{p\}$. Since $U$ is also open and connected, $U$
  must be a component of $T\setminus\{p\}$, i.e., a maximal
  connected subset of $T\setminus\{p\}$. Indeed, if
  $U\sub M\sub T\setminus\{p\}$ with $U\ne M$, then $M$ has $U$
  as a non-trivial open and relatively closed subset, and so $M$
  cannot be connected. It follows that $B=\overline U=U\cup\{p\}$ is a branch of 
  $p$ is $T$.
 \end{proof}

 \subsection*{Decompositions and subdivisions of trees}
 \label{sec:decomp-subd-trees}
 We now assume that $T$ is a tree such that each branch point $p$
 of $T$ has only finitely many branches (i.e.,  $\deg_T(p)< \infty$ 
 or, equivalently, $T\setminus \{p\}$
 has only  finitely many components). 
 
 Let $\V$ be
 a finite (possibly empty) set of points in $T$ that does not
 contain any leaf of $T$. We want to decompose $T$ into pieces by
 ``cutting'' $T$ at the points in $\V$.
 % For this we consider the
 % complementary components of $T\setminus \V$.
 A {\em tile} $X$
 (in the decomposition induced by $\V$) is the closure of a
 component of $T\setminus \V$. It follows from
 Lemma~\ref{lem:top_T}~\ref{item:top_T2} that each tile is a
 subtree of $T$. The set of all tiles obtained in this way from
 $\V$ is denoted by $\X$ and called the \emph{decomposition
   induced by $\V$}. We will see momentarily that the set $\X$
 is always a  finite cover of $T$ (see
 Lemma~\ref{lem:vt}~\ref{item:vt8}).  
   
Most of the following statements  are intuitively clear,  but we will include full proofs for the sake of completeness.

% 
%  If $\V\ne \emptyset$, then
% $\# \partial X\ge 1$ for each tile $X$. In general, one cannot
% say more about the cardinality of $\partial X$. We will consider
% decompositions $\X$ with the property that each $X\in \X$ has at
% most two boundary points. Then $\X$ is called an {\em edge-like
%   decomposition} of $T$. We say that a tile $X$ in a edge-like
% decomposition of $T$ is a {\em leaf-tile} if $\#\partial X=1$
% and an {\em edge-tile} if $\#\partial X=2$.

\begin{lemma}
  \label{lem:vt} Let $T$ be a tree such that every branch point of $T$ has only finitely many branches, and let $\V\sub
  T$ be a finite (possibly empty)  subset of $T$ that contains no leaf of $T$.
  Let $\X$ denote the decomposition of $T$ induced by $\V$.  
  Then the  following statements are true: 
  \begin{enumerate}
  \item
    \label{item:vt1}
    If $X\in \X$ and $v\in \V$, then $X$ is contained in the
    closure of one of the components of $T\setminus \{v\}$ and
    disjoint from the other components of $T\setminus \{v\}$.
    
  \item 
    \label{item:vt2}
    Each tile $X\in \X$ is a subtree of $T$ with
    $ \partial X= \V\cap X$ and $\inte(X)=X\setminus \V$. The
    closure of $\inte(X)$ is equal to $X$.
      
  \item
    \label{item:vt3}
    If $\V\ne \emptyset$ and  $X\in \X$, then $\partial X\neq \emptyset$ and each point
    $v\in \partial X$ is a leaf of the subtree $X$.

  \item 
    \label{item:vt5}
    Two distinct tiles $X,Y\in \X$ have at most one point in
    common. Such a common point of $X$ and $Y$ belongs to $\V$
    and is a boundary point of both $X$ and $Y$.
    
  \item
    \label{item:vt6}
    If $v\in \V$ and $n=\deg_T(v)\in \N$, $n\ge 2$, is the number of
    branches of $v$ in $T$, then $v$ is contained in precisely
    $n$ distinct tiles in $\X$.  Each  branch of $v$ in $T$ contains precisely one of these tiles.  

  \item
    \label{item:vt7}
    If $X\in \X$ and $\partial X=\{v\}\sub \V$ is a singleton
    set, then $X$ is a branch of $v$ in $T$.
   
  \item
    \label{item:vt8}
    The set of tiles $\X$ is a finite cover of $T$.
    
     \item
    \label{item:vt4}
    Suppose, in addition, that $T$ is a trivalent tree with a
    dense set of branch points. Then each $X\in \X$ is also a
    trivalent tree with a dense set of branch points.

\end{enumerate}
\end{lemma}

Note that if  $\V=\emptyset$, then  $\X=\{T\}$; so $X=T$ is the only 
tile in $\X$ and $\partial X=\partial T=\emptyset$. In this case, the statements in 
the previous lemma are vacuously or trivially true.

\begin{proof}
\ref{item:vt1}  We have $X=\overline U$, where
  $U$ is a component of $T\setminus \V$. Moreover,  
  since $v\in \V$ is not a leaf of 
$T$, our hypotheses imply that the set $T\setminus \{v\}$
has $n$ distinct components $W_1, \dots, W_n$, where $n\in \N$, $n\ge 2$. Since $U$ is a connected subset of $T\setminus \V\sub T \setminus \{v\}$, it is contained in one of these components, say $U\sub W_1$. 
Then $X=\overline U\sub \mybar{W}_1=W_1\cup \{v\}$, and $X$ is disjoint from $W_2\cup \dots \cup W_n =T\setminus  
 \mybar{W}_1$.

\smallskip
\ref{item:vt2} If $X\in \X$ is a tile, then $X=\overline U$, where 
$U$ is a component of $T\setminus \V$. It then follows from 
Lemma~\ref{lem:top_T}~\ref{item:top_T2} that  $X=\overline U$ is a subtree of $T$ with $\partial X=\partial \overline U\sub \partial U\sub \V$. Hence $\partial X\sub \V\cap X$, because $X$ is a closed subset of $T$. 

Conversely, if $v\in \V\cap X$, then by what we have seen in \ref{item:vt1}, $v$ is in the closure $\overline W=W \cup\{v\}$ of at least one component $W$ of $T\setminus\{v\}$ disjoint from 
$X$. This implies that $v\in \partial X$, and so $\partial X=\V\cap X$. The  last identity also shows  that  $\inte(X)=X\setminus \partial X=X\setminus \V$.  

Note that $U$ is disjoint from $\V$, and $\partial U\sub\V$ by 
Lemma~\ref{lem:top_T}~\ref{item:top_T2}. We also have $X=\overline U=U\cup \partial U$. This implies that 
 $$\inte(X)=X\setminus \V=(U \cup \partial U)\setminus \V=U. $$
 Since $\overline U=X$ and $\inte(X)=U$, the closure of $\inte(X)$ is equal to  $X$.  

 %It remains to show that $\partial X\ne \emptyset$. 

%  Pick $v\in \V^n\ne \emptyset$. Then $U$ is a connected subset of $\qT\setminus \{v\}\supset \qT\setminus \V^n$. So $U$ lies in a unique component $W_1$ of  $\qT\setminus \{v\}$. Since $v$, as every point in $\V^n$, is a double point of $\qT$, there exists precisely one other component $W_2\ne W_1$ of 
%  $\qT\setminus \{v\}$. If $x\in U\sub  W_1$ and $y\in W_2$ are arbitrary, then $v\in [x,y]\cap \V^n$ by Lemma~\ref{lem:top_T}~\ref{item:top_T1}. The same lemma then implies that $x$ and $y$ lie in different components of $\qT\setminus \V^n$. It follows that $x\ne y$ , and so $U\cap W_2=\emptyset$.
% Since $W_2$ is open Lemma~\ref{lem:top_T}~\ref{item:top_T2}, the sets $X=\overline U$ and $W_2\ne \emptyset $ are disjoint. In  particular, $X\ne \qT$. Since $\qT$ is connected, this implies that $\partial X\ne \emptyset$; otherwise,  the non-empty set $X\ne \qT$ would be a closed and open subset of the connected space $\qT$. This is impossible. 

\smallskip \ref{item:vt3}
If    $\V\neq \emptyset$, then  it  follows from \ref{item:vt1}
that $X\ne T$. Since $T$ is connected, this implies that $\partial X\ne \emptyset$; otherwise,  the non-empty set $X\ne T$ would be an open and closed   subset of the connected space $T$. This is impossible. 

We have $X=\overline U$, where $U$ is a component of 
$T\setminus \V$. If $v\in \partial X\sub \V$, then 
$X\setminus \{v\}$ is connected, because $U\sub X\setminus \{v\}
\sub \overline U=X$ (every set squeezed between a  connected set and its closure is connected). Hence $v$ is a leaf of $X$.

\smallskip \ref{item:vt5} This immediately follows from 
Lemma~\ref{lem:top_T}~\ref{item:top_T3}.

\smallskip \ref{item:vt6} For   $n=\deg_T(v)$ let   $W_1, \dots, W_n$ denote the  distinct 
components of $T\setminus \{v\}$. Here $n\in \N$, $n\ge 2$. 
 We have 
$\mybar{W}_k=W_k\cup\{v\}$  and so 
$v\in \mybar{W}_k$  for 
$k=1, \dots, n$.

For  each  $k\in \{1, \dots, n\}$ we choose a point $x_k\in W_k\sub T\setminus \{v\}$. Then 
by  Lemma~\ref{lem:top_T}~\ref{item:top_T1}
we  have
\begin{equation}\label{eq:vonarc}
 v\in [x_k,x_\ell] \text{ for all $k,\ell \in \{1, \dots, n\}$ with $k\ne \ell$.} 
\end{equation}
Since $v\in \mybar {W}_k\cap \mybar{W}_\ell$, we may  assume that  $x_k$ and $x_\ell$ are so close to $v$ that $[x_k,x_\ell]$ contains no other point  in the finite set $\V$. Then 
$[x_k,v)$ is a non-empty connected subset of $T\setminus \V$ for each $k\in \{1, \dots, n\}$, and so it must be contained in a unique component $U_k$ of $T\setminus \V$. These components 
$U_1, \dots, U_n$ of $T\setminus \V$ are pairwise disjoint open sets as follows from \eqref{eq:vonarc}
and  Lemma~\ref{lem:top_T}~\ref{item:top_T1} and ~\ref{item:top_T2}.  
Then   $X_k\coloneqq \mybar{U}_k$ for $k=1, \dots, n$ is a tile in $\X$  that contains the arc $[x_k,v]$, and so $v\in X_k$. Moreover,
the tiles $X_1, \dots, X_n$ are all distinct, because each tile $X_k$ contains 
the non-empty open set $U_k$ that is disjoint from all the other tiles in this list. So $v$ is contained in at least $n$ distinct tiles in $\X$.

 Suppose $Z=\overline U\in \X$ is any tile with $v\in Z$, where $U$ is a component of $T\setminus \V$. We choose a point $z\in U$.
 Then $z$   
 must be contained in one of the components $W_1, \dots, W_n$  of   
 $T\setminus \{v\}$, say $z\in W_1$. 
 Then $[x_1,z]\sub T\setminus \{v\}$ by Lemma~\ref{lem:top_T}~\ref{item:top_T1}. We may assume that $x_1$ and $z$ are so close to $v$ that  $[x_1,z]$ contains no point in $\V\setminus\{v\}$. Then $[x_1,z]\cap \V=\emptyset$, and so $x_1$ and $z$ are contained in the same component of  
 $T\setminus \V$. It follows that $U=U_1$, and so 
 $Z=\overline {U}=\mybar{U}_1=X_1$. This shows that 
 $X_1, \dots, X_n$  are the only tiles in $\X$ that contain $v$. So $v$ is contained in  $n$ distinct tiles $X_1, \dots, X_n\in \X$, and in no other tiles in $\X$. 
 
Let  $k\in \{1, \dots, n\}$. Then   the set  $U_k$ is a connected subset of $T\setminus \V\sub T\setminus \{v\}$, and so it is contained in a unique component of $T\setminus \{v\}$. This component must be  $W_k$, because $x_k\in U_k\cap W_k$. Hence 
 $X_k=\mybar{U}_k\sub \mybar{W}_k$, which shows that $X_k$ is contained in the branch $\mybar{W}_k$ of $v$ in $T$.  The tile $X_k$ cannot be contained in any other branch $\mybar{W}_\ell=W_\ell\cup\{v\}$ with $\ell\ne k$, because $X_k$ is disjoint from $W_\ell$ as follows from  \ref{item:vt1}.

 \smallskip \ref{item:vt7}
Suppose that $\partial X= \{v\}\sub \V$. We have $X=\overline U$,
where $U$ is a component of $T\setminus \V$. As we have seen in the proof of \ref{item:vt1}, there  is a component $W$ of
$\qT\setminus\{v\}$ with $U\subset W$.

\smallskip
{\em Claim:}  $U=W$. 
 To see this, we argue by contradiction and assume that $U\ne W$.
 Then there exists a point $x\in U\sub W$, as well as a point 
 $y\in W\setminus U$. Then $[x,y]\cap 
 \V\ne\emptyset$, because otherwise $y\in U$. 
 So as we travel from $x$ to $y$ along $[x,y]$, there must be a first point $u\in [x,y]$ that belongs to $\V$. Then $[x,u)\sub U$, and so $[x,u]\sub \overline U=X$. 
 
 By  \ref{item:vt6} the point $u\in \V$ belongs to  at least one tile in $\X$ distinct from $X$, and so  $u\in X$ is a boundary point of $X$ as follows from \ref{item:vt5}.
 Hence $u\in \partial X=\{v\}$  and so $u=v$. 
 Since $v=u\in [x,y]$,    
  Lemma~\ref{lem:top_T}~\ref{item:top_T1} implies that $x$ and $y$ lie in different components of $T\setminus \{v\}$. Since $x$ and $y$ lie in the same component $W$ of $T\setminus \{v\}$, this is a contradiction. The Claim follows.
  
  \smallskip 
 Since  $U=W$ by the Claim, we conclude  that 
  $X=\overline U=\overline W$ is the closure of the component $W$ of $T\setminus \{v\}$, and so a branch of $v$ in $T$.

 \smallskip \ref{item:vt8} If $\V=\emptyset$, this is obvious, because then we have $\X=\{T\}$.  
 
 Now suppose  $\V\ne \emptyset$. Then  each tile in $\X$  contains some point in $\V$ as follows from \ref{item:vt2} and  \ref{item:vt3}, and each point 
 in $\V$  is contained in at most $N$ tiles  by \ref{item:vt6}, where 
for  $N\in \N$ we choose some upper bound for the number of branches for the finitely many points in $\V$. This implies that the number of tiles in $\X$ is bounded above by 
 $ \#\V \cdot N$, and so $\X$ is a finite set. 
 
 Each point in $T\setminus \V$ lies in a tile in $\X$ by
 definition of tiles; this is also true for the points in $\V$ by
 \ref{item:vt6}. Hence $\X$ is a finite cover of $T$.
 
 \smallskip 
  \ref{item:vt4} Under the given additional assumptions on $T$, 
  let $B$ denote the set of branch points of $T$ and $B_X\sub X$ denote the set of branch points of its subtree $X$.  
  If $p\in\partial X=\V\cap X$, then $p$ is a leaf  of $X$ by  \ref{item:vt3}. 
  
  If $p\in 
  \inte(X)=X\setminus \partial X$, then the number of branches of $p$ in $X$ is the same as the number of branches of $p$ in $T$  
  (see Lemma~\ref{lem:subt}~\ref{item:subt0}).
 This shows that   $p\in X$ is a branch point of $X$ if and only if
 $p\in \inte(X)$ and $p$  is a branch point of  $T$; in this case, $p$ is a triple point of $X$. In particular, $X$ is a trivalent tree and 
 $B_X=B\cap \inte(X)$. 
 
 Since $B$ is dense in $T$, the set  
 $B_X=B\cap \inte(X)$ is dense in  $\inte(X)$, and hence dense in 
 $X$, because the closure of $\inte(X)$ is equal to $X$ by 
  \ref{item:vt2}.
 It follows that $X$ is a trivalent tree with a dense set $B_X$ of branch points.     
   \end{proof} 

We will now record some relations of two tile  decompositions of a tree $T$ induced by a set $\V$ and a larger set $\V'$. 

\begin{lemma}
  \label{lem:vtt}
  Let $T$ be a tree such that every branch point of $T$ has only
  finitely many branches.  Let $ \V\sub \V' \sub T$
  be finite (possibly empty) subsets of $T$ that contain no leaf of $T$.
  Let $\X$ and $\X'$ denote the decompositions of $T$ induced by
  $\V$ and $\V'$, respectively.  Then the following statements
  are true:
  \begin{enumerate}
  \item 
    \label{item:vtt1}
    Each tile $X'\in \X'$ is contained in a unique tile
    $X\in \X$.
     
  \item
    \label{item:vtt2}
    Each tile $X\in \X$ is equal to the union of all tiles
    $X'\in \X'$ with $X'\sub X$.
  \item
    \label{item:vtt3}
    Let $X\in \X$, $\V_X\coloneqq \V'\cap \inte(X)$, and   
    $\X_X$ be the decomposition of $X$ induced by $\V_X$. Then 
    $\X_X=\{ X'\in \X': X'\sub X\}$.
    
  \item
    \label{item:vtt4}
    Let $X'\in \X'$, and $X\in \X$ be the unique tile with
    $X'\sub X$.  If $\partial_X X'$ denotes the relative boundary
    of $X'$ in $X$, then $\partial_X X'=\partial X'\setminus \V$.
\end{enumerate}
\end{lemma}

 \begin{proof}
 \smallskip \ref{item:vtt1} If $X'\in \X'$ , then there exists 
a component $W$ of $T\setminus \V'$ with $X'=\overline W$. Since $\V\sub \V'$, the set $W$ is a connected subset of $T\setminus \V\supset T\setminus \V'$ and so contained in a unique component $U$ of $T\setminus \V$. Then 
$X'$ is contained in the tile $X\coloneqq \overline U\in \X$, because 
$ X'=\overline{W} \sub \overline U=X.$  
There can be no other tile in $\X$ containing $X'$, because by 
Lemma~\ref{lem:vt}~\ref{item:vt2} the set $X'$ is a subtree of $T$ and hence an infinite set, but  distinct tiles in $\X$ can have at most one point in common by Lemma~\ref{lem:vt}~\ref{item:vt5}.

\smallskip \ref{item:vtt2} If $X\in \X$, then $X=\overline U$, where $U$ is a component of $T\setminus \V$. Since $U$ is connected, this set cannot contain isolated points. This implies that 
the set $U\setminus \V'\sub X$ is dense in $U$ and hence also dense in $X=\overline U$. 

If $x\in U\setminus \V'$ is arbitrary, then there exists a component $W$ of 
$T\setminus \V'$ with $x\in W$. Since  $W$ is a connected subset of $T\setminus \V'\sub T\setminus \V$, 
this set must be contained in a component of $T\setminus \V$. Since $x\in U\cap W$, it follows that $W\sub U$. 
Then $X'\coloneqq \overline W$ is a tile in $\X'$ with $x\in X'$ and $X'= \overline W\sub \overline U=X$. 

This shows that if we denote by $Y$ the union of all tiles $X'\in \X'$ with  $X'\sub X$, then $Y\sub X$  contains the set $U\setminus \V'$. By  Lemma~\ref{lem:vt}~\ref{item:vt8} applied to $\V'$ there are only finitely many tiles in $\X'$, and so $Y$ is closed. Since 
$U\setminus \V'$ is dense in $X$ and $U\setminus \V'
\sub Y$, it follows that $X=Y$ as desired.

\smallskip \ref{item:vtt3}
We first verify the hypotheses of Lemma~\ref{lem:vt} for the tree
$X$ and its finite subset $\V_X$. First note that $X$ is a subtree of $T$, and
hence a tree. If $p\in X$ is arbitrary, then the number of
branches of $p$ in $X$ is bounded by the number of branches of
$p$ in $T$ (see Lemma~\ref{lem:subt}~\ref{item:subt0}), and hence
finite.

If $v\in \V_X\sub 
\inte(X)=X\setminus \partial X$, then $v$ is not a leaf of $X$;  
otherwise, $v\in  \V_X$  is a leaf of $T$ by
 Lemma~\ref{lem:subt}~\ref{item:subt1}, but $ \V_X \sub \V'$ contains no leaves of $T$ by our hypotheses. This shows that
  $\V_X\sub X$ contains no leaf of $X$;  so  we can apply all the 
 statements from Lemma~\ref{lem:vt}  for the decomposition $\X_X$ of the tree 
$X$  induced by $\V_X$.

We first show that  
$\{X'\in \X':X'\sub X\}\sub \X_X$. To this end, let $X'\in \X'$ with 
$X'\sub X$ be arbitrary. Then there exists a component $U$ of $T\setminus \V'$ with  $X'=\overline U$. Since 
$U$ is a connected subset of $X\setminus \V'\sub X\setminus \V_X$, there exists a component $W$ of $X\setminus \V_X$ 
with $U\sub W\sub X$. 

\smallskip 
{\em Claim:} $ W\sub \overline U$. In order to see this, we argue by contradiction and assume that there exists a point $x\in 
W\setminus \overline U$. Since $W$ is a relatively open subset 
of the tree $X$ (as follows from 
Lemma~\ref{lem:top_T}~\ref{item:top_T2}), we may assume (by moving $x$ slightly if necessary) that $x$ does not belong to the finite set $\V'$. 

We choose $y\in U$, and consider the arc $[x,y]\sub X$.
Since $x\in W\setminus (U\cup \V')  $ and $y\in U$ lie in different 
components of $T\setminus \V'$, there exists a point $v\in [x,y]\cap
\V'\sub X$ by Lemma~\ref{lem:top_T}~\ref{item:top_T1}. Here $v\ne x,y$. We cannot have $v \in \inte(X)$, because then  
$v\in\V'\cap \inte (X) = \V_X$; 
so $x$ and $y$ would lie in different components of $X\setminus   \V_X$, but these points lie in the same component $W$ of 
$X\setminus   \V_X$. It follows that  $v\in \partial X$.  
 By Lemma~\ref{lem:vt}~\ref{item:vt3} this  implies that 
$v$ is a leaf of  $X$.  This leads to a contradiction, because we have $v\in [x,y]$ and so $x$ and $y$ lie in different components of $X\setminus \{v\}$ by Lemma~\ref{lem:top_T}~\ref{item:top_T1}, but $X\setminus \{v\}$ is connected, because $v$ is a leaf of $X$. The Claim follows.

\smallskip
By the Claim, we have $U\sub W\sub \overline U$, and so 
$X'=\overline U=\overline W$. This shows that 
$X'$ is the closure of the component $W$ of $X\setminus \V_X$,
and so $X'$ belongs to the set of tiles $\X_X$ of the
decomposition of $X$ induced by $\V_X$. We have proved $\{X'\in
\X':X'\sub X\}\sub \X_X$. 

To show the other inclusion $ \X_X\sub  \{X'\in
\X':X'\sub X\}$,  first note that  $\X_X$ is a finite cover of $X$ by 
 Lemma~\ref{lem:vt}~\ref{item:vt8}. Moreover,  each   tile in $\X_X$ has non-empty   interior relative to $X$ (as follows  from the definition of tiles in combination with 
 Lemma~\ref{lem:top_T}~\ref{item:top_T2}) and this interior of a tile 
  is disjoint from all the other tiles in $\X_X$ (as follows from Lemma~\ref{lem:vt}~\ref{item:vt5}).

 Now let $Y\in \X_X$ be arbitrary. Then we  can choose a point $x$ in the interior of $Y$ relative to $X$. By \ref{item:vtt2} there 
exists $X'\in \X'$ with $X'\sub X$ and $x\in X'$. By what we have seen, $X'\in \X_X$. Since $x$ cannot belong to  any tile in 
   $\X_X$ except $Y$, we conclude that $Y=X'$. This 
   shows the other inclusion  $\X_X\sub \{X'\in \X: X'\sub X\}$. It follows  that  these two sets are equal, as desired.

   \smallskip 
     \ref{item:vtt4}
     Let $X'\in \X'$ be arbitrary. Then  by   \ref{item:vtt1} there exists  a unique tile $X\in \X$ with
     $X'\sub X$. To identify the relative boundary $\partial_XX'$,  we can apply the results from 
   Lemma~\ref{lem:vt}   to  $X$ and its decomposition $\X_X$ induced by $\V_X= \V' \cap \inte(X)$, because the relevant hypotheses  of Lemma~\ref{lem:vt}  are true as was pointed out in the proof    of \ref{item:vtt3}. In particular, by repeated application of   Lemma~\ref{lem:vt}~\ref{item:vt2} we see that  
   \begin{align*} \partial_X X'&=X'\cap \V_X =X'\cap \V'\cap \inte(X)\\ 
   & =\partial X'\cap \inte(X)=\partial X'\cap (X\setminus\V)=
   \partial X'\setminus \V.
  \end{align*}  
  In the last equality, we used that $\partial X'\sub X'\sub X$. The statement follows.  
\end{proof}

Let $T$ be a tree such that every point in $T$ has only finitely many branches (as in Lemmas~\ref{lem:vt}~and~\ref{lem:vtt}).  Let $\{\V^n\}_{n\in \N_0}$ be an increasing sequence of finite
 subsets  $T$, meaning  that 
 \begin{equation}
   \label{eq:Vn_incrs}
   \V^0\subset \V^1 \subset \V^2 \subset \dots\,.
 \end{equation}
In addition,  we require that $\V^0=\emptyset$, and that no   set $\V^n$ for $n\in \N$ contains a leaf of $T$. We denote  by $\X^n$
 the  decomposition of $T$
 induced by $\V^n$. Then it immediately follows from 
 Lemma~\ref{lem:vt}~\ref{item:vt8} and 
 Lemma~\ref{lem:vtt}~\ref{item:vtt1} and \ref{item:vtt2} that the sequence
 $\{\X^n\}_{n\in \N_0}$ is a subdivision of the tree $T$
 according to Definition~\ref{def:subdiv}.  We call $\{\X^n\}$
 the \emph{subdivision of} $T$ \emph{induced by} $\{\V^n\}$.
 Subdivisions of $T$ obtained in this way are one of the main ingredients in the proof of  the ``if'' implication in 
 Theorem~\ref{thm:CSST_qs}. We will discuss this in detail  in Sections~\ref{sec:subdivisions-trees}
 and~\ref{sec:decomp}.

 % As we will see in Sections~\ref{sec:subdivisions-trees} and~\ref{sec:decomp},  decompositions of $T$ obtained in this way are one of the main ingredients in the proof of  the ``if" implication in 
 % Theorem~\ref{thm:CSST_qs}.

\subsection*{Height and center}
\label{sec:height-center}

Let $T$ be a tree and $p\in T$ be a branch point of $T$. Then
there are points $x,y,z\in T$ that lie in distinct branches of
$p$. Conversely, we often want to find $p$ when $x,y,z\in T$ are
given. In addition, when $x,y,z$ are branch points, we want to
relate their heights $H_T(x), H_T(y),H_T(z)$ to the height
$H_T(p)$ of $p$. We remind the reader that height was defined in \eqref{eq:HT}.

\begin{lemma}
  \label{lem:center}
  Let $T$ be a tree, and $x,y,z\in T$. 
  
  \begin{enumerate} 
  \item 
    \label{item:center1} 
    Then
    $[x,y]\cap[y,z] \cap [z,x]$ is a singleton set.

  \item \label{item:center2} If $x,y,z$ are branch points of $T$,
    and $\{c\}=[x,y]\cap[y,z] \cap [z,x]$, then $c$ is a branch
    point of $T$ with
    \begin{equation*}
       H_T(c)\ge \min \{ H_T(x), H_T(y), H_T(z)\}.
    \end{equation*}
\end{enumerate}
\end{lemma}

We call the point $c\in [x,y]\cap[y,z] \cap [z,x]$ the
\emph{center} of $x,y,z$ (in $T$). 

\begin{proof} \ref{item:center1}   As we travel from $z$ to $x$ along the arc $[z,x]$, there exists a first point $c\in  [z,x]$ that lies on $[x,y]$. 
Then $c\in [x,y]\cap [z,x]$ and $[x,y]\cap [z,c)=\emptyset$. 
Moreover,  
\begin{equation}\label{eq:xyrep} 
[x,y]=[x,c)\cup\{c\} \cup  (c,y]
\end{equation}  is a decomposition of
$[x,y]$ into pairwise disjoint sets. Note that $[x,c)$ or $(c,y]$ could be empty, but this does not affect this statement. It follows that   the sets 
\begin{equation}\label{eq:disj} 
[x,c), \, [y,c),\, [z,c),\, \{c\} \text { are pairwise disjoint.}
\end{equation} 
Here we can write the  endpoints of the half-open arcs in a different order, because $[x,c)=(c,x]$, etc. 

It follows from \eqref{eq:disj} that $[x,c)\cup \{c\}\cup (c,z]$  is an arc in $T$ joining $x$ and $z$. Since $[z,x]$ is the unique such arc, we conclude that 
$$ [z,x]=[x,c)\cup \{c\}\cup (c,z], $$
and similarly,
$$ [y,z]=[y,c)\cup \{c\}\cup (c,z]. $$
Together with \eqref{eq:xyrep} and \eqref{eq:disj}, this implies that 
the arcs $[x,y]$, $[y,z]$, $[z,x]$ have the point $c$ in common, but no other point. 

\smallskip
 \ref{item:center2}   
 We may assume that $c\ne x,y,z$, because otherwise the statement is obvious. Then the sets $[x,c)$, $[y,c)$, $[z,c)$ 
are non-empty, connected, and pairwise disjoint by  \eqref{eq:disj}. It follows from  Lemma~\ref{lem:top_T}~\ref{item:top_T1} that each of these sets must lie in a different component of $T\setminus\{c\}$. In particular,
$T\setminus\{c\}$ has at least three such components, and so 
$c$ is a branch point of $T$. 
Let $B_x$, $B_y$, $B_z$ be the distinct branches of $c$ that contain $[x,c)$, $[y,c)$, $[z,c)$, respectively. 

Since $(x,c]$ is a connected subset of $T\setminus \{x\}$,  it must be contained in a branch of $x$ in $T$. Since $x$ is a branch point 
of $T$, there exists another  branch
$\widetilde B_x$ of $x$ in $T$  distinct from the branch of $x$ containing $(x,c]$. We can choose $\widetilde B_x$ so that 
$\diam(\widetilde B_x)\ge H_T(x)$ (by definition of $H_T(x)$, there are actually at least two such choices).
 
Then $\widetilde B_x\cap [x,c]=\{x\}$, and so $\widetilde B_x\cup [x,c)$ is a connected subset of $T\setminus\{c\}$.
Hence 
 $$ B_x\supset \widetilde B_x\cup [x,c)\supset  \widetilde B_x, $$
 which implies that 
 $$ \diam(B_x)\ge \diam  (\widetilde B_x)\ge H_T(x). $$
 Similarly,  $\diam(B_y)\ge H_T(y)$ and  
 $\diam(B_z)\ge H_T(z)$. It follows that 
 \begin{align*}
  H_T(c)&\ge \min \{\diam(B_x), \diam(B_y), \diam(B_z)\}\\
 &\ge \min \{  H_T(x),  H_T(y),  H_T(z)\}, 
 \end{align*} 
 as desired.  
 \end{proof}

%  Assume there are distinct points $c_1,c_2\in [x,y]\cap[y,z]
%  \cap [z,x]$. We assume that $c_1<c_2$ on $[x,y]$. Since
%  $[x,c_2]\cup [c_2,y]$ is the unique arc from $x$ to $y$, it
%  follows that $[c_2,y]$ does not contain $c_1$. It follows that
%  $c_2<c_1$ on $[y,z]$. Repeating this argument, we obtain that
%  $c_1<c_2$ on $[z,x]$, and $[c_2<c_1]$ on $[x,y]$, which is
%  a contradiction. 
%
%  We now show that $[x,y]\cap[y,z]\cap[z,x]\neq
%  \emptyset$. Let $c_x$ be the last common
%  point in $[x,y]$ and $[x,z]$, $c_y$ be the last common point in
%  $[y,z]$ and $[y,x]$, and $c_z$ be the last common point in
%  $[z,x]$ and $[z,y]$. Then $[c_x,c_y] \cap [y,z]= \{c_y\}$ and
%  $[c_x,c_y]\cap [z,x]=\{c_x\}$. Similar statements hold for
%  $[c_y,c_z]$ and $[c_z,c_x]$. It follows that $[c_x,c_y]\cup
%  [c_y,c_z] \cup[c_z,c_x]$ is a closed loop. Since $T$ is a tree,
%  this is only possible if $c=c_x=c_y=c_z$. Thus $c\in
%  [x,y]\cap[y,z]\cap[z,x]\neq\emptyset$. 

\section{Quasisymmetries preserve uniform branching}
\label{sec:unif-separ-dens}

After the preparations in the previous two sections, we now
start the proof of Theorem~\ref{thm:CSST_qs}. 
We will show that uniform relative separation and uniform
relative density of branch points of a tree are preserved under
quasisymmetries. Note that we do
not assume that the trees under consideration are  trivalent, doubling, or of bounded
turning.  

We first need some auxiliary facts.  
\begin{lemma}
  \label{lem:short_branch_qs}
  Let  $F\colon T \to T'$ be a quasisymmetry between trees $T$ and $T'$. Suppose  $B_1$  and $B_2$ are  two branches in $T$ of a point $x\in T$, and 
  \begin{equation*}
     \diam(B_1)\geq \diam(B_2). 
  \end{equation*}
  Let $B'_1\coloneqq F(B_1)$ and $B'_2\coloneqq
  F(B_2)$ be their images in $T'$. Then
  \begin{equation*}
    \diam(B'_1) \gtrsim \diam(B'_2),
  \end{equation*}
  where $C(\gtrsim)$  depends only on the distortion function $\eta$ of $F$.  
\end{lemma}

If  $x$ is a leaf, this is trivially true, because then $B_1=B_2$. 
In general, we will not necessarily have $\diam(B'_1)\geq
\diam(B'_2)$, but 
the lemma says that this is  true ``up to a fixed constant''. 

\begin{proof}
 Under the given assumptions,  we can find a point
  $z_1\in B_1$ with $\abs{z_1-x}\geq \tfrac{1}{2} \diam(B_1)$. It
  follows that for each $z_2\in B_2$, 
  \begin{equation*}
    \abs{z_1-x} 
    \geq 
    \tfrac{1}{2} \diam(B_1) 
    \geq 
    \tfrac{1}{2} \diam(B_2)
    \geq 
    \tfrac{1}{2} \abs{z_2-x}.
  \end{equation*}
  Let $x'\coloneqq F(x)$, $z'_1\coloneqq F(z_1)$,  $z'_2\coloneqq
  F(z_2)\in T'$. Since $F$ is a quasisymmetry, we conclude that   
  \begin{equation*}
    \abs{z'_1 - x'} 
    \gtrsim
    \abs{z'_2 -x'},
  \end{equation*}
  where $C(\gtrsim)$  depends only on the distortion function $\eta$ of $F$, but not on the other
  choices. Now $z_2\in B_2$ was arbitrary, and so we obtain
  \begin{equation*}
    \diam(B'_1) 
    \geq 
    \abs{z'_1-x'} 
    \gtrsim 
    \sup_{z'_2\in B'_2}\abs{z'_2 -x'} \geq \tfrac{1}{2} \diam(B'_2).
  \end{equation*}
  The statement follows.
\end{proof}

We record the following immediate consequence. 

\begin{cor}
  \label{cor:weight_qs}  Let  $F\colon T \to T'$ be a quasisymmetry between trees $T$ and $T'$, 
  let $x\in T$ be a branch point of $T$, and $B$ be a branch
  of $x$ in $T$ with $H_T(x)=\diam(B)$. If  we set  $x'\coloneqq F(x)$ and $B'\coloneqq
  F(B)$,  then 
  \begin{equation*}
    H_{T'}(x') \asymp \diam(B').
  \end{equation*}
  Here $C(\asymp)$ depends only on the distortion function $\eta$ of $F$. 
\end{cor}

\begin{proof}
  The branches $B_1, B_2, B_3, \dots$  of $x$ in $T$ can be labeled such
  that 
  \begin{equation*}
    \diam(B_1) \geq \diam(B_2) \geq \diam(B_3) \geq \dots 
  \end{equation*}
  and $B_3=B$. 
  
  Let $B'_n\coloneqq F(B_n)$ for all $n\in \N$ for which $B_n$ is defined. Since $F\: T\ra T'$ is a quasisymmetry, and in particular a homeomorphism from $T$ onto $T'$, the point $x'=F(x)$ is a branch point of $T'$, and 
  $B'_1, B'_2, B'_3, \dots$ are the branches of $x'$ in $T'$.
  Note that $B'=F(B)=F(B_3)=B_3'$. 
   
By  Lemma~\ref{lem:short_branch_qs} we know that
  \begin{align*}
    \diam(B'_1) &\gtrsim \diam(B'_3),
    \
                  \diam(B'_2) \gtrsim \diam(B'_3), \    \text{and } 
    \\
  \diam(B'_3) &\gtrsim \diam(B'_n)
                     \text{ for all $n\geq 4$}, 
  \end{align*}
  where $C(\gtrsim)$ depends only on the
  distortion function $\eta$ of $F$. 
  
  This shows that  the three 
  branches $B'_1$, $B_2'$, $B_3'$ of $x'$ in $T'$ have diameter 
  $\gtrsim  \diam(B'_3)$, and so $H_{T'}(x')\gtrsim  \diam(B'_3)$. On 
  the other hand, we also see that all other branches have diameter 
  $\lesssim  \diam(B'_3)$, and so $H_{T'}(x')\lesssim   \diam(B'_3)$.
  Hence $H_{T'}(x')\asymp \diam(B'_3)=\diam(B')$, and the statement follows.
\end{proof}

We now  consider uniform relative
separation of branch points. 

\begin{lemma}[Qs-invariance of uniform relative separation]
  \label{lem:unif_sepa_qs}
  Suppose $F\colon T \to T'$ is a quasisymmetry between trees $T$
  and $T'$.  If $T$ has uniformly relatively separated branch
  points, then the same is true for $T'$.
  \end{lemma}

\begin{proof}
  Suppose $T$ has uniformly relatively separated branch
  points. By definition, this means that there is a constant
  $C>0$ such that
  \begin{equation}\label{eq:Thest0}
    \abs{x-y} \geq \frac{1}{C} \min\{H_T(x), H_T(y)\}
  \end{equation}
  for all distinct branch points $x,y\in T$. Since $F$ is a
  quasisymmetry and hence a homeomorphism,  the branch points of $T$ and $T'$
  correspond to each other under the map $F$. So in order to
  establish the statement, it is enough to show an estimate
  analogous to \eqref{eq:Thest0} for the image points of $x$ and
  $y$ under $F$.  Without loss of generality, we may assume that
  $H_T(x) \leq H_T(y)$, i.e., that the minimum in
  \eqref{eq:Thest0} is attained for $x$.

  Let $B$ be a branch of $x$ in $T$ with
  $\diam(B)= H_T(x)$, and   consider an
  arbitrary point  $z\in B$. 
  Then
  \begin{equation*}
    \abs{x-y} 
    \geq 
    \frac{1}{C} H_T(x) 
    =
    \frac{1}{C} \diam(B) 
    \geq
    \frac{1}{C} \abs{x-z}.
  \end{equation*}
  Let
  $x'\coloneqq F(x), y'\coloneqq F(y), z'\coloneqq F(z)\in T'$,
  and $B'\coloneqq F(B)$ be the images under $F$. If  $F$
  is an  $\eta$-quasisymmetry, then the previous inequality implies that
  \begin{equation*}
    \abs{x'-y'} 
    \gtrsim
    \abs{x'-z'},
  \end{equation*}
  where $C(\gtrsim)$  depends only on $C$ and $\eta$.  Now  $z\in B$ was arbitrary, and
  so  
  \begin{align*}
    \abs{x'-y'} &
                  \gtrsim 
                  \sup_{z'\in B'}\abs{x'-z'} 
                  \geq 
                  \tfrac{1}{2} \diam(B')
    \\    
                &\asymp 
                  \tfrac{1}{2} H_{T'}(x'), &&\text{ by
                  Corollary~\ref{cor:weight_qs}}
    \\
                &\geq
                  \tfrac{1}{2} \min\{H_{T'}(x'), H_{T'}(y')\}. 
  \end{align*}
  This implies that $T'$ has uniformly relatively separated
  branch points with a constant that depends only on the uniform
  relative separation constant of $T$ and the distortion function
  $\eta$ of the quasisymmetric homeomorphism $F$.
\end{proof}

Let us now consider uniform relative density of branch points. 
This condition is equivalent to the following, seemingly
stronger, property.

\begin{lemma}
  \label{lem:unif_dense_diam}
A  tree  $T$ has  uniformly relatively  dense
 branch points  if and only if for all $x,y\in T$, $x\ne y$,  there is a branch point $z\in [x,y]$ with 
  \begin{equation*}
    H_T(z) \gtrsim \diam[x,y],
  \end{equation*}
  where $C(\gtrsim)$ is independent of $x$ and $y$. 
\end{lemma}

\begin{proof}
  Since we always have $\diam[x,y] \geq \abs{x-y}$, it is obvious that the above
  condition implies uniform relative density of branch points. 

  Conversely, assume that $T$ has  uniformly relatively dense branch points. 
   Let $x,y\in T$ with $x\ne y$
  be arbitrary. Then there is a point $w\in [x,y]$ with $\abs{x-w}
  =\frac{1}{2} \diam[x,y]$. Uniform relative density of branch points now implies that
  there is a branch point $z\in [x,w]\sub  [x,y]$ with 
  $$H_T(z) \gtrsim
  \abs{x-w} = \tfrac{1}{2} \diam[x,y]. $$ The statement follows.   
\end{proof}

After these preparations, we are now ready for the second main
result of this section.

\begin{lemma}
  [Qs-invariance of uniform relative density]
  \label{lem:unif_dense_qs}  Let  $F\colon T \to T'$ be a quasisymmetry between trees $T$ and $T'$.
 If $T$ has uniformly relatively dense branch points, then the same is true for $T'$. 
\end{lemma}
  
\begin{proof} Suppose $T$ has uniformly relatively dense branch points. 
  Let $x,y\in T$ with $x\ne y$ be arbitrary.  Then by Lemma~\ref{lem:unif_dense_diam}
  there is a branch point $z\in [x,y]$ of $T$ with
  \begin{equation} \label{eq:Tunifd0}
    H_T(z) \gtrsim \diam[x,y]. 
  \end{equation}

 Since $F$ is a homeomorphism, the branch points and arcs in $T$ 
 and $T'$ correspond to each other under $F$. This implies that in order to establish the statement, it is enough to show an inequality
 analogous to \eqref{eq:Tunifd0} for the image points of $x$, $y$, and $z$ under $F$.

 To see this, let $B\subset T$ be a branch of $z$ in $T$ with
 $H_T(z) = \diam(B)$. Then there is a point $w\in B$ with
  \begin{equation*}
    \abs{z-w} 
    \geq 
    \tfrac{1}{2} \diam(B) 
    = 
    \tfrac{1}{2} H_T(z) \gtrsim \diam[x,y].
  \end{equation*}
 Thus,  for each point $q\in [x,y]$ we  obtain
  \begin{equation*}
    \abs{z-w} \gtrsim \diam[x,y]\geq \abs{z-q}.
  \end{equation*}
  Let $x'\coloneqq F(x)$, $y'\coloneqq F(y)$, $z'\coloneqq F(z)$,
  $w'\coloneqq F(w)$, $q'\coloneqq F(q)$, and $B'\coloneqq F(B)$. 
  Note that $z'\in F([x,y])=[x',y']$. 
  Using Corollary~\ref{cor:weight_qs} and the fact that $F$ is a  quasisymmetry, 
  we see that 
  \begin{equation*}
    H_{T'}(z') 
    \asymp 
    \diam(B')
    \geq
    \abs{z'-w'}
    \gtrsim
    \abs{z'-q'}.
  \end{equation*}
  Since $q\in [x,y]$ was arbitrary, this gives
    \begin{equation*}
    H_{T'}(z') 
    \gtrsim 
    \sup_{q'\in [x',y']}\abs{z'-q'}
    \geq \tfrac{1}{2} \diam[x',y']\ge  \tfrac{1}{2} |x'-y'|.
  \end{equation*}
 Here $z'=F(z)\in [x',y']$ in a branch point of $T'$, and $C (\gtrsim)$ is independent of $x'$, $y'$, 
  and $z'$. The statement follows. 
\end{proof}

Our considerations show that uniform branching of a tree is
invariant under quasisymmetric equivalence. To complete the proof
of the ``only if'' implication in Theorem~\ref{thm:CSST_qs}, we
need to show that the continuum self-similar tree is uniformly
branching. Naturally, this means that first we have to define it
precisely. We address this in the next section.

%\begin{lemma}
%  \label{lem:qs_CSST_unif}
%  Let $\cT$ be the continuous self-similar tree, and $f\colon \cT\to
%  T$ be a quasisymmetry. Then $T$ is a trivalent qc-tree that 
%  is  uniformly branching. \end{lemma}
%
%\begin{proof} Since $f$ is a homeomorphism, it is clear that $T$ is a trivalent metric tree. It is doubling and of bounded turning, because $\cT$ has these properties and they are preserved under the  quasisymmetry $f$. So $T$ is a trivalent qc-tree. 
%
% By Proposition~\ref{prop:CSST_unif} the CSST $\cT$ is uniformly branching, because its branch points are uniformly separated and uniformly dense. It follows from 
%  Lemma~\ref{lem:unif_sepa_qs} and
%  Lemma~\ref{lem:unif_dense_qs} that $T$ has the latter two properties as well, because they are preserved under the quasisymmetry $f$. This shows that  $T$ is also uniformly branching and the statement follows. 
%\end{proof}

\section{The continuum self-similar tree}
\label{sec:cont-self-simil}

The  continuum self-similar tree  (abbreviated as CSST) is a standard model for  a
quasiconformal tree that is trivalent and uniformly branching. 
It was studied in the paper \cite{BT}, to which we refer for
more details and references to the literature.  

For the definition of the CSST we
consider the maps $g_k\colon \C\to \C$, $k=1,2,3$, given by
\begin{equation}
  \label{eq:def_fj}
  g_1(z)=\tfrac{1}{2}z-\tfrac{1}{2},
  \quad
  g_2(z) =\tfrac{1}{2}\bar{z} + \tfrac{1}{2},
  \quad
  g_3(z) = \tfrac{i}{2}\bar{z} + \tfrac{i}{2}.
\end{equation}
% \begin{equation}
%   \label{eq:def_fj}
%   g_1(z)=\tfrac{1}{2}z,
%   \quad
%   g_2(z) =\tfrac{1}{2}\bar{z} + \tfrac{1}{2},
%   \quad
%   g_3(z) = \tfrac{i}{2}\bar{z} + \tfrac{1}{2}.
% \end{equation}
% \begin{figure}[t]
%  \begin{overpic}[width=10cm,tics=5,
%    %grid
%    ]{CSST.png}
%    \put(-1,46){$0$}
%    \put(99,46){$1$}
%  \end{overpic}
%  \caption{The CSST.}
%  \label{fig:CSST}
%\end{figure}
The {\em continuum self-similar tree}  is the attractor $\cT$ of the iterated function system
$\{g_1,g_2,g_3\}$ (see \cite[Theorem 9.1]{Fa03} for a general result in this direction). In other words,  it is the unique non-empty compact set
$\cT\subset \C$ satisfying
\begin{equation}\label{eq:attracIFS}
  \cT = g_1(\cT) \cup g_2(\cT) \cup g_3(\cT). 
\end{equation}
See Figure~\ref{fig:CSST} for an illustration of the set $\cT$. 

% \begin{remark}
%   \label{rem:BT_conj}
%   In \cite{BT} the CSST was defined in a slightly different way as the attractor of an  iterated 
%   function system given by maps  $f_k$, $k=1, 2, 3$, that
%   are conjugates of our maps $g_k$ (see \cite[(1.1), Proposition 1.1, and Definition 1.2]{BT}). More precisely, 
%   if $\tau\colon \C\to\C$
%   is defined as $\tau(z) = \frac{1}{2}z +\frac{1}{2}$,  then 
%   $f_k = \tau^{-1}\circ g_k \circ \tau$ for
%   $k=1,2,3$.  This implies that if $\widetilde {\cT}$ is the version of the CSST as defined in \cite{BT}, then $\cT=\tau(\widetilde {\cT})$.
%   This does not make any essential difference, and we can use the results from \cite{BT} for the CSST with some  obvious adjustments. 
% \end{remark}

We  summarize some properties of the CSST. Since $\cT\sub \C$,  the CSST  inherits the Euclidean metric from
$\C$. Then $\cT$ is a trivalent metric tree \cite[Propositions
1.4 and 1.5]{BT}. It  is a {\em
  quasi-convex} subset of $\C$ in the following sense: there exists  a constant $L>0$ such that if  $x,y\in \cT$ are arbitrary and
$[x,y]$ is the unique arc in $\cT$ joining $x$ and $y$, then
\begin{equation}
  \label{eq:cT_geod}
  \abs{x-y} \leq \length[x,y] \leq L \abs{x-y}
\end{equation}
(see \cite[Proposition~1.6]{BT}).
 In particular, $\cT$ is of bounded
turning, since 
\begin{equation*}
  \diam[x,y]\le \length [x,y]\le L\abs{x-y}.
\end{equation*}
 As a subset of $\C$, the space $\cT$ is also doubling.  Thus, $\cT$
is a quasiconformal tree. In Proposition~\ref{prop:CSST_unif}
below, we will see that $\cT$ is uniformly branching.

Inequality ~\eqref{eq:cT_geod} implies that  if we  define
\begin{equation}
  \label{eq:def_geod_cT}
  \varrho(x,y) \coloneqq \length[x,y] 
\end{equation}
for $x,y\in \cT$, then $(\cT, \varrho)$ is a geodesic tree  that
is bi-Lipschitz equivalent to $\cT$
equipped with the Euclidean metric (by the identity map on
$\cT$). The space $(\cT, \varrho)$ is isometric to the abstract version
of the CSST as defined in the introduction (see the  
discussion at the end of
Section~4 in \cite{BT}). 

Considering $\cT$ as a subset of $\C$ has the advantage that 
we have the simple characterization \eqref{eq:attracIFS} of  $\cT$   with  explicit maps  $g_k$  as in  \eqref{eq:def_fj}. On the other
hand, we never consider points in $\C\setminus \cT$. Therefore, 
 we think of  $\cT$ as the ambient space when we discuss 
  topological properties of a   set $M\subset \cT$. 
 For example, we denote by $\partial
M$ the  (relative) boundary of $M$ in $\cT$. Naturally, this
is  in general different from the boundary of $M$ as a
subset of $\C$. Similarly, the interior of $M$ is always
understood to be relative to $\cT$.

The CSST  can be decomposed  into subtrees in a natural way. These
subtrees will be labeled by words in an alphabet. Here we use the
set $\A\coloneqq\{1,2,3\}$ as our \emph{alphabet}. It consists
of the \emph{letters} $1,2,3$. For each $n\in \N$, let
$\A^n\coloneqq\{w_1\dots w_n : w_j\in \A \text{ for }
j=1,\dots,n\}$ be the set of all \emph{words} of {\em length}
$n$. We set $\A^0\coloneqq\{\emptyset\}$; so the empty word
$\emptyset$ is the only word of length $0$. If $n\in \N_0$ and $w\in \A^n$, we denote by $\ell(w) \coloneqq n$
the length of $w$. Finally, we denote by
$\A^*\coloneqq \bigcup_{n\in \N_0} \A^n$ the set of all
\emph{finite} words. 

If $u=u_1\dots
u_k\in \A^k$ and $w=w_1\dots w_n\in \A^n$ with $k,n\in \N_0$, then 
$$uw\coloneqq u_1\dots u_kw_1\dots w_n\in \A^{k+n}$$ denotes the concatenation of the 
words $u$ and $w$. 

For a finite word $w=w_1\dots w_n\in \A^n$ of  length  $n\in
\N_0$, we define the map $g_w \colon \C\to \C$ by
\begin{equation*}
  g_w\coloneqq g_{w_1}\circ g_{w_2} \circ \dots \circ g_{w_n}.
\end{equation*}
Here it is understood  that $g_\emptyset\coloneqq\id_{\C}$ is the
identity map on $\C$ for the empty word $w=\emptyset$. The map
$g_w$ is a Euclidean similarity that scales distances by the
factor $2^{-\ell(w)}$. 

For each  $w\in \A^*$ we now set
\begin{equation}
  \label{eq:def_Tw}
  \cT_w\coloneqq g_w(\cT). 
\end{equation}
It follows from repeated application of \eqref{eq:attracIFS} 
that $\cT_w\sub \cT$. Hence $\cT_w=g_w(\cT)$ is a subtree of
$\cT$, because  $\cT_w$ is a tree as the image of $\cT$ under the
homeomorphism $g_w$. We use the same notation  for the restriction of $g_w$ to $\cT$ and consider this as a map 
$g_w\: \cT\ra g_w(\cT)=\cT_w\sub \cT$.

Since $\diam(\cT)=2$, we have
\begin{equation}
  \label{eq:diasmtile} \diam(\cT_w) = 2^{-\ell(w)+1}
 \end{equation}
 (see 
 \cite[(4.15)] {BT}).
We
call the sets of the form $\cT_w= g_w(\cT)$ the {\em tiles} of
$\cT$ and denote by
\begin{equation}
  \label{eq:defcX}
  \cX^*\coloneqq \{\cT_w: w\in \A^*\}
\end{equation}
the set of all tiles. If $\cT_w\in \cX^*$, then we call
$n=\ell(w)\in \N_0$ the {\em level} of the tile, which we also
denote by $\ell(\cT_w) \coloneqq n$. We then say that $\cT_w$ is a tile of
{\em level} $n$, or simply an $n$-{\em tile}. By \eqref{eq:diasmtile}
the level of  a tile is well defined.  The set  $\cT_\emptyset=\cT$ is the only $0$-tile, and  there are three
$1$-tiles, namely  $\cT_1$, $\cT_2$, $\cT_3$. For an illustration see Figure~\ref{fig:subtrees} which shows the $1$-tiles, and Figure~\ref{fig:2tiles} which shows some of the $2$-tiles (note that here  the dotted lines indicate tiles, but are  not part of the CSST).

\begin{figure}[h]
% \vspace{-1cm}
 \begin{overpic}[ scale=0.7
 , clip=true, trim=0mm 15mm 0mm 0mm
  %,width=10cm, tics=10,
   %  , grid
    ]{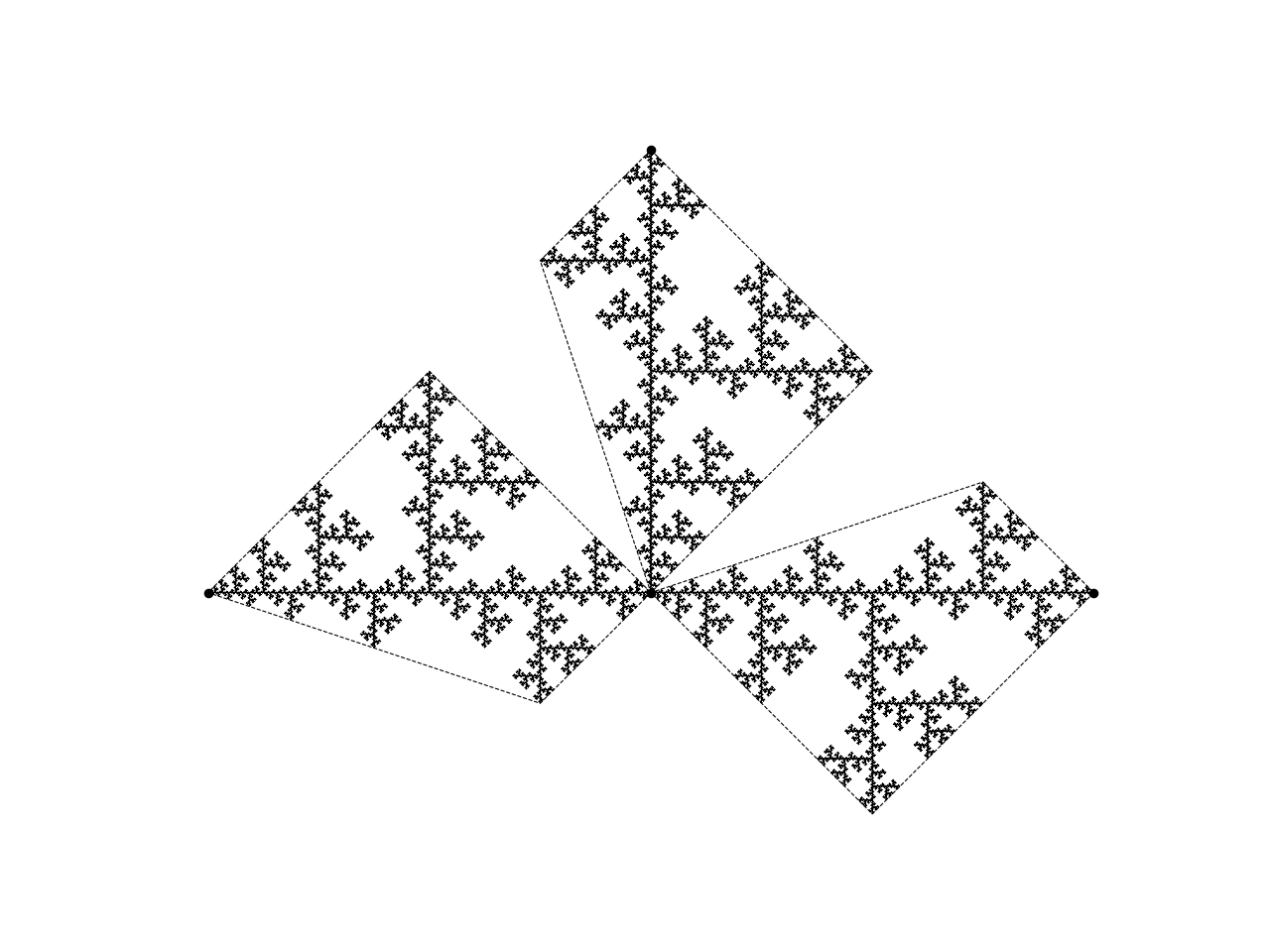}
    \put(49.5,56){ $i$}
    \put(12,15.5){ $-1$}
     \put(85,15.5){ $1$}
          \put(19, 29){ $\cT_1$}
            \put(76,6){ $\cT_2$}
              \put(60,46){ $\cT_3$}
            \put(49.5, 15.5){ $0$}
%             \put(32, 23.5){ $\cT_1$}
%                \put(67.5, 31.5){ $\cT_2$}
%                  \put(45.5, 45){ $\cT_3$}
                          \end{overpic}
 % \vspace{-1cm}
\caption{The CSST  and its subtrees $\cT_1$, 
$\cT_2$, 
$\cT_3$.}
\label{fig:subtrees}
\end{figure}

We record some  properties of $1$-tiles for later use.

\begin{lemma}  \label{lem:1tilesCSST}
  The following statements are true:
   \begin{enumerate}
  \item
    \label{item:1tilesCSST1}
    The point $0\in \C$ is a branch point of  $\cT$ with the three branches  $\cT_1,\cT_2,\cT_3$. 
  \item
    \label{item:1tilesCSST2}
    The points $-1$ and $1$ are leaves of $\cT$ with 
    \begin{align*}
      -1&\in \cT_1  \;\text{ and }\; {-1}\notin \cT_2,\cT_3;
      \\
      1&\in \cT_2 \;\text{ and }\; {\phantom{-}1}\notin \cT_1,\cT_3.
    \end{align*}
    % \begin{align*}
    %   -1\in \cT_1, \,1\in \cT_2, 
    % \end{align*}
%    and $1\notin\cT_1,-1\notin\cT_2,\,-1,1\notin \cT_3$.
  \item
    \label{item:1tilesCSST3}
    Let $x,y\in \cT$ be contained in distinct $1$-tiles. Then
    \begin{equation*}
      \abs{x-y}
      \gtrsim
      \abs{x} + \abs{y},
    \end{equation*}
    where $C(\gtrsim)$ is independent of $x$ and $y$.  
  \end{enumerate}
\end{lemma}

In particular, \ref{item:1tilesCSST1} means that the set
$\cT\setminus \{0\}$ has precisely three distinct components $U_k$,
$k=1,2,3$, and with suitable labeling we have 
$\mybar {U}_k=U_k\cup \{0\}=\cT_k$. Hence Lemma~\ref{lem:vt}~\ref{item:vt2} implies that 
  \begin{equation}
    \label{eq:bdry_T1}
    \partial \cT_1
    =
    \partial \cT_2
    =
    \partial \cT_3
    =\{0\}.
  \end{equation}
Moreover,  the subtrees $\cT_1$, $\cT_2$, $\cT_3$ of
  $\cT$ are pairwise disjoint except for the common point $0$. In other words, 
   \begin{equation}
    \label{eq:1tilesint}
 \cT_k\cap\cT_\ell=\{0\} \text{ for $k,\ell\in \{1,2,3\}$ with $k\ne \ell$.}   
\end{equation}

\begin{proof} [Proof of Lemma~\ref{lem:1tilesCSST}] 
 \ref{item:1tilesCSST1}
  This follows from 
  \cite[Lemma~4.5~(i)]{BT}.
  
  \smallskip \noindent 
   \ref{item:1tilesCSST2} By  \cite[Proposition~4.2 and Lemma~4.4 (ii)]{BT} the points 
    $-1$ and $1$ belong to $\cT$ and are leaves of $\cT$. 
    Moreover, $-1=g_1(-1)\in  \cT_1$ and
  $1=g_2(1) \in \cT_2$ (see   \eqref{eq:def_fj}). Since $0$ is the only common point of two
  distinct $1$-tiles, we also know that $1\notin \cT_2,
  \cT_3$ and  $-1\notin \cT_1, \cT_3$.

\smallskip  \noindent 
 \ref{item:1tilesCSST3} Let $x,y\in \cT$ be contained
  in two distinct $1$-tiles. Then $0\in [x,y]$ by
  \ref{item:1tilesCSST1}. If $L>0$ is  the constant in 
  \eqref{eq:cT_geod}, then we obtain 
  \begin{equation*}
    L\abs{x-y}
    \geq
    \length[x,y]
    =
    \length[x,0] + \length[0,y]
    \geq
    \abs{x} + \abs{y},
  \end{equation*}
  and \ref{item:1tilesCSST3} follows. 
\end{proof}

If  $w\in \A^*$ and  $n=\ell(w)\in \N_0$, then
$g_w\colon \cT\to \cT_w$ is a homeomorphism that maps the
$1$-tiles $\cT_1,\cT_2,\cT_3\subset \cT$ onto the $(n+1)$-tiles
$\cT_{w1}, \cT_{w2}, \cT_{w3}\subset \cT_w$. So the statements
in Lemma~\ref{lem:1tilesCSST} can be translated via the map $g_w$ to
statements about $\cT_w$. For example, $g_w(0)$ is a branch point of
$\cT_w$ with three branches given by the 
$(n+1)$-tiles $\cT_{w1}, \cT_{w2}, \cT_{w3}$. In particular, 
\begin{equation} \label{eq:wunion}
\cT_w =\cT_{w1}\cup \cT_{w2} \cup \cT_{w3}.
\end{equation}
Since this is true for all $w\in \A^*$, it is easy to see that if  
$m,n\in \N_0$ with $n\geq m$ and  $X$ is an   $n$-tile,  then 
there exists 
an $m$-tile $Y$ with $X\subset Y$.

\subsection*{Branch points of  \texorpdfstring{$\cT$}{T} and boundaries of tiles}
\label{sec:branch-points-ct}

In order to show that $\cT$ is uniformly branching  (see Proposition~\ref{prop:CSST_unif}), we need to locate the
branch points of $\cT$ and understand their height. The key
observation is that a point $p$ is a branch point of $\cT$ if and
only if it is a (relative) boundary point of some tile in $\cT$ (see
Lemma~\ref{lem:bdry}~\ref{item:bdry3} and
Lemma~\ref{lem:branch}~\ref{item:branch1}). 

% If $M$ is a subset of $\cT$, then we denote by $\partial M$ its
% relative boundary in $\cT$, i.e., the boundary of $M$ if we
% consider $\cT$ and not $\C$ as the ambient space. Note that this
% is in general different from the boundary of $M$ as a subset of
% $\C$.

As we will see, there are two ways to describe (relative) boundary points of tiles: as
images of the leaves $-1$ or $1$ of $\cT$  under some map $g_w$ or as images of the branch
point $0$ of $\cT$. 
 
\begin{lemma}[Boundaries of tiles]
  \label{lem:bdry} The following statements are true: 
  \begin{enumerate}
  \item 
    \label{item:bdry1} Let $\cT_w$ with $w\in \A^n$ be a tile of level
  $n\in \N$. Then 
    $\cT_w$ has one  or two boundary points and
    $\partial \cT_w\subset  \{g_w(-1), g_w(1)\}$.   
    
     Moreover, if 
    $w=u3$ with $u\in \A^{n-1}$, then  $\partial \cT_w$ consists of a single point.  
% 
%  \item
%    \label{item:bdry2}
%   
    % \begin{equation*}
    % w=w_1\dots w_{n-1}3,\; w=1\dots 1, \text{ or } w=2\dots 2. 
    % \end{equation*}
    %
  \item
    \label{item:bdry3}
    For $n\in \N$ the set of all boundary points of all $n$-tiles
    is equal to
    \begin{equation*}
      \cV^n
      \coloneqq
      \{g_u(0) \colon u\in \A^m,\, 0\leq m\leq n-1\}. 
    \end{equation*}
  \end{enumerate} 
\end{lemma} 
  
\begin{proof}   \ref{item:bdry1}
  We prove this by induction on $n=\ell(w)\in \N$. 
 
 For  $n=1$ the statement follows from \eqref{eq:bdry_T1} and 
  \begin{equation}
    \label{eq:gk12}
    0=g_1(1)= g_2(-1)= g_3(-1).
  \end{equation}
%  Thus \ref{item:bdry1} and \ref{item:bdry2} are true for
%  $n=1$. Recall that $\emptyset\in \A^0$ is the only word of
%  length $0$ and $0=g_\emptyset(0)$. This means that
%  \ref{item:bdry3} holds for $n=1$ as well.

  Suppose the statement is true for all words of length
  $ n\in \N$, and  let $w\in \A^{n+1}$ be arbitrary. Then $w=uk$ with
  $u\in \A^n$ and $k\in \A=\{1,2,3\}$.
  % In the following, we will use the
  % fact that $g_u$ is a homeomorphism of $\cT$ onto $\cT_u$. We
  % will use this without further discussion to translate valid
  % statements for $\cT$ into analogous statements for the subtree
  % $\cT_u$ of $\cT$.  For example, $0$ is a branch point of
  % $\cT$, and so $g_u(0)$ is a branch point of $\cT_u$. We also
  % use mapping properties of the maps $g_k$ such as
  % \eqref{eq:gk12} without further discussion.
The map   $g_u$ is a homeomorphism that
  sends $\cT$ onto $\cT_u$ and the $1$-tile $\cT_\ell$ onto the
  $(n+1)$-tile $\cT_{u\ell}$ for $\ell=1,2,3$. It follows from
  \eqref{eq:bdry_T1}, \eqref{eq:1tilesint}, and \eqref{eq:wunion} that the sets 
 $\cT_{u\ell}\setminus \{g_u(0)\}$, $\ell \in \{1,2,3\}$,  are pairwise disjoint, but that every neighborhood of $g_u(0)$ contains points from each of these sets. This implies that $g_u(0)$ is a boundary point of  
$ \cT_{u\ell}$ for $\ell=1,2,3$. 
 In particular,  $g_u(0)\in  \partial \cT_w=\partial \cT_{uk}$, and so 
$ \partial \cT_w$  contains at least  one point.
  
  Let $p\in \partial \cT_w$ be arbitrary. Then 
  $p\in \cT_w\sub \cT_u$ is contained either  in the (relative) interior of $\cT_u$ or in 
  $\partial \cT_u$. In the first case, each sufficiently small
  (relative) neighborhood $U$ of $p$ is contained in
  $ \cT_u= \cT_{u1}\cup \cT_{u2} \cup  \cT_{u3}$;   
so $U$ must contain a point in  $\cT_u\setminus \cT_w$, because 
 $p\in \partial \cT_w$. Applying the homeomorphism $g_u^{-1}$, we see that every neighborhood of $g_u^{-1}(p)\in \cT_k$ contains points from the set $\cT\setminus\cT_k$, and so $g_u^{-1}(p)\in \partial \cT_k$. Hence  $g_u^{-1}(p)=0$ by \eqref{eq:bdry_T1}, and so $p=g_u(0)$. 
  
   In
  the second case, 
  $p \in \partial \cT_u \subset \{ g_u(-1),g_u(1)\}$, since
  \ref{item:bdry1} is true for $\cT_u$ by induction hypothesis. 
  Lemma~\ref{lem:1tilesCSST}~\ref{item:1tilesCSST2} shows that
  each $1$-tile contains at most one of the points $-1$ and $1$; namely, $\cT_1$  contains only $-1$, $\cT_2$ contains only
  $1$, and $\cT_3$ contains neither $-1$ nor $1$.  
 We conclude
  that $\cT_w= g_u(\cT_k)$ contains at most one of the points
  $g_u(-1)$ and $g_u(1)$ (in its boundary). So  $\partial \cT_w$
 consists of at most two points, namely the point $g_u(0)$ and at most one of the points $g_u(-1)$ and $g_u(1)$. 
%
%  Since $-1,1\notin \cT_3$, we have
%  $g(-1,), g(1)\notin \cT_{u3} = g_u(\cT_3)$. Thus in the
%  case $k=3$, meaning that $w=u3$, $\partial\cT_w\subset \cT_w$
%  contains none of the points $g_u(-1),g_u(1)$, thus at most $1$
%  point. We have proved \ref{item:bdry2} for $\cT_w$.
%
%  % When $k=3$, meaning that $w=u3$, $\partial\cT_w$ contains none
%  % of the points $g_u(-1),g_u(1)$, thus at most $1$ point. We have
%  % proved \ref{item:bdry2} for $\cT_w$.
%
% % \smallskip
 More precisely,  our discussion shows that we have the
 following alternatives for $k\in \{1,2,3\}$.  
 
 \smallskip 
If $k=1$, then  we have 
 $\partial \cT_w=\partial  \cT_{u1}\sub \{g_u(-1), g_u(0)\}. $
 Since 
 $$  g_u(-1) = g_u(g_1(-1)) = g_w(-1)  \text{ and } g_u(0) = g_u(g_1(1)) = g_w(1),$$ 
 it follows that 
 $\partial \cT_w=\partial  \cT_{u1}\sub \{g_u(-1), g_u(0)\}=\{g_w(-1), g_w(1)\}$. 
 
 \smallskip
If $k=2$, then  
 $\partial \cT_w=\partial  \cT_{u2}\sub \{g_u(0), g_u(1)\}. $
 Now 
 $$ g_u(0) = g_u(g_2(-1)) = g_w(-1)    \text{ and }
    g_u(1) = g_u(g_2(1)) = g_w(1), $$ 
    which implies 
 $\partial \cT_w=\partial \cT_{u2}\sub \{g_u(0), g_u(1)\}=\{g_w(-1), g_w(1)\}$.

\smallskip
If   $k=3$, then 
 $\partial \cT_w=\partial  \cT_{u3}=\{g_u(0)\}. $
In particular,  $ \cT_w$ has precisely one boundary point in this case. Moreover,  we have
$$ g_u(0) = g_u(g_3(-1)) = g_w(-1), $$
and so $\partial \cT_w=\partial \cT_{u3}=\{g_u(0)\}=\{g_w(-1)\}$. 

\smallskip
This shows that the statement is true for the tile $\cT_w$. This completes the inductive step and \ref{item:bdry1} follows. 
  
%  \begin{align*}
%    \intertext{When $k=1$:}
%    &p=g_u(0) = g_u(g_1(1)) = g_w(1), 
%    \\
%    \text{ or }\;
%    &p=g_u(-1) = g_u(g_1(-1)) = g_w(-1);
%    %  
%    \intertext{when  $k=2$:}
%    &p=g_u(0) = g_u(g_2(-1)) = g_w(-1), 
%    \\
%    \text{ or }\;
%    &p=g_u(1) = g_u(g_2(1)) = g_w(1);
%    %
%    \intertext{when  $k=3$:}
%    &p=
%  \end{align*}
   
  % \smallskip
  % Observe from the case $k=3$ above that $\partial \cT_w$
  % consists of a single point when $w=u3$, proving
  % \ref{item:bdry2} for $\cT_w$.

  \smallskip
  \ref{item:bdry3} For $n\in \N$ 
we denote by 
   $$B^n\coloneqq\bigcup_ {w\in \A^n} \partial \cT_w$$
  the set of all boundary points of all $n$-tiles. We have to show that 
  $B^n=\cV^n$. We prove this by induction on $n\in \N$.

 Let $n=1$. Then   \eqref{eq:bdry_T1} 
 shows that $ B^1=\{0\}=\{g_\emptyset(0)\} =\cV^1$ as desired.  
   
  Suppose the statement is true for some $n\in \N$. To show that it is also true for $n+1$, we first verify  that   $\cV^{n+1}\subset B^{n+1}$. To this end, let
  $p\in \cV^{n+1}$ be arbitrary. Then $p=g_u(0)$ for some $u\in \A^m$
  with $0\leq m\leq n$. If $m=n$, then $p$
  is in the boundary of the $(n+1)$-tiles
  $\cT_{u1}, \cT_{u2}, \cT_{u3}$, as have seen in the proof of  \ref{item:bdry1}. So $p\in B^{n+1}$.
  
  If $m\leq n-1$, we have
  $p\in \cV^{n}$, and so  $p\in \partial \cT_w\sub \cT_w$ for some 
  $w\in \A^n$  by induction  hypothesis.  
  By \eqref{eq:wunion} there exists $k\in \{1,2,3\}$ such that 
  $p\in \cT_{wk}\subset \cT_w$. Then  every neighborhood of $p$ contains points in 
  $\cT\setminus \cT_w\subset \cT\setminus \cT_{wk}$. This implies  that  $p\in \partial \cT_{wk}$, and so  
  $p\in B^{n+1}$ also in this case. Thus
  $\cV^{n+1}\subset
  B^{n+1}$. 

  To see the reverse inclusion $B^{n+1}\subset \cV^{n+1}$, let
  $p\in B^{n+1}$ be arbitrary. Then $p$ is a boundary point of an $(n+1)$-tile, say $p\in \partial\cT_{uk}$, where 
  $u\in \A^n$ and $k\in \{1,2,3\}$. In the proof of  \ref{item:bdry1}
  we have seen that
  either $p=g_u(0)$ or $p\in \partial \cT_u$. In the first case,
  $p\in \cV^{n+1}$, while in the second case,
  $p\in \cV^{n}\subset \cV^{n+1}$ by induction hypothesis. This
shows  $B^{n+1}\subset \cV^{n+1}$, and
so $B^{n+1}=\cV^{n+1}$.

This completes the inductive step and the statement follows. 
\end{proof} 

For $n\in \N$ let $\cV^n$ be as in Lemma~\ref{lem:bdry}~\ref{item:bdry3},
and define 
\begin{equation*}
  \cV^*
  \coloneqq
  \bigcup_{n\in \N} \cV^n
  =
  \{g_w(0) : w\in \A^*\}. 
\end{equation*}
Then $\cV^*=\bigcup_{X\in \cX^*} \partial X$
 (note that for the only $0$-tile $X=\cT$, we have $\partial X=\emptyset$). So $\cV^*$ 
 is equal to the set of all boundary points of all tiles of $\cT$.

\begin{lemma}
  [Branch points of $\cT$]
  \label{lem:branch}
The following statements are true: 
  \begin{enumerate}
  \item
  \label{item:branch1}
  The set $\cV^*$ is equal to the set of all branch points of $\cT$.

\item 
  \label{item:branch2} The points $g_w(0)$, $w\in \A^*$, are all distinct. More
  precisely, if $v,w\in \A^*$, $v\ne w$, then
  \begin{equation*}
    \abs{g_v(0)-g_w(0)}
    \gtrsim
    \min \{ 2^{-\ell(v)}, 2^{-\ell(w)}\}>0 
  \end{equation*}
  with $C(\gtrsim)$ independent of $v$ and $w$.

\item 
  \label{item:branch3}  
  If $w\in \A^*$ and $B_1$, $B_2$, $B_3$ are the three branches
  of the branch point $p=g_w(0)$ of $\cT$, then with suitable
  labeling we have
  \begin{equation*}
    g_{w1}(\cT)
    \sub
    B_1,\,  g_{w2}(\cT)\sub B_2,\, g_{w3}(\cT)=B_3.
  \end{equation*}
  In particular, $H_{\cT}(p)=2^{-\ell(w)}$. 
\end{enumerate} 
\end{lemma}

\begin{proof}
\ref{item:branch1}
This was proved in \cite[Lemma 4.7]{BT}.

  \smallskip
\ref{item:branch2}
We have
  $g_1(0)=-\tfrac{1}{2}$,
  $g_2(0)=\tfrac{1}{2}$, 
  $g_3(0)=\tfrac{i}{2}$,
 and so
$\abs{g_k(0)}= \tfrac{1}{2}$ 
for $k\in \A=\{1,2,3\}$. It follows that  if 
$u\in \A^*$ and $k\in \A$ are  arbitrary, then  
\begin{align*}
  \notag
  \abs{g_u(0)-g_{uk}(0)}
  &=
    \abs{g_u(0)-g_u(g_k(0))}
    =
    2^{-\ell(u)} \abs{g_k(0)}
  \\
  &=
    2^{-\ell (u)-1}. 
\end{align*}
So  if $u=a_1\dots a_n\in \A^n$ with $n\in \N$, then 
\begin{align}
  \label{eq:dist0}
  \abs{g_{u}(0)}
  &\ge
    \abs{g_{a_1}(0)}- \abs{g_{a_1}(0)-g_{a_1\dots a_n}(0)}
  \\ \notag
  &\ge 
    \frac12- \sum_{j=1}^{n-1} \abs{g _{a_1\dots a_j}(0)
    -g_{a_1\dots a_ja_{j+1}}(0)} 
  \\ \notag
  &=
    \frac12- \sum_{j=1}^{n-1} 2^{-j-1}= 2^{-n}=2^{-\ell(u)} .
\end{align}

Now let $v,w\in \A^*$ with $v\ne w$ be arbitrary, and let
$u\in \A^*$ be the longest common initial word of $v$ and $w$. Then $v=uv'$ and $w=uw'$ with
$v',w'\in \A^*$, where either $v'$ or $w'$ is the empty word, or both 
  $v'$ and $w'$ are non-empty, but have a different initial letter.
 Accordingly,  we  consider two cases.
  
\smallskip 
{\em Case 1:}
$v'=\emptyset$ or $w'=\emptyset$.
We may assume $w'=\emptyset$. Then $v'\ne\emptyset$, and so
$\ell(v)\ge \ell (w)$. Moreover, by \eqref{eq:dist0} we have
\begin{align*}
  \abs{g_v(0)-g_w(0)}
  &=
    \abs{g_u(g_{v'}(0))-g_u(g_{w'}(0))}
  \\
  &=
    2^{-\ell(u)}\abs{g_{v'}(0)}\ge 2^{-\ell(u)-\ell(v')}
  \\
  &=
    2^{-\ell(v)}
    =
    \min \{ 2^{-\ell(v)}, 2^{-\ell(w)}\}.
\end{align*}  

 \smallskip 
{\em Case 2}:   $v',w'\ne \emptyset$. 
Then $v'$ and $w'$ necessarily start with different letters, say
$v'=kv''$ and $w'=\ell w''$ with $k,\ell \in \A=\{1,2,3\}$, $k\ne \ell$, and
$v'',w''\in \A^*$. Then $g_{v'}(0)\in \cT_k$ and
$g_{w'}(0)\in \cT_\ell$, i.e., these points are contained in
distinct $1$-tiles. Then
Lemma~\ref{lem:1tilesCSST}~\ref{item:1tilesCSST3}  and \eqref{eq:dist0} imply that 
\begin{align*}
  \abs{g_{v'}(0)-g_{w'}(0)}
  & \gtrsim
    \abs{g_{v'}(0)}+\abs{g_{w'}(0)}
  \\
  &\ge
    2^{-\ell(v')}+2^{-\ell(w')}
    \ge  
    \min \{ 2^{-\ell(v')}, 2^{-\ell(w')}\}
\end{align*}  
with $C(\gtrsim)$ independent of $v$ and $w$.
We conclude that 
\begin{align*}
  \abs{g_v(0)-g_w(0)}
  &=
    \abs{g_u(g_{v'}(0))-g_u(g_{w'}(0))}
  \\
  &=
    2^{-\ell(u)}\abs{g_{v'}(0)-g_{w'}(0)}
  \\ 
  &\gtrsim   2^{-\ell(u)} \min \{ 2^{-\ell(v')}, 2^{-\ell(w')}\}
  \\
  &=
    \min \{ 2^{-\ell(v)}, 2^{-\ell(w)}\}.
 \end{align*}  
The statement follows.

\smallskip
\ref{item:branch3}
If $w\in \A^*$, then $p=g_w(0)$ is a branch point of $\cT_w$ with
the three (distinct) branches $\cT_{w1}$, $\cT_{w2}$, $\cT_{w3}$.
Hence there exist unique distinct branches $B_1$, $B_2$, $B_3$ of $g_w(0)$
in $\cT$ such that $\cT_{wk}\sub B_k$ for $k=1,2,3$ 
(see Lemma~\ref{lem:subt}~\ref{item:subt0}).  By 
Lemma~\ref{lem:bdry}~\ref{item:bdry1} we have $\partial \cT_{w}\sub \{g_w(-1), g_w(1)\}$. Moreover, 
$g_w(-1), g_w(1)\notin \cT_{w3}$, since $-1,1\notin
\cT_3$. Thus $\cT_{w3}\cap \partial \cT_{w}=\emptyset$. Hence it
follows from Lemma~\ref{lem:subt}~\ref{item:subt2} that
$\cT_{w3}$ is actually a branch of $p=g_w(0)$ in $\cT$, and so
$\cT_{w3}=B_3$. The first part of the statement follows.

For the second statement note that
\begin{equation*}
  \diam(B_k)
  \ge
  \diam (\cT_{wk})
  =
  2^{-\ell(w)}
\end{equation*}
for $k=1,2$, while
\begin{equation*}
  \diam(B_3)
  =
  \diam (\cT_{w3})
  =
  2^{-\ell(w)}.
\end{equation*}
It follows that
\begin{equation*}
  H_{\cT}(p)
  =
  \min\{\diam(B_k): k\in \{1,2,3\}\}
  =
  2^{-\ell(w)},
\end{equation*}
as desired. 
\end{proof}

\begin{proposition}
  \label{prop:CSST_unif}
  The continuous self-similar tree $\cT$ is uniformly
  bran\-ching, i.e., its branch points are uniformly relatively
  separated and uniformly relatively dense.
\end{proposition}

\begin{proof}
  Let $p,q\in \cT$ be two distinct branch points of $\cT$. Then
  $p=g_v(0)$ and $q=g_w(0)$ with $v,w\in \A^*$, $v\ne w$, by
  Lemma~\ref{lem:branch}~\ref{item:branch1}. Moreover,
  % by
  % Lemma~\ref{lem:branch}~\ref{item:branch1}
  % and~\ref{item:branch3}
  we have
  \begin{align*}
    \abs{p-q}
    &=
      \abs{g_v(0)-g_w(0)}
    \\
    & \gtrsim
      \min\{ 2^{-\ell(v)}, 2^{-\ell(w)}\}
      && \text{ by Lemma~\ref{lem:branch}~\ref{item:branch2}}
    \\
    &=
      \min\{ H_{\cT}(g_v(0)), H_{\cT}(g_w(0))\}
      && \text{ by  Lemma~\ref{lem:branch}~\ref{item:branch3}}
    \\
    &= \min\{ H_{\cT}(p), H_{\cT}(q)\}. 
  \end{align*}
  Here the implicit constant is independent of $p$ and $q$. This
  shows that $\cT$ has branch points that are uniformly
  relatively separated.

  To establish the other property, let $x,y\in \cT$, $x\ne y$, be
  arbitrary. Since $\diam(\cT)=2$, and so $2\ge |x-y|>0$, there
  exists a unique number $n\in \N$ such that
  $2^{-n+2}\ge |x-y|>2^{-n+1}. $ Then we can find $w\in \A^*$ with
  $\ell(w)=n$ such that $x\in \cT_w$. Since
  \begin{equation*}
    \diam (\cT_w)
    =
    \diam (g_w(\cT))
    =
    2^{-n+1}<\abs{x-y},  
  \end{equation*}
  we then have $y\not \in \cT_w$. So as we travel from $x$ to $y$
  along the arc $[x,y]$ that joins $x$ and $y$ in $\cT$, there is
  a last point $p\in [x,y]\cap \cT_w$. Then $(p, y]$ is non-empty
  and disjoint from $\cT_w$, and so $p\in \partial \cT_w$. By
  Lemma~\ref{lem:bdry}~\ref{item:bdry3} we have
  $p\in \cV^n\subset \cV^*$, and so $p$ is a branch point of $\cT$ by
  Lemma~\ref{lem:branch}~\ref{item:branch1}.

  Since $p\in \cV^n$, we know that $p=g_v(0)$ for some
  $v\in \A^*$ with $0\leq \ell(v) \leq n-1$. Now 
  Lemma~\ref{lem:branch}~\ref{item:branch3} implies that 
  \begin{equation*} 
    H_{\cT}(p)
    =
    H_{\cT}(g_v(0))
    =
    2^{-\ell (v)}
    \ge
    2^{-\ell (w)}
    =
    2^{-n}
    \asymp \abs{x-y}.
  \end{equation*}
  Again the implicit multiplicative constant is independent of
  the choices here. This shows  that the branch points of $\cT$
  are also uniformly relatively dense. The statement follows.
\end{proof}

Note  that Proposition~\ref{prop:CSST_unif} together
with Lemma~\ref{lem:unif_sepa_qs} and
Lemma~\ref{lem:unif_dense_qs} essentially shows the ``only if''
implication in Theorem~\ref{thm:CSST_qs}. 

\subsection*{Miscellaneous results on tiles in
  \texorpdfstring{$\cT$}{T}}
\label{sec:some-misc-results}
We conclude this section with some auxiliary results, which will be
used in the proof of the other implication of
Theorem~\ref{thm:CSST_qs} (namely in the proof
of Lemma~\ref{lem:decomp}). 

Let $\cV\subset \cT$
be a finite set of branch points of $\cT$. Recall from
Section~\ref{sec:topology-trees} that $\cV$ then decomposes $\cT$
into a set of tiles $\cX$.
% Note that such a tile $X\in \cX$ is
% not necessarily a tile of $\cT$, meaning in the set $\cX^*$ as
% defined in \eqref{eq:def_Tw}.
Note that such a tile $X\in \cX$ is
not necessarily in $\cX^*$, i.e., a tile of $\cT$ as
defined in \eqref{eq:def_Tw}.
For
example, if  $\cV=\{-1/2\}$, then the decomposition $\cX$
induced by $\cV$ consists of the  sets  $X_1=\cT_{11}$,
$X_2= \cT_{12}\cup \cT_2\cup \cT_3$, and  $X_3=\cT_{13}$.
 Here $X_2$ is not a tile of $\cT$.  This can be seen from  Figure~\ref{fig:2tiles} which shows the relevant $2$-tiles.

\begin{figure}[h]
% \vspace{-1cm}
 \begin{overpic}[ scale=0.7
 , clip=true, trim=0mm 15mm 0mm 0mm
  %,width=10cm, tics=10,
    %, grid
    ]{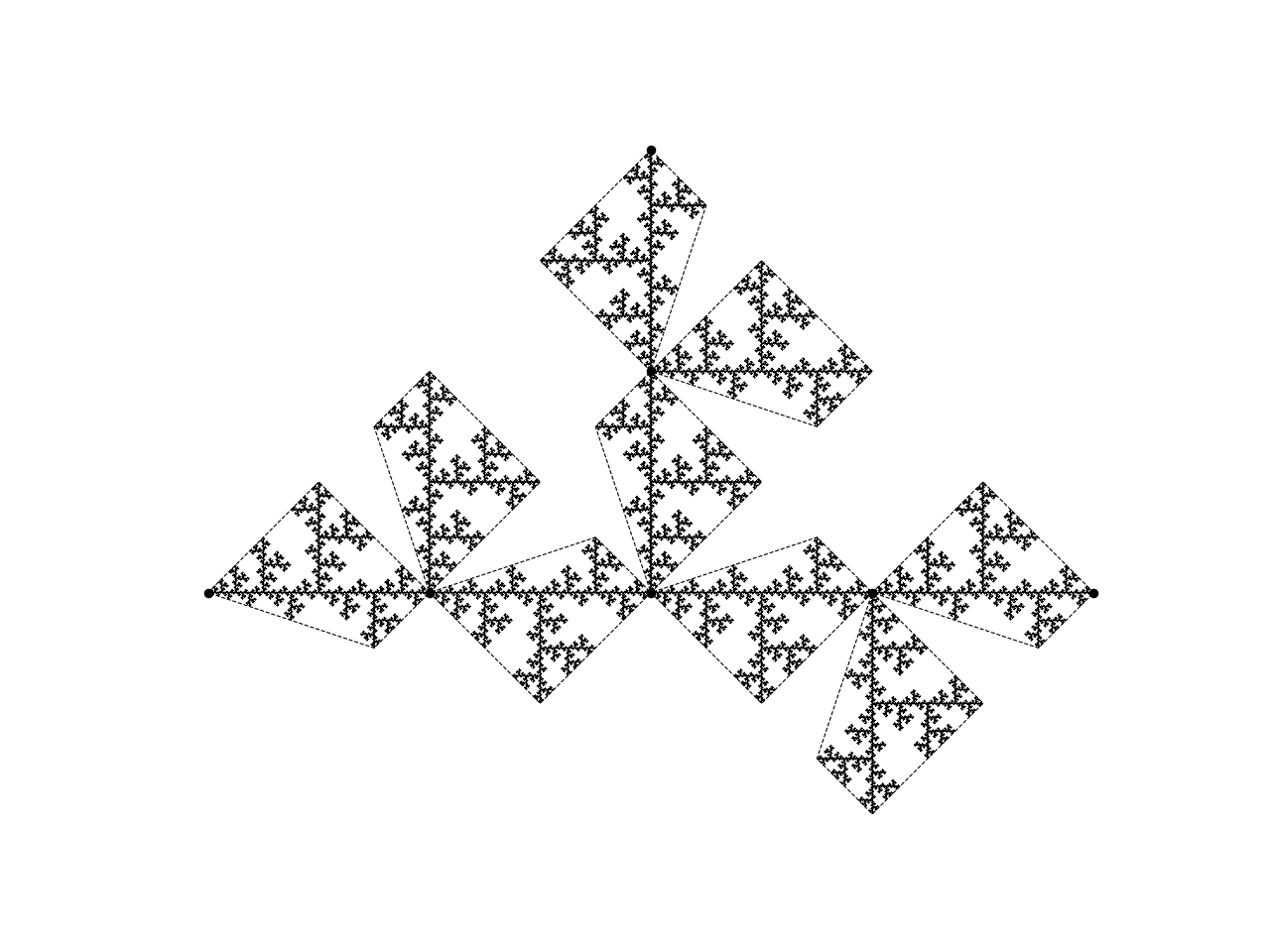}
 %   \put(49.5,56){ $i$}
    \put(12,15.5){ $-1$}
     \put(85,15.5){ $1$}
          \put(12, 23){ $\cT_{11}$}
           \put(45, 10){ $\cT_{12}$}
               \put(22, 33){ $\cT_{13}$}
   %        \put(76,6){ $\cT_2$}
               \put(82, 23){ $\cT_{22}$}
           %   \put(60,46){ $\cT_3$}
 %         \put(49.5, 15.5){ $0$}
%             \put(32, 23.5){ $\cT_1$}
%                \put(67.5, 31.5){ $\cT_2$}
%                  \put(45.5, 45){ $\cT_3$}
   \put(29.5, 13.5){ $-\tfrac12$}
     \put(66.5, 23){ $\tfrac12$}
                          \end{overpic}
 % \vspace{-1cm}
\caption{The CSST $\cT$ and some of its $2$-tiles.}
\label{fig:2tiles}
\end{figure}
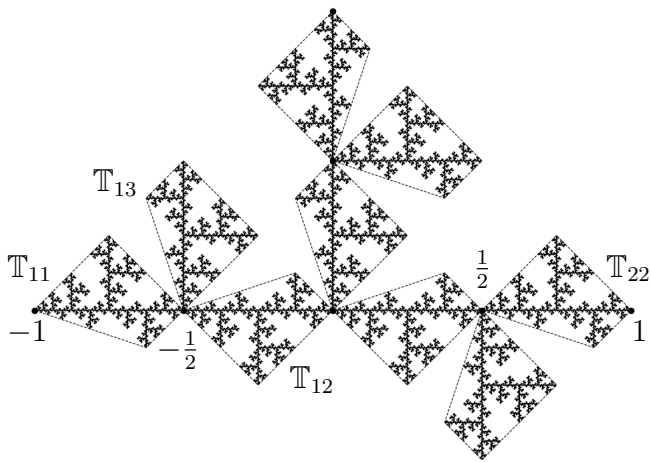

We now consider the special case when such a decomposition
of $\cT$ 
produces tiles as in  \eqref{eq:defcX}.

\begin{lemma}
  \label{lem:admh}
  Let $\cV\sub \cT$ be a finite set of branch points of $\cT$,
  and let $\cX$ be the set of tiles in the decomposition of $\cT$
  induced by $\cV$.
  
 Suppose that   $\cX\sub \cX^*$. Then the levels of  tiles in $\cX$
  satisfy
  \begin{equation*}
    \text{$\ell (X)\le \#\cV$ for each $X\in \cX$.}  
  \end{equation*}  
  If, in addition, 
  $\cV\neq \emptyset$, then we have $\ell(X) \geq 1$ for each
  $X\in \cX$.
\end{lemma}

The requirement $\cX\sub \cX^*$ means that each $X\in \cX$ is a
set of the form $X=\cT_w$ for some $w\in \A^*$ as in
\eqref{eq:def_Tw}. Accordingly, the level
$\ell(X)=\ell(w)\in \N_0$ of $X$ is well defined.

In this context, the concept of a {\em tile} has two different
meanings. Namely, a tile can be an element of the set of tiles
$\cX$ in the decomposition of $\cT$ induced by the given finite
set $\cV$, or a tile can be an element of $\cX^*$ as defined in
\eqref{eq:defcX}. If we require $\cX\sub \cX^*$, then every tile
in $\cX$ is a tile in $\cX^*$, but not every tile in $\cX^*$ is
necessarily a tile in $\cX$.  An $n$-tile, or a tile (without
further specification), is always understood to be in $\cX^*$. It
will be stated explicitly if we refer to a tile in $\cX$.

\begin{proof} 
  The second statement is immediate. Indeed, the only tile of
  level $0$ is $X= \cT$. If $\cV\ne \emptyset$, then for each
  tile $X\in \cX\sub \cX^*$ we have $X\ne \cT$ as follows from
  Lemma~\ref{lem:vt}~\ref{item:vt1}, and so $\ell(X)\ge 1$.

  Let us now prepare the proof of the first statement. Recall
  that if $m,n\in \N_0$ with $n\geq m$, then each $n$-tile
  $X\in \cX^*$ is contained in an $m$-tile $Y\in \cX^*$.  We also
  know by \eqref{eq:bdry_T1} that the point $0\in \cT$ belongs to
  the boundary of each $1$-tile, namely $\cT_1,\cT_2,\cT_3$.
  
  \smallskip
  \emph{Claim~1. }Let $X\in \cX^*$ be a tile with $0\in X$. Then either
  $0\in \partial X$ or $X=\cT$.
  \smallskip

  Indeed, suppose $n=\ell(X)\in \N_0$ is the level of a tile $X\in \cX^*$ with
   $0\in X$. Assume that
   $0$ is contained in the interior of $X$.
% Here and for the rest of the proof  we regard $\cT$ as the ambient space.
   Recall that we regard $\cT$ as the ambient space here, as
   before. 
  If $n\geq 1$, there
  exists  a $1$-tile $Y\supset X$. Then $0$ is an interior point of
  $Y$, but we know that $0\in \partial Y$. This is a contradiction. So $0\in \partial X$ unless  $n=0$. In the latter  case,    $X=\cT$ and $0\notin \partial\cT =\emptyset$. 
Claim~1 follows.
  \smallskip
  
  We are now ready to prove the first (i.e., the main) statement
  in the lemma by induction on $\#\cV\in \N_0$.
  
  To start the induction, assume $\#\cV=0$. Then $\cV=\emptyset$
  and $X=\cT$ is the only tile in $\cX$. Then clearly
  $\ell(X)=0\le \#\cV=0$, as desired.

  For the inductive step we assume that $\#\cV\ge 1$ and that the
  statement is true for all decompositions of $\cT$ into tiles induced  by
  sets of branch points with fewer than $\#\cV$ points.

  \smallskip 
  {\em Claim~2.} We have $0\in \cV$. 
  \smallskip

  To see this, note that $\cX$ is a covering of  $\cT$ (see Lemma~\ref{lem:vt}~\ref{item:vt8}). Hence there  exists a 
  tile $X\in \cX$ with $0\in X$. By Claim~1, there are two
  possible cases.  If $X=\cT$, we have $\cV=\emptyset$
  contradicting our assumption $\#\cV\geq 1$. Thus we must have
  $0\in \partial X\sub \cV$ (see 
  Lemma~\ref{lem:vt}~\ref{item:vt2}), which implies $0\in \cV$, as desired.
  Claim~2 follows.
  \smallskip
  
  We know that  $0$ splits $\cT$ into the three branches
  $\cT_1, \cT_2, \cT_3$. Since $0\in \cV$ by Claim~2,
  each tile $X\in \cX$ is contained in one such branch $\cT_k$ (see Lemma~\ref{lem:vt}~\ref{item:vt1}).
  We now want to   apply the induction hypothesis
  to these subtrees $\cT_k$.

  For this we  fix $k\in \{1,2,3\}$, and  let
  $$\cV_k \coloneqq \cV\cap
   \inte(\cT_k)= (\cV \cap \cT_k)\setminus\{0\}. $$  This  set  consists of branch points  of $\cT_k$ as follows from  
   Lemma~\ref{lem:subt}~\ref{item:subt0}, and induces a
  decomposition $\cX_k$ of the tree $\cT_k$ into tiles. Indeed,
  $X\in \cX_k$ if and only if $X\in \cX\sub \cX^*$ and $X\subset \cT_k$ (see Lemma~\ref{lem:vtt}~\ref{item:vtt3}).
  In particular,  $\cX_k\sub \cX^*$.
 Moreover, $\#\cV_k< \#\cV$, since $0\in \cV$ is not contained in
  $\cV_k$.
  
The map $g_k^{-1} \colon \cT_k \to\cT$ is a
  homeomorphism that sends  $(n+1)$-tiles contained in $\cT_k$ to
  $n$-tiles of $\cT$ (see \eqref{eq:def_fj} and \eqref{eq:def_Tw}).
   Let
  $\widetilde{\cV}_k\coloneqq g_k^{-1}(\cV_k)$. The set 
  $\widetilde{\cV}_k$ consists of branch points of $\cT$ and  induces the
  decomposition $\widetilde{\cX}_k = g_k^{-1}(\cX_k)=\{g_k^{-1}(X): X\in \cX_k\}$ of $\cT$ into
  tiles.  Then $\widetilde{\cX}_k \sub \cX^*$. Since $\#\widetilde{\cV}_k = \#\cV_k < \#\cV$, we may
  apply the induction hypothesis to $\widetilde{\cV}_k$ and conclude
  \begin{equation*}
    \ell(\widetilde{X})
    \leq
    \#\widetilde{\cV}_k
    =
    \#\cV_k
    \leq
    \#\cV-1
  \end{equation*}
  for each $\widetilde{X}\in \widetilde{\cX}_k$.  If $X\in \cX_k$, then  $g_k^{-1}(X)\in
  \widetilde{\cX}_k$ and so
  \begin{equation*}
    \ell(X) = \ell(g_k^{-1}(X)) +1 \leq \#\cV.
  \end{equation*}
 Since each $X\in \cX$ is contained in one  of the sets   $\cX_k$ with  $k\in \{1,2,3\}$, the inductive
  step follows. This completes  the proof.
\end{proof}

In our construction in Section~\ref{sec:decomp}, the $2$-tiles
containing one of the  leaves $-1$  or $1$ of $\cT$ will play a special
role. We record the following statement related to this. 

%\begin{figure}[h]
%% \vspace{-1cm}
% \begin{overpic}[ scale=0.7
% , clip=true, trim=0mm 15mm 0mm 0mm
%  %,width=10cm, tics=10,
%    %, grid
%    ]{2tiles.png}
% %   \put(49.5,56){ $i$}
%    \put(12,15.5){ $-1$}
%     \put(85,15.5){ $1$}
%          \put(12, 23){ $\cT_{11}$}
%               \put(22, 33){ $\cT_{13}$}
%   %        \put(76,6){ $\cT_2$}
%               \put(82, 23){ $\cT_{22}$}
%           %   \put(60,46){ $\cT_3$}
% %         \put(49.5, 15.5){ $0$}
%%             \put(32, 23.5){ $\cT_1$}
%%                \put(67.5, 31.5){ $\cT_2$}
%%                  \put(45.5, 45){ $\cT_3$}
%   \put(29.5, 13.5){ $-\tfrac12$}
%     \put(66.5, 23){ $\tfrac12$}
%                          \end{overpic}
% % \vspace{-1cm}
%\caption{The CSST $\cT$ and some of its $2$-tiles.}
%\label{f:2tiles}
%\end{figure}

\begin{lemma}
  \label{lem:2tileCSST}
  The tile $X=\cT_{11}$ is the only tile $X\in \cX^*$ with
  $-1\in X$ and $-1/2\in \partial X$.  The tile $X=\cT_{22}$ is
  the only tile $X\in \cX^*$ with $1\in X$ and $1/2\in \partial X$.
  % The $2$-tile $\cT_{11}$ is the only $2$-tile containing $-1$
  % and $\cT_{22}$ is the only $2$-tile containing $1$. 
\end{lemma}

See Figure~\ref{fig:2tiles} for an illustration.  

\begin{proof} We will only prove the first statement. The second statement is proved along very similar lines and we omit the details. 
  
 First we  consider $\cT_{11}$. Note  that
  $-1=g_{11}(-1)\in g_{11}(\cT)=\cT_{11}$. Since $-1$ is a leaf and not a branch point of $\cT$,  it
  is not a boundary point of $\cT_{11}$, because by  Lemma~\ref{lem:bdry}~\ref{item:bdry3} and Lemma~\ref{lem:branch}~\ref{item:branch1} 
the boundary point of any tile is a branch point of $\cT$. 
So
  $g_{11}(-1)=-1\notin \partial \cT_{11}$. Hence
  $-1/2=g_{11}(1)\in\partial \cT_{11}$, as follows from
  Lemma~\ref{lem:bdry}~\ref{item:bdry1}. Thus $\cT_{11}$ is a
  tile with $-1\in \cT_{11}$ and $-1/2\in \partial \cT_{11}$.
 
  Now suppose $X\in \cX^*$ is any tile with $-1\in X$ and $-1/2\in \partial X$, and let $n\in \N_0$ be the level of $X$. 
  Then $n\ne 0,1$, 
  because  the   $0$-tile $\cT$ has empty boundary and the $1$-tiles $\cT_1, \cT_2, \cT_3$ do not
  contain $-1/2$ in their boundaries (see \eqref{eq:bdry_T1}).

  We have $\diam(X)=2^{-n+1}\ge |-1/2-(-1)|=1/2$, which rules out $n\ge 3$. It follows that $n=2$, and so $X$ must be a $2$-tile,
  say $X=\cT_{k\ell}$ with $k,\ell\in \{1,2,3\}$.
  Then $-1\in X=\cT_{k\ell}\sub \cT_k$.  Since 
$\cT_2$ and $\cT_3$ do not contain $-1$ (see
  Lemma~\ref{lem:1tilesCSST}~\ref{item:1tilesCSST1}), we have $k=1$, and so $X$ must be one of the $2$-tiles
  $\cT_{11},\cT_{12}, \cT_{13}$. 
  
  The homeomorphism $g_1^{-1}\colon \cT_1\to \cT$ sends these
  $2$-tiles that are all contained in $\cT_1$ to
  $\cT_1,\cT_2,\cT_3$, respectively. Since
  $-1=g_1^{-1}(-1)\in g_1^{-1}(X)$ and 
%  $-1\in \cT_1\setminus (\cT_2\cup\cT_3)$,
  $-1\notin \cT_2, \cT_3$,
  we must have
  $g_1^{-1}(X)=\cT_1$ and so $X=\cT_{11}$. Hence $X=\cT_{11}$ is
  indeed the only tile $X\in \cX^*$ with $-1\in X$ and
  $-1/2\in \partial X$.
\end{proof}

\section{Subdivisions of quasiconformal trees}
\label{sec:subdivisions-trees}

The results in the previous two sections imply  that a metric space $T$ that is
quasisymmetrically equivalent to the CSST must be  a uniformly
branching trivalent quasiconformal tree. This is the ``only
if'' implication in Theorem~\ref{thm:CSST_qs}. In this section  we start with
the proof of the ``if'' implication.

In the following, we assume that $T$ is a uniformly
branching trivalent quasiconformal tree. We need to show
that such a tree $T$  is quasisymmetrically equivalent to $\cT$. By rescaling the metric on $T$ if necessary, we may assume that $\diam(T)=1$ for convenience. 

Recall from Section~\ref{sec:topology-trees} that  a finite set
$\V\subset T$ that does not  contain any leaf of $T$  decomposes
$T$ into a set of tiles $\X$. These tiles are subtrees of $T$. We
want to map them to tiles of $\cT$, i.e., elements of $\cX^*$
(see \eqref{eq:defcX}). By 
Lemma~\ref{lem:bdry}~\ref{item:bdry1} each  tile in $\cT$ distinct from $\cT$ itself has
one or two boundary points. For this reason, we are
interested in decompositions such that every tile $X\in \X$ has one or two 
boundary points. Accordingly,  we call  $\X$ an \emph{edge-like
  decomposition} of $T$ if $\V=\emptyset$ and $\X=\{T\}$ (as a degenerate case), or  if $\V\ne \emptyset$ and $\#\partial 
  X\le 2$ for each $X\in \X$. Note that in the latter case 
  $1\le \# \partial X\le 2$ by Lemma~\ref{lem:vt}~\ref{item:vt3}. We say that a tile $X\in \X$ in an edge-like decomposition $\X$ of $T$ is  a {\em
  leaf-tile} if $\#\partial X=1$ and an {\em edge-tile} if
$\#\partial X=2$. 

The reason for this terminology is that in an 
edge-like
  decomposition $\X$ of $T$ the tiles in $\X$ satisfy  the same incidence relations as the edges in a finite simplicial tree. Here the 
  leaf-tiles in $\X$ correspond to edges of the simplicial tree that contain a leaf  and the edge-tiles in $\X$ to  edges  that do not contain a leaf of the simplicial tree. 

We now set  $\V^0=\emptyset$,  fix
$\delta \in (0,1)$, and  for $n\in \N$  define
\begin{equation}
  \label{eq:defvn}
  \V^n=\{ p\in T:
  \text{ $p$ is a branch point of $T$ with } H_T(p)\ge\delta^n\}.
\end{equation}
Recall that the height $H_T$ was defined in \eqref{eq:HT}. 
Clearly, $\{\V^n\}$  is an increasing sequence as in \eqref{eq:Vn_incrs} and none of the sets $\V^n$ contains a leaf of 
$T$.   We will momentarily see that each set $\V^n$ is finite.  We denote by
$\X^n$ the  set of
tiles in  the decomposition of $T$ induced by $\V^n$. Then the sequence $\{\X^n\}$ forms a subdivision of $T$ in the
sense of Definition~\ref{def:subdiv}
(see the discussion after \eqref{eq:Vn_incrs}). 

As we will see, this subdivision $\{\X^n\}$ has some good  properties, in
particular it  is a quasi-visual subdivision of $T$ (see Definition~\ref{def:qvsub}). By choosing $\delta\in (0,1)$
sufficiently small, we  can also  ensure that the points in $\V^{n+1}$
separate the points in $\V^n$ in a suitable way. We summarize this in the following statement.

%Recall from Lemma~\ref{lem:top_T}~\ref{item:top_T2} that each
%tile $X\in \X^n$ is a tree. We often consider the decomposition
%that $\V^{n+1}$ induces on $X$. To this end we define
%$\V_X\coloneqq \V^{n+1} \cap \inte(X)$. Note that a set $X'$
%belongs to the decomposition induced by $\V_X$ on $X$ if and only
%if $X'\in \X^{n+1}$ and $X'\subset X$.

\begin{proposition}
  \label{prop:decomp}
  Let $T$ be a uniformly branching trivalent quasiconformal
  tree with $\diam(T)=1$. Let $\delta\in (0,1)$, and $\{\V^n\}_{n\in\N_0}$ be as in
  \eqref{eq:defvn}. Then the following statements are true:
  \begin{enumerate}
   \item
   \label{item:decomp0}
  $\V^n$ is a finite set for each $n\in \N_0$. 
\item
   \label{item:decomp5}
   $\{\X^n\}_{n\in \N_0}$ is a quasi-visual subdivision of $T$.
   % 
   % save counter of enumeration
   \setcounter{mylistnum}{\value{enumi}}
 \end{enumerate}
 Let $n\in \N_0$, $X\in \X^n$, and $\V_X\coloneqq \V^{n+1} \cap
 \inte(X)$. 
% The decomposition of $X$ induced by $\V_X$ is denoted by $\X_X$.
 Then we have: 
 \begin{enumerate}
   % reset counter
   \setcounter{enumi}{\value{mylistnum}}
   
    \item 
   \label{item:decomp1}
   $\#\partial X\leq 2$, and if we denote by $\X_X$ the decomposition of $X$ induced by
   $\V_X$, then  $\X_X$ is edge-like.

 % \item
 %   \label{item:decomp2}
 %   $\diam(X)\asymp \delta^n$
 %   with $C(\asymp)$ independent of $n$, $X$, and $\delta$.

  % i.e.,
  % $\#\partial X\le 2$ for each $X\in \X^n$.

 \item 
    \label{item:decomp3}
    There exists $N=N(\delta)\in \N$  depending on $\delta$, but independent of $n$ and $X$, such that
    $\#\V_X\le N$.
   % 
   % save counter of enumeration
   \setcounter{mylistnum}{\value{enumi}}
 \end{enumerate}
 If $\delta\in (0,1)$ is sufficiently small (independent of $n$ and $X$), then we also have:
 \begin{enumerate}
   % reset counter
   \setcounter{enumi}{\value{mylistnum}}
 \item
   \label{item:decomp_Vfinite}
   $\#\V_X\geq 2$. 

 \item 
   \label{item:decomp4}
   If $\#\partial X=2 $ and
   $\partial X=\{u,v\}\sub T$, then $(u,v)\cap \V_X$ contains at
   least three elements.
   
 \end{enumerate} 
\end{proposition}

We will see in the proof that $\{\X^n\}$ is in fact a
quasi-visual subdivision that satisfies the (stronger)
conditions of Lemma~\ref{lem:visual}. 
The first part of  \ref{item:decomp1} implies  that $\X^n$ is
an edge-like decomposition of the whole tree $T$ for each $n\in 
\N_0$.
It follows from  \ref{item:decomp_Vfinite} that if 
$\delta\in (0,1)$ is small enough, then we have $\V^n\neq
\emptyset$ for each $n\in \N$.

% \begin{proposition}
%   \label{prop:decomp}
%   Let $T$ be a uniformly branching trivalent quasiconformal
%   tree. Let $\delta\in (0,1)$ and $\{\V^n\}$ be as in
%   \eqref{eq:defvn}. Then
%   the following statements are true:
% \begin{enumerate}
% \item
%   \label{item:decomp1}
%   $\V^n$ is a finite set and the
%   decomposition $\X^n$ induced by $\V^n$ is edge-like for each
%   $n\in \N_0$. Furthermore, given any $X\in \X^n$, the
%   decomposition of $X$ induced by $\V_X$ is edge-like.  
%   % i.e.,
%   % $\#\partial X\le 2$ for each $X\in \X^n$.

% \item
%   \label{item:decomp2}
%   $\diam(X)\asymp \delta^n$ for each $n\in \N$ and $X\in \X^n$
%   with $C(\asymp)$ independent of $n$, $X$, and $\delta$.

% \item 
%    \label{item:decomp5}
%    $\{\X^n\}_{n\in \N_0}$ is a quasi-visual subdivision of $T$.

%  \item 
%     \label{item:decomp3}
%     There exists $N=N(\delta)\in \N$ such that
%     $2\le \#(\V^{n+1}\cap \inte(X))\le N$ for each $n\in \N_0$
%     and $X\in \X^n$.
   
%  \item 
%    \label{item:decomp4}
%    If $n\in \N$, $X\in \X^n$ is an edge-tile, and $\partial X=\{p,q\}$,
%    then $(p,q)\cap \V^{n+1}$ contains at least three elements.  
   
%  \item
%    For each $n\in \N$ we have $\V^n\ne \emptyset$, and
%  \end{enumerate} 
% \end{proposition}

\begin{proof} \ref{item:decomp0} This is clear for $n=0$, since $\V^0=\emptyset$.  

Let $n\in \N$, and $u,v\in \V^n$ be distinct points.  Since
  $T$ has uniformly relatively separated branch points, by definition of $\V^n$ we have
  \begin{equation}
    \label{eq:vw_ge_dn}
    \abs{u-v}\gtrsim \min\{H_T(u), H_T(v)\} \geq \delta^n,
  \end{equation}
  where $C(\gtrsim)$ is the constant in 
  Definition~\ref{def:unif_sepa} for the tree $T$. Since $T$ is compact,  \eqref{eq:vw_ge_dn} implies that $\V^n$ is finite. 

\smallskip
We now first prove  \ref{item:decomp1} %and  \ref{item:decomp2}
before we establish  \ref{item:decomp5}. To this end, consider $n\in
\N_0$, $X\in \X^n$, $\V_X=\V^{n+1}\cap \inte(X)$, and, as in the statement of  \ref{item:decomp1},  let $\X_X$
be the decomposition of $X$ induced by $\V_X$.

  \smallskip
  \ref{item:decomp1}
 To show 
  that $\#\partial X\le 2$, we argue by contradiction and assume
  that there are three distinct points $x,y,z\in\partial X\sub X$. Then 
  $x,y,z\in\partial X\sub \V^n$ by Lemma~\ref{lem:vt}~\ref{item:vt2}. We now consider the center  of $x,y,z$ (see Lemma~\ref{lem:center}~\ref{item:center1}), 
  namely the point $c\in T$ with
  \begin{equation*}
    \{c\}= [x,y]\cap [y,z]\cap[z,x].  
  \end{equation*}
  Since $X$ is a
  subtree of $T$, we have $c\in X$.  By Lemma~\ref{lem:center}~\ref{item:center2} this is a branch point of $T$
  with
  \begin{equation*}
    H_T(c)\ge \min \{ H_T(x), H_T(y), H_T(z)\} \ge \delta^n.
  \end{equation*}
   So $c$ belongs to $\V^n$. At least two of the three points $x$,
  $y$, $z$ are distinct from $c$, say $x,y\ne c$. Then
  $c\in (x,y)$ and so $x$ and $y$ cannot lie in the same tile
  $X$ as follows from Lemma~\ref{lem:top_T}~\ref{item:top_T1} and Lemma~\ref{lem:vt}~\ref{item:vt1}. This is a contradiction, showing that indeed
  $\#\partial X\le 2$.

  % Now consider a tile $Y$ from the decomposition of $X$ induced
  % by $\V_X$.
  Now consider a tile $Y\in \X_X$.
  Then $Y\sub X$, and $Y\in \X^{n+1}$ as follows from
  Lemma~\ref{lem:vtt}~\ref{item:vtt3}.  If $\partial_X Y$ denotes
  the relative boundary of $Y$ as a subset of $X$, then
  $\partial_X Y\sub \partial Y$ (see
  Lemma~\ref{lem:vtt}~\ref{item:vtt4}); the latter inclusion is
  actually true in every topological space).  Since
  $Y\in \X^{n+1}$ we have $\#\partial Y\leq 2$ by the previous
  considerations (applied to $Y$), and so
  $\#\partial_{X} Y \leq 2$. This means that the decomposition $\X_X$ of
  $X$ induced by $\V_X$ is edge-like, and \ref{item:decomp1}
  follows.

  \smallskip
  \ref{item:decomp5}
  We now show that $\{\X^n\}$ is a quasi-visual subdivision of
  $T$. As we have already discussed in the beginning of this
  section, $\{\X^n\}$ is a subdivision of $T$ in the sense of
  Definition~\ref{def:subdiv}. It remains to show that it is a
  quasi-visual approximation (as in
  Definition~\ref{def:qv_approx}). For this we will verify
  conditions~\ref{item:visual1} and~\ref{item:visual2} in
  Lemma~\ref{lem:visual}. Fix some $n\in \N_0$. 

  \smallskip
  \emph{Claim 1.}
  $\diam(X)\asymp \delta^n$ for all $X\in \X^n$, where
  $C(\asymp)$ is independent of $n$, $X$, and $\delta$.
  \smallskip

  Let $X\in \X^n$ be arbitrary. 
  If $n=0$, then $X=T$ and $\diam(T)=1= \delta^0$. So the
  statement is obviously true in this case.

 Now assume $n\geq 1$, and   let $u,v\in X$ be arbitrary 
  distinct points. Then $(u,v)\cap\V^n=\emptyset$; 
otherwise,     $u$ and $v$
  could not be contained in the same tile $X$ as follows from Lemma~\ref{lem:vt}~\ref{item:vt1}.  Since $T$ has
  uniformly relatively dense branch points, there is a branch point $p$ of $T$ with
  $p\in (u,v)$ and 
  \begin{equation*}
    H_T(p) \gtrsim \abs{u-v},  
  \end{equation*}
  where $C(\gtrsim)$ only 
 depends on the  constant  in Definition~\ref{eq:unif_sepa} for $T$
 (to see this, choose a branch point $p\in [u',v']\sub (u,v)$ with 
 $u',v'\in (u,v)$ very close to $u,v$, respectively).  
  Then $p\notin \V^n$, and so  
  $$\delta^n > H_T(p)\gtrsim  \abs{u-v}. $$ It follows that
  $\diam(X) \lesssim \delta^n$ with $C(\lesssim)$ independent of $n$, $X$, and $\delta$. 

 In order to show an inequality in the opposite direction, recall 
   that $\#\partial X\leq 2$ by \ref{item:decomp1} and that   $\partial X\sub \V^n$ by Lemma~\ref{lem:vt}~\ref{item:vt2}.

  If $\partial X=\emptyset$, then $X=T$ and so  
  $\diam(X)= 1\geq \delta^n$.
  
Assume $\partial X\ne \emptyset$. 
If $\partial X$  consists of one point $p\in T$, then $p\in \V^n$ and so $p$ is 
  a branch point of $T$ with $H_T(p)\ge \delta^n$. In this case,
  $X$ is a branch of $p$ in $T$ (see Lemma~\ref{lem:vt}~\ref{item:vt7}), and so
  \begin{equation*}
    \diam(X)
    \ge
    H_T(p)\ge \delta^n.  
  \end{equation*}
 
 If $\partial X$ consists of two distinct
  points $u,v\in \V^n$, then by  \eqref{eq:vw_ge_dn} we have
  \begin{equation*}
    \diam (X)
    \ge
    \abs{u-v}
    \gtrsim
    \delta^n,  
  \end{equation*}
  where $C(\gtrsim)$ is the constant from
  Definition~\ref{def:unif_sepa}. 

   Claim~1 follows,  and so  condition
  \ref{item:visual1} in Lemma~\ref{lem:visual} is true.

  % These considerations show that \ref{item:decomp2} is true with
  % a constant $C(\asymp)$ independent of $n$, $X$, and $\delta$.

  % The first condition~\ref{item:visual1}
  % in this lemma is obviously true by \ref{item:decomp2}.

  \smallskip
  \emph{Claim~2.} $\dist(X,Y) \gtrsim \delta^n$ for all $X,Y\in
  \X^n$   with $X\cap Y= \emptyset$, where $C(\gtrsim)$ is
  independent of $n$, $X$, $Y$, and $\delta$.
  \smallskip
  
  To see this second claim, fix arbitrary tiles $X,Y\in \X^n$
  with $X\cap Y=\emptyset$. We can find points $x\in X$ and
  $y\in Y$ with $\abs{x-y}=\dist(X, Y).$ Now consider the unique
  arc $[x,y]$ joining $x$ and $y$ in $T$. As we travel from
  $x\in X$ to $y\not\in X$ along $[x,y]$, there is a last point
  $u\in [x,y]$ with $u\in X$. Then $(u,y]$ is non-empty and
  disjoint from $X$ which implies that
  $u\in \partial X\sub \V^n$.  As we travel from $u\not\in Y$ to
  $y\in Y$ along $[u,y]\sub [x,y]$, there is a first point
  $v\in [u,y]$ with $v\in Y$. Then $v\in \partial Y\sub \V^n$ by
  a similar reasoning.
  
 The points $u,v\in \V^n$ are distinct, because  $u\in X$, $v\in Y$,  and  $X\cap Y=\emptyset$. It follows  that 
  $\abs{u-v}\gtrsim \delta^n$ by \eqref{eq:vw_ge_dn}.   Since $T$ is a quasiconformal tree and hence of bounded turning, we
  have $\diam [x,y]\asymp |x-y|$. This implies that 
  \begin{equation*}
    \dist(X,Y)
    =
    \abs{x-y}
    \asymp
    \diam [x,y]
   \ge
    \abs{u-v}
    \gtrsim
    \delta^n.
  \end{equation*}
  Here all implicit multiplicative constants are independent of
  $n$, $X$, $Y$, and $\delta$. We have shown Claim~2, which is
  condition \ref{item:visual2} in Lemma~\ref{lem:visual}. 
  
  \smallskip
Lemma~\ref{lem:visual} now implies that $\{\X^n\}$ is indeed a
  quasi-visual subdivision of $T$, and \ref{item:decomp5}
  follows.   

  \smallskip
  \ref{item:decomp3}
   By \eqref{eq:vw_ge_dn} two distinct points $u,v\in \V_X\subset \V^{n+1}$ have separation  
  $\abs{u-v}\gtrsim \delta^{n+1}$. On the
  other hand, $\V_X$ is contained in $X$ with
  $\diam(X) \asymp \delta^n$, as we saw in Claim~1. Since $T$ is
  doubling, it follows 
  that there is a constant $N\in \N$ with
  \begin{equation*}
     \#\V_X\le N.
  \end{equation*}  
  Here $N$ depends on $\delta$, the doubling constant of $T$, and
  the constants in \eqref{eq:vw_ge_dn} and Claim~1 (which depend
  on the constants in 
  Definition~\ref{def:unif_sepa} and
  Definition~\ref{def:unif_dense}), but not on $n$ and $X$.
  \smallskip

  Having verified \ref{item:decomp0}--\ref{item:decomp3} for any
  $\delta\in (0,1)$, we now prove the remaining statements that
  require us to choose $\delta$ sufficiently small.

  \smallskip
  \ref{item:decomp_Vfinite}
  By Claim~1 %\ref{item:decomp2}
  we have  $\diam(X)\asymp\delta^n$
   and $\diam(Y)\asymp \delta^{n+1}$ for
  each $Y\in \X^{n+1}$. Here the constant $C(\asymp)$ depends only on the
  constants in Definition~\ref{def:unif_sepa} and
  Definition~\ref{def:unif_dense}, as we have seen. This implies 
  that there exists a constant $\delta_1\in (0,1)$ independent of $n$ and $X$ with the following property:  if  $0<\delta\le \delta_1$ (as we now assume), 
  then 
  \begin{equation}
    \label{eq:XYbdd}
    \diam(Y)\le \tfrac 14 \diam(X)
  \end{equation} 
  for all $Y\in \X^{n+1}$. 

  It follows from Lemma~\ref{lem:vtt}~\ref{item:vtt3} that the
  decomposition $\X_X$ of $X$ induced by $\V_X$ is given by all
  tiles $Y\in \X^{n+1}$ with $Y\sub X$. Then clearly
  $\V_X\ne \emptyset$, because otherwise
  $\X_X=\{X\}\sub \X^{n+1}$ which is impossible by
  \eqref{eq:XYbdd}.
  
  To show that $\#\V_X \geq 2$, assume on the contrary that
  $\V_X=\V^{n+1}\cap \inte(X)$ contains only one point $p$. This
  is a branch point of $T$ and hence a branch point of $X$,
  because $p\in \inte(X)$ (see
  Lemma~\ref{lem:subt}~\ref{item:subt0}). Then the decomposition
  $\X_X$ of $X$ induced by $\V_X$ consists precisely of the three
  branches $Y_1, Y_2, Y_3$ of $p$ in $X$. Moreover, these
  branches are tiles in $\X^{n+1}$ and share the common point
  $p$.  Then
  $$ \diam(X) \le 2 \max\{\diam(Y_k): \, k\in \{1,2,3\}\},$$ 
but this  is impossible by    \eqref{eq:XYbdd}. 

 We conclude that if $0<\delta\le \delta_1$, then necessarily 
  $\#\V_X \geq 2$, and \ref{item:decomp_Vfinite} follows.
  
  \smallskip
  \ref{item:decomp4}
  Now assume that $\#\partial X=2$ and $\partial X= \{u,v\}$.
  Then $u$ and $v$  are distinct points in $\partial X \sub \V^n$, and so
  $\abs{u-v}\gtrsim \delta^n$ by \eqref{eq:vw_ge_dn}. On the
  other hand, the points in $\V^{n+1}\cap[u,v]$ divide $[u,v]\sub X$
  into non-overlapping arcs. Each of these arcs lies in a tile in
  $\X^{n+1}$, and so has diameter $\lesssim \delta^{n+1}$ by
  % \ref{item:decomp2}.
  Claim~1. 
  If there are $k\in \N$ of these arcs it
  follows that
  \begin{equation*}
    \delta^n
    \lesssim
    \abs{u-v}
    \lesssim k\delta^{n+1}, 
  \end{equation*}
  where the constants $C(\lesssim)$ do not depend on $n$, $X$, 
 or  $\delta$.  If $\delta$ is small enough, say 
  $0<\delta\le \delta_2$, where $\delta_2\in (0,1)$ can be chosen independently of
   $n$ and $X$,
  then necessarily $k\ge 4$, and so $\V^{n+1}\cap (u,v)$ contains
  at least three elements. Note that $(u,v)\subset X\setminus \partial X=\inte(X)$.
  Thus $\V_X\cap (u,v) = \V^{n+1} \cap
  (u,v)$ contains at least three elements,
  and \ref{item:decomp4} follows.

  \smallskip
  If we choose $0<\delta\leq \min\{\delta_1,\delta_2\}$, where
  $\delta_1,\delta_2\in (0,1)$ are the constants from
  \ref{item:decomp_Vfinite} and  
  \ref{item:decomp4}, respectively, then   \ref{item:decomp_Vfinite} and  
  \ref{item:decomp4} are both true. This finishes the proof. 
\end{proof}

\section{Quasi-visual subdivisions of the CSST}
\label{sec:decomp}

% *** rewrite ***

% *** terminology: ``edge-like'' decomposition, instead of
% ``good''. 

% I use here ``edge-'' and ``leaf-tile'' instead of ``arc-'' and
% ``end-tile''

% Instead of ``$F\colon T\to \cT$ an admissible homeomorphism (for
% $\V$)'', I use ``$F\colon T\to \cT$ a tile-homeomorphism for
% $\X$''. 

% The points $p,q$ in Lemma~\ref{lem:decomp} did not have a special
% name. I call them \emph{marked leaves of $T$}.

% ***
Let $T$ be a uniformly
branching trivalent quasiconformal tree. 
Then it is possible to show that for every 
quasi-visual subdivision of $T$ there is an isomorphic quasi-visual subdivision of  the CSST $\cT$. From this one can deduce the quasisymmetric equivalence of $T$ and $\cT$  by  Proposition~\ref{prop:qv_f_qs}.

We will actually prove a slightly less general statement, namely we
restrict ourselves to a quasi-visual  subdivision $\{\X^n\}$ of $T$ as in
Proposition~\ref{prop:decomp}.  
This is formulated in  the following proposition, which  is the  main result of this section. 

\begin{proposition}
  \label{prop:exqs}
  Let $T$ be a uniformly branching trivalent quasiconformal tree with $\diam(T)=1$. 
 Suppose  $\V^n$ for $n\in \N_0$  is  as in \eqref{eq:defvn} 
with $\delta\in (0,1)$  so small that 
all the statements  
in Proposition~\ref{prop:decomp} are true for the decompositions
$\X^n$  of $T$ induced by the sets $\V^n$. Then there exists a 
quasi-visual subdivision $\{\Y^n\}_{n\in \N_0}$ of $\cT$ that is 
isomorphic to the quasi-visual subdivision of
  $\{\X^n\}_{n\in \N_0}$ of $T$.
\end{proposition}
% Recall that the notion of isomorphic subdivisions was introduced in the discussion before  Proposition~\ref{prop:qv_f_qs}.
Recall that the notion of isomorphic subdivisions was based on \eqref{eq:isosubdiv_cap} and \eqref{eq:isosubdiv_incl}.
By applying  Proposition~\ref{prop:qv_f_qs}, we see that there exists a quasisymmetric homeomorphism $F\: T\ra \cT$ that   induces the  isomorphism between the 
quasi-visual subdivision $\{\X^n\}$ of $T$ and  
the 
quasi-visual subdivision $\{\Y^n\}$ of $\cT$. In particular, 
 $T$ and $\cT$ are quasisymmetrically equivalent.  This gives the ``if'' implication in  Theorem~\ref{thm:CSST_qs}.

The proof of Proposition~\ref{prop:exqs} requires some preparation. 
We first study homeomorphisms that send tiles in a decomposition of $T$ to tiles in $\cT$.

%\smallskip

\subsection*{Decompositions of trees and tile-homeomorphisms}
\label{sec:finite-decomp-ct}

By Proposition~\ref{prop:decomp}~\ref{item:decomp1}  the decompositions $\X^n$  of  $T$ as in 
Proposition~\ref{prop:exqs} are edge-like in the sense that for each tile $X\in \X^n$ its boundary is the empty set  
(only if $X=T$), or contains one point (then $X$ is a leaf-tile), or two points (then  $X$ is an
edge-tile). The points in $\partial X$ play a special role. They are leaves of $X$ and the only points   where $X$ intersects other tiles of the same level (see
 Lemma~\ref{lem:vt}~\ref{item:vt3}~and~\ref{item:vt5}). 
Accordingly, we
say that the points in $\partial X$ are \emph{marked leaves} of
$X$.

% The  tiles of a higher level decompose $X$. We want to find an isomorphic decomposition of $\cT$ into tiles in $\cX^*$ with similar locations of the marked points. Here we consider 
% $-1$ or $1$ as a marked leaf of $\cT$ if $\#\partial X=1$, or 
% both  $-1$ and  $1$ as a marked leaves of $\cT$ if 
% $\#\partial X=2$.

Recall that $X\in \X^n$ (viewed as a tree) has an edge-like
decomposition $\X_X$ induced by $\V_X=\V^{n+1}\cap \inte(X)$  
(see
Proposition~\ref{prop:decomp}~\ref{item:decomp1}). We want to
find a decomposition of $\cT$ into tiles in
$\cX^*$ that is isomorphic to $\X_X$ and respects the marked leaves. To this end, we
consider $-1$ or $1$ as a marked leaf of $\cT$ if
$\#\partial X=1$, or both $-1$ and $1$ as  marked leaves of
$\cT$ if $\#\partial X=2$.

To formulate a corresponding statement, we consider an arbitrary  trivalent tree $T$ (instead of $X$) with a dense set of branch points.  The following 
theorem allows us to construct homeomorphisms between $T$ and $\cT$ that map leaves in a
prescribed way.

\begin{theorem}
  \label{thm:homTpqr}
  Let $T$ be a trivalent metric tree whose branch points are
  dense in $T$. Then $T$ is homeomorphic to $\cT$.
 Moreover, 
  if  $p$, $q$, $r$ are three  distinct leaves of
  $T$, then there exists a homeomorphism $F\: T\ra \cT$ such that
  $F(p)=-1$, $F(q)=1$, $F(r)=i$.
\end{theorem} 
The first part follows from \cite[Theorem 1.7]{BT} and the second part from \cite[Theorem~5.4]{BT}. Note that $-1$, $1$, $i$ are leaves of $\cT$ (see the discussion after  \cite[Theorem~5.4]{BT}; that $-1$ and $1$ are leaves of $\cT$ was already pointed out in  Lemma~\ref{lem:1tilesCSST}~\ref{item:1tilesCSST2}). 
 If $T$ is as in Theorem~\ref{thm:homTpqr}, then  $T$ is homeomorphic to $\cT$, and so $T$ has at least three leaves (actually infinitely many). The second part of the theorem says  that any three   leaves of $T$ can be sent to the prescribed leaves 
$-1$, $1$, $i$ of the CSST by a homeomorphism between $T$ and $\cT$. We will apply
this in a weaker form where we send  at most two leaves 
of $T$ to $-1$ and $1$.

Now  let $\X$ be an edge-like decomposition of $T$ induced by a
finite set of branch points $\V$ of $T$.  We say that a
homeomorphism $F\: T\ra \cT$ is a {\em tile-homeomorphism} (for
$\X$) if $F(X)\in \cX^*$ (see
\eqref{eq:defcX}) for each $X\in \X$. Then  $F$ maps
tiles in $T$ to tiles in $\cT$, and so the level $\ell(F(X))$ (of $F(X)$ as a tile
of $\cT$)  is defined for each $X\in \X$.  The following statement is a basic existence
result for tile-homeomorphisms.

\begin{lemma}
  \label{lem:decomp}
  Let $T$ be a trivalent metric tree whose branch points are
  dense in $T$. 
  Suppose $\V\sub T$ is a finite set of branch points of $T$ that
  induces an edge-like decomposition of $T$ into the set of
  tiles $\X$. 

  We assume that  either $T$  has no marked leaves, one marked leaf $p\in T$, 
 or two marked leaves  $p,q\in T$, $p\ne q$. If  $\V\ne \emptyset$,
 we also assume that the marked leaf $p$ (when present)  lies in a leaf-tile $P\in \X$, and the other marked leaf $q$ (when present) lies  in a leaf-tile $ Q\in \X$ distinct from $P$. 
  
    Then there exists  a tile-homeomorphism $F\: T\ra \cT$ for $\X$ such that the following statements are true: 
   \begin{enumerate}
  \item 
    \label{item:T_fin_decomp1}
     $F(p) = -1$, or alternatively, $F(p)=1$, when $T$  has  one  marked leaf $p$; or
           $F(p)=-1$ and  $F(q)=1$,
when $T$ has two marked leaves $p$ and $q$.

     \item 
    \label{item:T_fin_decomp2}
    If $T$ has one marked leaf $p$ and $\#\V\geq  2$, 
    then we may also  assume that $F$ satisfies $\ell(F(P))=2$.
    
      \item 
    \label{item:T_fin_decomp3}
    If $T$ has two  marked leaves $p,q\in T$ and $[p,q]\cap \V$
    contains at least three  points, then we may also  assume that
    $F$ satisfies $\ell(F(P))=\ell(F(Q)=2$.

        \item
    \label{item:T_fin_decomp4}
    For each  $X\in \X$ we have $ \ell(F(X)) \leq \#\V$,       
   and,  if $\V\neq \emptyset$, also   
$\ell(F(X)) \geq 1$.
  \end{enumerate}
\end{lemma}

% Note that the unique $2$-tiles in $\cT$ containing one of the
% leaves $-1$ and $1$, are $\cT_{11}$ respectively $\cT_{22}$, see
% Lemma~\ref{lem:2tileCSST}. Their diameter is $1/2$ and we have
% $-1,-1/2\in \cT_{11}$ as well as $1/2,1\in \cT_{22}$.

If we are in the situation of \ref{item:T_fin_decomp2} or
\ref{item:T_fin_decomp3}, the homeomorphism $F$ sends the
leaf-tiles containing the marked leaves of $T$ to $2$-tiles in
$\cT$.  Actually, in \ref{item:T_fin_decomp2} we have
$F(P)=\cT_{11}$ or $F(P)=\cT_{22}$ depending on whether $F(p)=-1$
or $F(p)=1$, and in \ref{item:T_fin_decomp3} we have
$F(P)=\cT_{11}$ and $F(Q)=\cT_{22}$.  It is important that these
images of $P$ and $Q$ have fixed diameter $1/2$.  This (seemingly technical)
condition will be crucial to ensure that the subdivision of $\cT$
given in Proposition~\ref{prop:exqs} is quasi-visual. In
particular, we use it to show that neighboring tiles have
comparable diameter, as required by
Definition~\ref{def:qv_approx}~\ref{item:qv_approx1}. This is the
reason why we constructed decompositions of $T$ satisfying the
corresponding conditions \ref{item:decomp_Vfinite} and
\ref{item:decomp4} in Proposition~\ref{prop:decomp}.

The tile-homeomorphism  $F\colon T\to \cT$ in Lemma~\ref{lem:decomp} is merely a convenient device that allows us to transfer  information on leaves and tiles  from $T$ to $\cT$.
The main   point is
that we can find a decomposition of 
 $\cT$ into tiles (in $\cX^*$) 
isomorphic to the  given edge-like decomposition of $T$.

\begin{proof}
  [Proof of Lemma~\ref{lem:decomp}]
  \ref{item:T_fin_decomp1} Depending on the number of marked leaves of $T$ and the desired normalization for $F$, we distinguish several cases in the ensuing discussion:
  
\smallskip
{\em Case 1:} $T$ has no marked leaf.
 
\smallskip
{\em Case 2a:} $T$ has one  marked leaf $p$, and the desired normalization is $F(p)=-1$.

\smallskip
{\em Case 2b:} $T$ has one  marked leaf $p$, and the desired  normalization is $F(p)=1$.

 \smallskip
{\em Case 3:} $T$ has two  marked leaves $p$ and $q$,  and the desired  normalization is $F(p)=-1$ and $F(q)=1$.

\smallskip

We now prove the statement  by induction on 
$\#\V\ge 0$. In the proof, Case 1 is the easiest  to handle, since we do not have to worry about normalizations, but  will have to make some careful choices 
in Cases 2a, 2b, 3 due to the presence of marked leaves.

Suppose first that $\#\V=0$, and so $\V=\emptyset$. We invoke
Theorem~\ref{thm:homTpqr} to obtain a homeomorphism
$F\: T\ra \cT$, where we can impose the normalization 
 $F(p)=-1$ or  $F(p)=1$ when $T$ has one marked leaf $p$ (Cases 2a and 2b),
 or the normalization $F(p)=-1$ and $F(q)=1$ when $T$ has two  marked leaves $p$
 and $q$ (Case 3). 
 
  Then
$F$ is a tile-homeomorphism for $\X$, since there is only one
tile $X=T$; it is 
is mapped to $F(T)=\cT=g_\emptyset(\cT)$, which is a tile of
$\cT$.

For the inductive step suppose $\#\V\ge 1$ and that the statement
is true for all marked trees homeomorphic to $\cT$ with  decompositions induced by sets of branch points with
fewer than $\#\V$ elements.

% For the inductive step suppose $\#\V\ge 1$ and that the statement
% is true for all trivalent metric trees with a dense set of triple
% points and their simple decompositions induced by sets of branch
% points with fewer than $\#\V$ elements that are marked as in the
% statement. 

Since $\#\V\ge 1$, we have $\V\ne \emptyset$. We choose  a (branch)
point $x\in \V$. If $T$ has marked leaves, then they are distinct from the branch point $x$. Moreover,  if  $T$ is marked by two leaves  $p$ and $q$ (Case 3), we may 
assume in addition that they lie in distinct components of
$T\setminus \{x\}$. Indeed, we may then choose any point
$x\in (p,q)\cap \V$. The latter set  is non-empty as follows from the fact that 
  $p$ and $q$ are contained in  distinct leaf-tiles  $P,Q\in \X$
   by our assumptions.

Let $B_1$, $B_2$, $B_3$ be the branches of $x$ in $T$. Here we may
assume that $p\in B_1$ in Case 2a, $p\in B_2$ in Case 2b, and $p\in B_1$, $q\in B_2$ in Case 3. Note that these branches are 
precisely the tiles in the decomposition of $T$ induced by 
$\{x\}$. 

The simple idea for the proof is now to apply the induction
hypothesis to each of the trees $B_k$ with suitable markings
and copy the  image of the resulting homeomorphism into the
subtree $\cT_k$ by the map $g_k$ (as in \eqref{eq:def_fj}). If we assemble these maps into
one, we obtain the desired homeomorphism $F$.

First, note that Lemma~\ref{lem:vt}~\ref{item:vt4} implies that each branch $B_k$ is also a trivalent metric
tree with a dense set of branch points. The (possibly empty) set
\begin{equation*}
  \V_k
  \coloneqq \V\cap \inte(B_k)=
  \V\cap (B_k\setminus \{x\})
\end{equation*}
consists of branch points of $B_k$, and $\#\V_k < \#\V$ (since
$\V_k \subset \V\setminus\{x\}$). 

 The set $\V_k$ induces a
 decomposition of $B_k$. Its set of tiles $\X_k$
consists precisely of the tiles $X\in \X$ with $X\sub B_k$
(see Lemma~\ref{lem:vtt}~\ref{item:vtt3}).
% If $\partial_ {B_k} X$ denotes the relative boundary of $X\in
% \X_k$ in $B_k$, then by Lemma~\ref{lem:vtt}~\ref{item:vtt4} we have 
We denote 
the relative boundary of $X\in
\X_k$ in $B_k$ by $\partial_ {k} X$. Then, using Lemma~\ref{lem:vtt}~\ref{item:vtt4}, we have 
\begin{equation}\label{eq:relbdX}
 \partial_ {k} X=\partial X\setminus \{x\}, 
\end{equation}
 and in particular, 
$\# \partial_ {k} X\le \#\partial X\le 2$ for each $X\in \X_k$. This implies that 
$\X_k$ is an edge-like decomposition of $B_k$.

Since $x\in \partial B_k$, this point is  a leaf of $B_k$ by  Lemma~\ref{lem:vt}~\ref{item:vt3}. Moreover, each marked leaf of $T$ (when present) is a leaf of the unique branch $B_k$ that contains it, as follows from Lemma~\ref{lem:subt}~\ref{item:subt1}. We now 
mark: 
% $B_1$ by $x$ and, in addition, by $p$ in Cases 2a and 3;
%
%  $B_2$ by $x$ and, in addition, by $p$ in Case 2b, and by $q$ in 
% Case 3; 
%
% $B_3$ by $x$. 
\begin{align*}
  &\text{$B_1$ by $x$ and, in addition, by $p$ in Cases 2a and 3; }
\\
 &\text{$B_2$ by $x$ and, in addition, by $p$ in Case 2b, and by $q$ in 
Case 3;} 
\\
  &\text{$B_3$ by $x$.}
\end{align*}

In other words, each branch $B_k$ is marked by its leaf $x$ and in addition
by any  marked  leaf of $T$ that $B_k$ may contain.

\smallskip
\emph{Claim.} Let $k\in \{1,2,3\}$ and suppose $B_k$ is marked by 
 its leaf $x$ and possibly one other of its leaves as indicated. 
 If $\V_k\ne \emptyset$, then  $x$ lies in a leaf-tile $X\in \X_k$. If $B_k$ has another marked leaf, then this leaf lies in a leaf-tile $Y\in \X_k$ distinct from $X$.
\smallskip

In other words, each marked branch $B_k$ with the decomposition $\X_k$ induced by $\V_k$ satisfies the hypotheses of the lemma 
and we can  apply the induction hypothesis. 

To prove the claim, first observe that there exist precisely three distinct tiles $X_1,X_2,X_3\in\X$ that contain the triple point $x$.
 With suitable labeling we have  $X_k\sub B_k$ for $k=1,2,3$ 
(see Lemma~\ref{lem:vt}~\ref{item:vt6}).

Fix $k\in \{1,2,3\}$ and suppose $\V_k\ne \emptyset$ as in the Claim. We  now set $X\coloneqq X_k\in \X_k$. 
Then  $\partial X\setminus\{x\}=\partial_{k}X\ne \emptyset$ by
  \eqref{eq:relbdX} and Lemma~\ref{lem:vt}~\ref{item:vt3}. On the other hand, $x\in \partial X$ by  Lemma~\ref{lem:vt}~\ref{item:vt5}, and 
 $\#\partial  X\le 2$, because $\X$ is an edge-like decomposition of $T$. We conclude that $\partial X$ contains precisely two distinct points,
 one of which is $x$. So $X$ is an edge-tile in $\X$. Then 
  $\partial_{k}X=\partial X\setminus\{x\}$ is a singleton set, and so  $X$ is a leaf-tile in $\X_k$ that contains $x$. This shows the first part of the Claim.
  
  Suppose $B_k$ has another marked leaf besides $x$. This is only possible if $k\in \{1,2\}$. We assume $k=1$; the case $k=2$ is very similar and we skip the  details. 
   
Then $p$ is a marked leaf of $B_k=B_1$. By our assumptions, 
$p$ is contained in a leaf tile $P\in \X$. By
 Lemma~\ref{lem:vt}~\ref{item:vt1} the tile $P$ is contained in precisely one of the branches of $x$ in $T$. This branch must be 
 $B_1$, because $p\in P$ is belongs to $B_1$, but not to $B_2$ or $B_3$. We conclude that  $P\sub B_1$ and so $P\in \X_1$. 
 
 Since $\partial_{k}P\ne \emptyset$ and
 $\#\partial_{k}P\le \#\partial P=1$, the set $\partial_{k}P$
 consists of a single point.  This means $P$ is a leaf-tile in
 $\X_1$ that contains $p$. We have $P\ne X$, because $X$ is an
 edge-tile and $P$ is a leaf-tile in $\X$.  So if we set $Y=P$, then we see that the second part of the  Claim is also true.
 \smallskip

By the Claim we can  apply the induction hypothesis to $B_k$, marked by
leaves as above, and its edge-like  decomposition $\X_k$ induced by $\V_k$ (recall that $\#\V_k<\#\V$). Then we can find 
tile-homeomorphisms $F_k\: B_k\ra \cT$ for $\X_k$ that are
normalized by
\begin{align*}
  F_1(x)&=\phantom{-}1,
       \text{ and } F_1(p)=-1  \text{ in Cases 2a and 3}, 
          \\
   F_2(x)&=-1,
           \text{ and } F_2(p)=\phantom{-}
           1 \text{ in Case 2b, }   F_2(q)=1 
         \text{ in Case 3, }
           \\
  F_3(x)&=-1.
\end{align*}
Recall that these cases were defined at the beginning of the
proof. They refer to the markings of $T$ (and not of
$B_k$). Case~1 is included here, because then the statements about $p$ and $q$ are void.

We now define 
\begin{equation}
  \label{eq:defF}
  F\: T=B_1\cup B_2 \cup B_3\ra \cT
  \text{ by setting } 
  F(z)=g_k(F_k(z))
\end{equation}
if $z\in B_k$ for $k\in \{1,2,3\}$.
This is well-defined, because $z=x$ is the only point contained
in more than one branch and for which multiple definitions apply;
but we have 
\begin{align*}
  g_1(F_1(x))&=g_1(1)\phantom{-}=0, \\
  g_2(F_2(x))& =g_2(-1)=0, \\
  g_3(F_3(x))&=g_3(-1)=0.  
\end{align*} 
So the definitions are consistent and $F(x)=0$. 

The map $F$ sends $T$ onto
$\cT$, because
\begin{equation*}
  F(T)
  =
  \bigcup_{k=1}^3 g_k(F_k(B_k))
  =
  \bigcup_{k=1}^3 g_k(\cT)=\cT. 
\end{equation*}

Since $F|B_k=g_k\circ F_k$ is continuous for $k=1,2,3$, the map
$F$ is continuous. On each set $B_k$ the map $F|B_k=g_k\circ F_k$
is injective. Moreover, $F$ sends the pairwise disjoint sets
\begin{equation*}
  \{x\}, \,
  B_1\setminus\{x\},\,
  B_2\setminus\{x\},\,
  B_3\setminus\{x\}
\end{equation*}
onto the pairwise disjoint sets
\begin{equation*}
  \{0\},\,
  \cT_1\setminus\{0\},\,
  \cT_2\setminus\{0\},\,
  \cT_3\setminus\{0\}.
\end{equation*}
This implies that $F$ is injective, and hence a continuous
bijection $F\: T\ra \cT$. Since $T$ is compact, $F$ is actually a
homeomorphism. It has the desired normalization, because
\begin{align*}
  F(p)&=g_1(F_1(p))=g_1(-1)=-1
          \text{ in Cases 2a and 3},
  \\
   F(p)&=g_2(F_2(p))=g_2(1)=1
                 \text{ in Case 2b},
  \\
F(q)&=g_2(F_2(q))=g_2(1)=1
                  \text{ in Case 3}.
\end{align*}
   
It remains to show that $F$ is a tile-homeomorphism for $\X$.
Let $X\in \X$ be arbitrary. Then $X$ is contained in one of the
branches $B_k$, and so is a tile in $\X_k$.
Since $F_k$ is a tile-homeomorphism for $\X_k$, we know that
$F_k(X)$ is a tile of $\cT$. This implies that
$F(X)=g_k(F_k(X))$ is a tile of $\cT$, because $g_k$ sends tiles
of $\cT$ to tiles of $\cT$.

This completes the inductive step and the statement follows. The
argument actually shows that the tile-homeomorphism
$F\:T \ra \cT$ with the desired normalization can be constructed
so that $F(x)=0$ with  $x\in \V$ arbitrary,  when $T$ has  one  marked leaf, or 
 with $x\in (p,q)\cap \V$ arbitrary, when $T$ has two marked leaves $p$ and $q$.
    
\smallskip
We now first establish     \ref{item:T_fin_decomp3}, before we show  \ref{item:T_fin_decomp2}.
     
\smallskip
\ref{item:T_fin_decomp3}
Let $p'$ and $q'$ be the first, respectively the last, point on
$(p,q) \cap \V$ as we travel from $p$ to $q$ along $[p,q]$. Then 
 $[p,p']$ is contained in a tile in $\X$; this tile must be $P$, because it is the only tile in $\X$ that contains $p$ (this follows from Lemma~\ref{lem:vt}~\ref{item:vt5}). We see that $p'\in \V  \cap P= \partial P$. Since $P$ is a leaf-tile, we have $\partial P=\{p'\}$. Similarly, $\partial Q=\{q'\}$.
 
There exists
a point $x\in (p',q')\cap \V$ by our assumptions. As in the inductive
step in \ref{item:T_fin_decomp1}, we decompose $T$ into the three
branches $B_1$, $B_1$, $B_3$ with this choice of $x$. Again we
may assume $p\in B_1$ and $q\in B_2$. Note that then
$p'\in (p,x)\sub B_1$, and so
$p'\in \V_1=\V\cap (B_1\setminus\{x\})$. Similarly, $q'\in (x,q)$
and $q'\in \V_2=\V\cap (B_2\setminus\{x\})$. We now choose
tile-homeomorphisms $F_k\:B_k\ra \cT$ for $\X_k$ as before, but
by the remark at the end of the proof of
\ref{item:T_fin_decomp1}, we can do this so that  $F_1(p')=0$
and $F_2(q')=0$. Then the map $F$ as defined in \eqref{eq:defF}
satisfies
\begin{align*}
  F(p')
  &=
    g_1(F_1(p'))
    =
    g_1(0)=-1/2, 
  \\
  F(q')
  &=
    g_2(F_2(q'))
    =
    g_2(0)=1/2. 
 \end{align*}
 Since $p\in P$, $p'\in \partial P$, and  $F$ is a
  tile-homeomorphism,  we have 
   $Z\coloneqq F(P)\in \cX^*$, $-1=F(p)\in Z$,  and  
   $-1/2= F(p')\in \partial Z$. Lemma~\ref{lem:2tileCSST} 
   implies that $Z=F(P)=\cT_{11}$, which is a tile of 
 level $2$. By a similar reasoning, $F(Q)=\cT_{22}$, which is again of level $2$. The statement follows.

 \smallskip
 \ref{item:T_fin_decomp2}
 This is a slight variant of the argument  for
 \ref{item:T_fin_decomp3}. The leaf-tile  $P\in X$ with  $p\in P$ is the 
 unique tile in $\X$ that contains $p$. It boundary $\partial P\sub \V$ 
 consists of a single point $p'\in \V$. Now we consider the three
 tiles containing $p'\in \V$. One of them is $P$, but not all
 three can be leaf-tiles, because then $\#\V=1$. So there exists
 an edge-tile $X$ with $p'\in \partial X$. Then $\partial X\sub \V$
 contains another point $x\in \V$ distinct from $p'$. As one
 travels along the arc $[p,x]$ starting from $p$, one exits $P$,
 but this is only possible through $p'$ which is the only
 boundary point of $P$. Hence $p'\in (p,x)$. We now use a
 construction as in \ref{item:T_fin_decomp1} for the branch point
 $x\in \V$ and its three branches $B_1$, $B_2$, $B_3$.
 
  To obtain the first normalization, we may assume that
  $p\in B_1$.  Since $p'\in (p,x)$, we can choose the
  homeomorphism $F_1\: B_1\ra \cT$ so that $F_1(p)=-1$ and
  $F_1(p')=0$. Then, as before, we obtain a tile-homeomorphism
  $F\: T\ra \cT$ for $\X$. It satisfies 
  \begin{align*}
    F(p)
    &=
      g_1(F_1(p))=g_1(-1)=-1,
    \\
    F(p')
    & =
      g_1(F_1(p'))
      =
      g_1(0)
      =
      -1/2.
  \end{align*}
  Using again Lemma~\ref{lem:2tileCSST}, we see  that
  $F(P) = \cT_{11}$, and so $F(P)$  is a tile in $\cT$ of level $2$.
 
  To obtain the second normalization, we may assume that
  $p\in B_2$ and now choose $F_2\: B_2\ra \cT$ so that $F_2(p)=1$
  and $F_2(p')=0$. Then
  \begin{align*}
    F(p)
    &=
      g_2(F_2(p))
      =
      g_2(1)
      =
      1,
    \\
    F(p')
    & =
      g_2(F_2(p'))
      =
      g_2(0)
      =
      1/2.%\qedhere
  \end{align*}
  From Lemma~\ref{lem:2tileCSST} we see that $F(P)= \cT_{22}$, which is a
  tile of level $2$. We have proved \ref{item:T_fin_decomp2}.

  \smallskip
  \ref{item:qv_approx4}
  Let $\cV\coloneqq F(\V)\subset \cT$. Since $F\: T\ra \cT$  is a
 homeomorphism, $\cV$ is a set of branch points of $\cT$ 
 that  induces a  decomposition of $\cT$ into the set of tiles 
  $\cX=F(\X)$. Since $F$ is a tile-homeomorphism, we 
  have $\cX\sub \cX^*$. The statement now immediately
   follows  from
  Lemma~\ref{lem:admh}.

\smallskip
The proof is complete.
\end{proof}

\subsection*{Subdividing  \texorpdfstring{$\cT$}{T}}
\label{sec:subdividing-ct}
After these preparations, we are now ready to prove Proposition~\ref{prop:exqs}, the main result
in this section.

The desired quasi-visual subdivision $\{\Y^n\}$ will be
constructed inductively from auxiliary tile-homeomorphisms $F^n\:
T\ra \cT$ for $\X^n$, $n\in \N_0$. We then set 
\begin{equation*}
  \Y^n\coloneqq F^n(\X^n) =
  \{F^n(X) : X\in \X^n\}. 
\end{equation*}
Recall that if $F^n$ is a tile-homeomorphism, then  for each tile  $X\in \X^n$ we have
$F^n(X)\in \cX^*$, and so $\Y^n\subset \cX^*$ (meaning that $\Y^n$ is a set
of tiles in $\cT$).

 The next map $F^{n+1}\: T\ra \cT$ ``refines'' $F^n$ in the sense that for $X\in \X^n$ we have 
$F^{n+1}(X)=F^n(X)$. Thus, for  $X\in \X^n$  and $X'\in
\X^{n+1}$ with $X'\sub X$, we have $F^{n+1}(X')\sub
F^{n+1}(X)=F^n(X)$ in $\cT$. This will ensure that $\{\X^n\}$ and
$\{\Y^n\}$ are isomorphic subdivisions according to
\eqref{eq:isosubdiv_cap} and \eqref{eq:isosubdiv_incl}. 

We will  show that 
$\{\Y^n\}$ is actually  a quasi-visual subdivision
of $\cT$. Then by Proposition~\ref{prop:qv_f_qs}
there exists  a quasisymmetry $F\colon T\to \cT$  
that induces the isomorphism between $\{\X^n\}$ and 
$\{\Y^n\}$. In fact, one can  show that the sequence $\{F^n\}$
converges uniformly to this quasisymmetry $F$;
 we will leave the easy argument for this to the reader.

Note that for a given tile $X\in \X^n$, the
level of $F^n(X)\in \cX^*$, i.e., $\ell(F^n(X))$, will in general be
different from $n$ (the level of $X$ in $T$). Moreover, if $X,Y\in \X^n$, the
levels of the tiles $F(X)$ and $F(Y)$ (in $\cT$) may be
different.

\begin{proof}
  [Proof of Proposition~\ref{prop:exqs}]
 By our assumptions $T$ is a uniformly branching trivalent
quasiconformal tree, and  the set  $\V^n$ for $n\in \N_0$  is  as in \eqref{eq:defvn} 
with $\delta\in (0,1)$  so small that 
all the statements  
in Proposition~\ref{prop:decomp} are true for the edge-like decompositions
$\X^n$  of $T$ induced by $\V^n$.  In particular, 
$\V^0=\emptyset$, and the sets $\V^n$ form an increasing sequence of branch points as in \eqref{eq:Vn_incrs}. The sequence $\{\X^n\}$ is a subdivision of $T$. Condition~\ref{item:decomp_Vfinite} in 
Proposition~\ref{prop:decomp} implies that $\V^n\ne \emptyset$ 
for $n\in \N$. 

For $n\in \N_0$ we will now 
 inductively  construct homeomorphisms
$F^n\: T\ra \cT$ with the following properties:

\begin{enumerate}[label=(\Alph*)]
\item 
  \label{item:Fn1}
  For each $n\in \N_0$ the map $F^n\:T\ra\cT$ is a tile-homeomorphism for $\X^n$, i.e., $F^n$ is a homeomorphism from  $T$ onto $\cT$ such that $F^n(X)\in \cX^*$ for 
  $X\in \X^n$.
  %
  % save counter of enumeration
  \setcounter{mylistnum}{\value{enumi}}
\end{enumerate}
In addition, for all  $n\in \N$ and $X\in \X^{n-1}$, the following statements are true:
\begin{enumerate}[label=(\Alph*)]
  %   % reset counter
  \setcounter{enumi}{\value{mylistnum}}
  
\item 
  \label{item:Fn2}    
  The maps $F^{n-1}$ and $F^{n}$ are compatible in
  the sense that
  \begin{equation*}
    F^{n-1}(X)=F^{n}(X).
  \end{equation*}
%  for each $X\in \X^{n-1}$ and $n\in \N$.
  % 
\item 
  \label{item:Fn3}
  There exists a constant $N\in \N$ independent of $n$ and $X$ with the following property:
  if $Y \in \X^{n}$ with $Y\sub X$, then
  \begin{equation*}
    \ell(F^{n-1}(X))+1
    \le
    \ell(F^{n}(Y))
    \le
    \ell(F^{n-1}(X))+N.
  \end{equation*}
\item
  \label{item:Fn4}
  If $Y\in \X^n$ with $Y\subset X$
  and $Y\cap \partial X\ne \emptyset$, then
  \begin{equation*}
    \ell(F^{n}(Y))
    =
    \ell(F^{n-1}(X))+2. 
  \end{equation*}
\end{enumerate}

The constant $N\in\N$ in \ref{item:Fn3} will be the same as in 
statement \ref{item:decomp3}
in Proposition~\ref{prop:decomp}, which is true by our
assumptions. Statement~\ref{item:Fn3}  means that the level of the tile 
$Y'=F^n(Y)\in \cX^*$ is strictly, but only by a bounded amount
larger than the level of $X'=F^{n-1}(X)\in \cX^*$. In
particular, $2^{-N}\diam(X')\leq \diam(Y') \leq \frac12
\diam(X')$. In \ref{item:Fn4} we actually have
$\diam(Y') = \frac14 \diam(X')$.

%  
%Once we have constructed the sequence of maps $\{F^n\}$, we will see that $\{F^n(\X^n)\}$ is a subdivision of $\cT$ with the same
%combinatorics as $\{\X^n\}$, meaning these subdivisions  are isomorphic in the
%sense of \eqref{eq:isosubdiv_cap} and
%\eqref{eq:isosubdiv_incl}. We will  show that 
%$\{F^n(\X^n)\}$ is actually  a quasi-visual subdivision
%of $\cT$. If we define $\Y^n=F^n(\X^n)$ for $n\in \N_0$, then 
%$\{\Y^n\}$ is a subdivision of $\cT$ with the desired properties. We now present the details. 

\smallskip 
{\em Construction of $F^0$ and $F^1$.}
%{\em Construction of $\{F^n\}$.}
Since $\V^0=\emptyset$, we have $\X^0=\{T\}$. We know that $T$ is
homeomorphic to $\cT$ (by Theorem~\ref{thm:homTpqr}) and so we
can choose a homeomorphism $F^0\: T\ra \cT$. Obviously, $F^0$ is
a tile-homeomorphism for $\X^0$, since
$F^0(T)=\cT=\cT_{\emptyset}$ is a tile of $\cT$.

Note that $\V_T=\V^1 \cap\inte(T)= \V^1$ and the decomposition $\X^1$ of
$T$ induced by $\V^1$ is edge-like according to
Proposition~\ref{prop:decomp}~\ref{item:decomp1}.
 
So we may apply Lemma~\ref{lem:decomp} to $T$
and $\V^1$, where no leaves of $T$ are marked. This gives us a
tile-homeomorphism $F^1\: T\ra \cT$ for $\X^1$. Then
\ref{item:Fn1} is true for
$n=0$ and $n=1$. To verify  \ref{item:Fn2}--\ref{item:Fn4} for $n=1$ note that $X=T\in \X^0$ is the only $0$-tile. 
Since $F^0(T)=\cT=F^1(T)$, we obviously have   \ref{item:Fn2}.

Recall that $\V^1\neq \emptyset$ and let $N\in \N$ be the number
from Proposition~\ref{prop:decomp}~\ref{item:decomp3}. By
definition of $N$ we then have $\#\V^1\le N$ for the number of
$1$-vertices in the $0$-tile $T$.  Together with
Lemma~\ref{lem:decomp}~\ref{item:T_fin_decomp4} this implies
\begin{equation*}
  1
  \leq
  \ell(F^1(Y)) \leq \#\V^1\leq N
\end{equation*}
for all $Y\in \X^1$. Since $\ell(F^0(X))=\ell(\cT)=0$,  \ref{item:Fn3}
holds for $n=1$. 
% Note that here necessarily $X=T$, and
% $\ell(F^0(X))=\ell(\cT)=0$.
Moreover, \ref{item:Fn4} is vacuously true for $n=1$, since
$\partial T=\emptyset$. Thus, with the maps $F^0$ and $F^1$ as above,
conditions  \ref{item:Fn1}--\ref{item:Fn4} are true for
$n=1$.

\smallskip
For the inductive step, suppose  that for some $n\in \N$ maps
$F^{n-1}$ and $F^n$  with the properties
\ref{item:Fn1}--\ref{item:Fn4} have been defined. The map
$F^{n+1}$ will be constructed separately on each tile
$X\in \X^n$. The idea is to modify the map $F^n|X$ to a
homeomorphism $F_X$ of $X$ onto the same image $F^n(X)$, and
``glue'' the maps $\{F_X\}$, ${X\in \X^n}$, together to obtain
$F^{n+1}$.

\smallskip
\emph{Constructing $F^{n+1}$ on an $n$-tile $X$.}
Fix a tile $X\in \X^n$. Define
$\V_X\coloneqq \V^{n+1}\cap \inte(X)$. Then $2\le \#\V_X\le N$
and $\V_X$ induces an edge-like decomposition of $X$ into tiles,
denoted by $\X_X$ (see \ref{item:decomp1}, \ref{item:decomp3},
and \ref{item:decomp_Vfinite} in 
Proposition~\ref{prop:decomp}). The set $\X_X$
consists precisely of the tiles $Y\in \X^{n+1}$ with $Y\sub
X$ (see Lemma~\ref{lem:vtt}~\ref{item:vtt3}). We want to construct a homeomorphism $F_X\: X\ra F^n(X)$ with
the following properties:
\begin{enumerate}[label=(\alph*)]
\item
  \label{item:FX1}
  The set $F_X(Y)$ is a tile of $\cT$ for each $Y\in \X_X$.
\item
  \label{item:FX2}
  $F_X(X)=F^n(X)$ and 
  $F_X|\partial X=F^n|\partial X$.
\item
  \label{item:FX3}
  For each $Y\in \X_X$, we have
  \begin{equation*}
    \ell(F^n(X)) +1
    \le
    \ell (F_X(Y) )
    \le
    \ell (F^n(X)) +N. 
  \end{equation*}
\item
  \label{item:FX4}
  If $Y\in \X_X$ with $Y\cap \partial X\ne \emptyset$, then
  \begin{equation*}
    \ell(F_X(Y))=\ell (F^n(X))+2.  
  \end{equation*}
\end{enumerate}

% \begin{align} 
%   \label{item:FX1}
%   &\text {For each $Y\in \X^{n+1}$  with $Y\sub X$,
%     the set $F_X(Y)$ is a tile of $\cT$.}\\     
%   \label{item:FX2}
%   & F_X(X)=F^n(X) \text { and }
%     F_X|\partial X=F^n|\partial X.  \\ 
%   \label{item:FX3}
%   &\text {For each $Y\in \X^{n+1}$ with $Y\sub X$, we have }\\
%   &\quad \quad \quad \ell(F^n(X)) +1\le \ell (F_X(Y) ) \le  \ell (F^n(X)) +N. \notag \\
%   &\text {and if $Y\cap \partial X\ne \emptyset$, then }
%     \ell(F_X(Y))=\ell (F^n(X))+2.  \notag
% \end{align}

Note that the properties \ref{item:FX1}--\ref{item:FX4}
correspond to the properties \ref{item:Fn1}--\ref{item:Fn4}. The
construction of $F_X$ is slightly different depending on whether
$X$ is an edge-tile or a leaf-tile.
% We will discuss the details
% in case $X$ is an edge-tile and will comment on the modifications
% for a leaf-tile $X$.
We will discuss the details
when $X$ is an edge-tile and will comment on the modifications
when it is a leaf-tile.

\smallskip
\emph{Case~1:} $X$ is an edge-tile.
Then $\#\partial X =2$, say $\partial X=\{p,q\}\sub \V^n$. Since
$F^n$ is a tile-homeomorphism for $\X^n$, the set $F^n(X)$ is a
tile of $\cT$, say $F^n(X)=\cT_w=g_w(\cT)$ with $w\in
\A^*$ (see \eqref{eq:def_Tw}). Since $F^n$ is a homeomorphism, we have
\begin{equation*}
  \partial \cT_w
  =
  \partial F^n(X)
  =
  F^n(\partial X)
  =
  F^n(\{p,q\}). 
\end{equation*}
In particular, $\cT_w=g_w(\cT)$ has two boundary points and so
$F^n(\{p,q\})=\partial \cT_w=\{g_w(-1), g_w(1)\}$ by
Lemma~\ref{lem:bdry}~\ref{item:bdry1}.  We may assume that
\begin{equation}
  \label{eq:Fngw}
  F^n(p)=g_w(-1)
  \text{ and }
  F^n(q)=g_w(1).
\end{equation}

By our assumptions the open arc $(p,q)$ contains at least three points in $\V_X$ (see
Proposition~\ref{prop:decomp}~\ref{item:decomp4}). In particular,
$p$ and $q$ lie in distinct tiles $P,Q \in \X^{n+1}$,
respectively, with $P,Q\sub X$. These are leaf-tiles in
$\X_X$. Indeed, $P$ has exactly two boundary points in $T$,
namely $p$ and one boundary point in $\V_X$ (if not, then $X=P$,
which is impossible since $\V_X\ne \emptyset$). It follows that the relative boundary of
$P$ in $X$, which is a subset of $\V_X$, contains precisely one
point, and so $P$ is a leaf-tile in $\X_X$. Similarly, $Q$ is a
leaf-tile in $\X_X$.

Therefore, we may apply Lemma~\ref{lem:decomp} to the tree $X$,
marked by $p$ and $q$ (which are leaves of $X$), and the set
$\V_X$. Thus there exists a tile-homeomorphism $h_X\: X\ra \cT$
for $\X_X$ with $h_X(p) =-1$ and $h_X(q)=1$ that satisfies
\ref{item:T_fin_decomp3} and \ref{item:T_fin_decomp4} in 
Lemma~\ref{lem:decomp} as well.

We now define $F_X\: X\ra F^n(X)=g_w(\cT)=\cT_w$ as
$F_X=g_w\circ h_X$. This is a homeomorphism of $X$ onto
$F^n(X)$. Since $g_w$ sends tiles of $\cT$ to tiles of $\cT$, we
have \ref{item:FX1}.
 
Clearly, we have  $F_X(X)=F^n(X)$ as well. Moreover, 
\begin{equation*}
  F_X(p)
  =
  g_w(h_X(p))
  =
  g_w(-1)
  =
  F^n(p),
\end{equation*}
by \eqref{eq:Fngw} and similarly $F_X(q)=F^n(q)$. Since
$\partial X=\{p,q\}$, we have shown \ref{item:FX2}.

Recall that $2\leq \#\V_X\leq N$.
% Recall that $N$ is the number in
% Proposition~\ref{prop:decomp}~\ref{item:decomp3}, meaning that
% $\#\V_X\leq N$.
Thus, using Lemma~\ref{lem:decomp}~\ref{item:T_fin_decomp4}, we
obtain for any $Y\in \X_X$ that
\begin{equation*}
  1
  \le
  \ell(h_X(Y))
  \le
  \#\V_X
  \le
  N.
\end{equation*}
Together with
\begin{align*}
\ell(F_X(Y))&= \ell(g_w(h_X(Y)))=\ell(w) +\ell(h_X(Y))\\
& =
\ell (F^n(X))+\ell(h_X(Y)), 
\end{align*}
this implies \ref{item:FX3}.

Suppose  $Y\in \X_X$ and $Y\cap\partial X\ne \emptyset$.
Then $Y$
contains one of the marked leaves  $p$ or $q$ of $X$; so $Y$ is identical with one of the tiles $P$ or $Q$, because each leaf of $X$ is  contained in unique tile in $\X_X$ as follows from Lemma~\ref{lem:vt}~\ref{item:vt5}.  Since 
  $h_X$
satisfies condition~\ref{item:T_fin_decomp3} in  Lemma~\ref{lem:decomp}, we have 
$\ell(h_X(Y))=2$. Thus
\begin{equation*}
  \ell(F_X(Y))
  =
  \ell (F^n(X))+\ell(h_X(Y))
  =\ell (F^n(X))+2,
\end{equation*}
and  \ref{item:FX4} follows. We have constructed
$F_X$, satisfying \ref{item:FX1}--\ref{item:FX4}, when $X\in
\X^n$ is an edge-tile. 
\smallskip

\emph{Case~2:} $X$  is a leaf-tile.
Then $\#\partial X=1$, say $\partial X=\{p\} \sub \V^n$. Again
there exists $w\in \A^*$ with $F^n(X)=\cT_w=g_w(\cT)$. Then
$\partial \cT_w=F^n(\partial X)= \{F^n(p)\}$ is a singleton set
and so $\partial \cT_w=\{g_w(-1)\}$ or $\partial \cT_w=\{g_w(1)\}$
by Lemma~\ref{lem:bdry} \ref{item:bdry1}.

The point $p$ is a leaf of $X$. Moreover, it is contained in a
tile $P\in \X_X$ which is a leaf-tile in $\X_X$ (by the exact
same argument used when $X$ is an edge-tile). We
know  $\#\V_X \geq 2$ and that the decomposition $\X_X$ of $X$ induced by
$\V_X$ is edge-like.
% from
% Proposition~\ref{prop:decomp}~\ref{item:decomp3}.

Thus we may apply Lemma~\ref{lem:decomp} to the tree $X$ marked
by its leaf $p$ and the  decomposition $\X_X$ of $X$ induced by the set $\V_X$. So there
exists a tile-homeomorphism $h_X\: X\ra \cT$ for $\X_X$ satisfying 
the (applicable) conditions  in
Lemma~\ref{lem:decomp}. We may normalize $h_X$ so that 
\begin{equation*}
  h_X(p)=-1 \text{ if } \partial \cT_w=\{g_w(-1)\},
  \,\,\text{or}\,\,
  h_X(p)=1 \text{ if } \partial \cT_w=\{g_w(1)\}.
\end{equation*}
% $h_X(p)=0$ if $\partial \cT_w=\{g_w(0)\}$ and so that
% $h_X(p)=1$ if $\partial \cT_w=\{g_w(1)\}$.

The homeomorphism  $F_X\colon X\to F^n(X)=g_w(\cT)=\cT_w$ is now defined by
$F_X=g_w\circ h_X$ as before. Then again \ref{item:FX1} and the
first part of \ref{item:FX2} are true. The normalization of
$h_X$ ensures that $F_X$ maps $p$ to  the unique
 point in $\partial \cT_w$. Since $F^n$ does the same, we have
$F^n(p)=F_X(p)$, which is the second part of \ref{item:FX2},
because $\partial X=\{p\}$.

Finally, statements~\ref{item:FX3} and \ref{item:FX4} follow from  similar arguments as in the
case of an edge-tile $X$.  For  \ref{item:FX4} one uses the fact 
that  $\ell(h_X(P))=2$ by
Lemma~\ref{lem:decomp}~\ref{item:T_fin_decomp2}.

This finishes the construction of the homeomorphism
$F_X\: X\ra F^n(X)$ with the properties
\ref{item:FX1}--\ref{item:FX4} for a leaf-tile $X\in \X^n$. Hence
we have constructed such $F_X$ for every tile $X\in \X^n$. 

\smallskip
We now define $F^{n+1}\: T\ra \cT$ by setting
$F^{n+1}(z)= F_X(z)$ if $z\in X\in \X^n$. Note that the tiles
$X\in \X^n$ cover $T$, and so each point $z\in T$ lies in one of
the tiles $X\in \X^n$.  Moreover, $F^{n+1}$ is well-defined;
indeed, suppose that $z\in X\cap Y$ with $X,Y\in \X^n$, $X\ne
Y$. Then $z\in \partial X\cap \partial Y$ (see Lemma~\ref{lem:vt}~\ref{item:vt5}), and so \ref{item:FX2}
shows that
\begin{equation*}
  F_X(z)=F^n(z)=F_Y(z).
\end{equation*}

The map $F^{n+1}$ is continuous, because $F^{n+1}|X=F_X$ is
continuous for each tile $X\in \X^n$ and these finitely many tiles cover $T$.
The map $F^{n+1}$ is also surjective, because
\begin{equation*}
  F^{n+1}(T)
  =
  \bigcup_{X\in \X^n}F_X(X)
  =
  \bigcup_{X\in \X^n}F^n(X)
  =
  F^n(T)=\cT. 
\end{equation*}
To show that it is injective, suppose that $x,y\in T$ and
$F^{n+1}(x)=F^{n+1}(y)$. Then there exist $X,Y\in \X^n$ such that
$x\in X$ and $y\in Y$. Now if $X=Y$ and so $x,y\in X$, then $x=y$, because 
$F^{n+1}|X=F_X$ 
is injective. 

 If
$X\ne Y$, then
\begin{equation*}
  u\coloneqq F^{n+1}(x)\sub F^{n+1}(X)=F_X(X)=F^n(X)
\end{equation*}
and similarly, $u=F^{n+1}(y)\in F^n(Y)\ne F^n(X)$.  So $u$ lies
in different tiles of the decomposition of $\cT$ by the set of
tiles $F^n(\X^n)$. Then
\begin{equation*}
  u\in \partial F^n(X)\cap  \partial F^n(Y)
  =
  F^n(\partial X\cap \partial Y). 
\end{equation*}
It follows that there exists a point
$v\in \partial X \cap \partial Y\sub \V^n$ with $F^n(v)=u$.  Then
by \ref{item:FX2} we have
\begin{equation*}
  F^{n+1}(v)=F_X(v)=F^n(v)=u=F^{n+1}(x).
\end{equation*}
Since $v,x\in X$, and $F^{n+1}|X=F_X$ is injective, we see that
$x=v$, Similarly, $y=v$, and so $x=v=y$, as desired.
 
We have shown that $F^{n+1}\: T\ra \cT$ is a continuous
bijection. Since $T$ is compact, it follows that
$F^{n+1}\: T\ra \cT$ is a homeomorphism. It remains to verify
conditions \ref{item:Fn1}--\ref{item:Fn4} for $n+1$.
 
First, if $Y\in \X^{n+1}$ is arbitrary, then there exists a
unique tile $X\in \X^n$ such that $Y\sub X$. It now follows from
\ref{item:FX1} that $F^{n+1}(Y)$ is a tile of $\cT$, because
$ F^{n+1}(Y)=F_X(Y)$. So we have \ref{item:Fn1}.
 
Condition~\ref{item:Fn2} is also clear, because if $X\in \X^n$ is
arbitrary, then by \ref{item:FX2} we have
\begin{equation*}
  F^n (X)
  =
  F_X(X)
  =
  F^{n+1}(X).
\end{equation*}

Finally, \ref{item:Fn3} and \ref{item:Fn4} (for $n+1$)
immediately follow from \ref{item:FX3} and
\ref{item:FX4}, respectively,  and the definition of $F^{n+1}$.
 
So we obtain a homeomorphism $F^{n+1}$ as desired, and the
inductive step is complete. We conclude that there exists a
sequence $\{F^n\}$ of maps with the  properties \ref{item:Fn1}--\ref{item:Fn4}.

\smallskip
Having completed the construction of the sequence $\{F^n\}$, we now show that
$\{F^n(\X^n)\}$ is a quasi-visual subdivision of $\cT$
isomorphic to $\{\X^n\}$. 

\smallskip
{\em $\{F^n(\X^n)\}$ is a subdivision of $\cT$.}
We know that the sequence $\{\X^n\}$ is a subdivision of the space $T$ as in
Definition~\ref{def:subdiv}. To see that $\{F^n(\X^n)\}$ is a
subdivision of $\cT$, first note that for each $n\in \N_0$ the
map $F^n$ is a homeomorphism by \ref{item:Fn1}. Thus, the
set  $F^n(\X^n)=\{F^n(X): X\in \X^n\}$ is a finite
collection of compact subsets of $\cT$ that cover $\cT$. We have 
that $\{F^0(\X^0)\}=\{\cT\}$. It follows  that the condition in 
Definition~\ref{def:subdiv}~\ref{item:subdiv1} is true.  
 
Let $n\in \N_0$. If $Y'\in F^{n+1}(\X^{n+1})$ is arbitrary, then
$Y'=F^{n+1}(Y)$ for some $Y\in \X^{n+1}$.  Then there exists
$X\in \X^n$ with $Y\sub X$ and so
\begin{equation*}
  Y'=F^{n+1}(Y)\sub F^{n+1}(X)=F^n(X)\in F^n(\X^n),
\end{equation*}
where \ref{item:Fn2} was used. Thus the condition in Definition~\ref{def:subdiv}~\ref{item:subdiv2} holds.

Finally, if $X'\in F^n(\X^n)$ is arbitrary, then $X'=F^n(X)$ with $X\in \X^n$. Since
$$ X=\bigcup \{Y\in \X^{n+1}: Y\sub X\}, $$ 
it follows that $X'$ can be written as a union of sets in  
$F^{n+1}(\X^{n+1})$, because (using \ref{item:Fn2} once more)
\begin{equation*}
  X'
  =
  F^n(X)
  =
  F^{n+1}(X)
  =
  \bigcup \{F^{n+1}(Y): Y \in \X^{n+1},\,  Y\sub X\}.
\end{equation*}
We see that the condition in  Definition~\ref{def:subdiv}~\ref{item:subdiv3} is also true, and  conclude that  $\{F^n(\X^n)\}$ is indeed a subdivision of $\cT$. 

\smallskip
{\em The subdivisions  $\{\X^n\}$ and $\{F^n(\X^n)\}$ are
  isomorphic.} We claim that the bijective correspondence given
by $X\leftrightarrow F^n(X)$ for  $X\in \X^n$ on each level
$n\in \N_0$ is an isomorphism  of $\{\X^n\}$ and $\{F^n(\X^n)\}$.

To see this,  fix $n\in \N_0$, and  let $X,Y\in \X^n$ be arbitrary. Then 
\begin{equation*}
  X\cap Y\neq \emptyset
  \quad
  \text{if and only if}
  \quad
  F^n(X)\cap F^n(Y)\neq \emptyset,
\end{equation*}
since $F^n$ is a homeomorphism by \ref{item:Fn1}. Thus
\eqref{eq:isosubdiv_cap} is satisfied. 
Moreover, let $X\in \X^n$ and $X'\in \X^{n+1}$ be arbitrary. Then
$F^{n+1}(X) = F^n(X)$ by \ref{item:Fn2}. Thus
\begin{equation*}
  X'\subset X
  \quad
  \text{if and only if}
  \quad
  F^{n+1}(X') \subset F^{n+1}(X) = F^n(X),
\end{equation*}
since $F^{n+1}$ is a homeomorphism. We have shown
\eqref{eq:isosubdiv_incl}; so $\{\X^n\}$ and $\{F^n(\X^n)\}$
are indeed isomorphic.

\smallskip  
{\em $\{ F^n(\X^n)\}$ is a quasi-visual subdivision of $\cT$}.
Since we already know that $\{F^n(\X^n)\}$ is a
subdivision of $\cT$, we have to show that $\{ F^n(\X^n)\}$ is a
quasi-visual approximation of $\cT$. For this we will verify 
conditions \ref{item:qv_approx1}--\ref{item:qv_approx4} in
Definition~\ref{def:qv_approx}.

To see \ref{item:qv_approx1}, let $n\in \N_0$ be arbitrary and
consider two tiles in $F^n(\X^n)$ with non-empty
intersection. These tiles can be represented in the form
$F^n(X)$ and $F^n(Y)$ with $X,Y\in \X^n$, where
$X\cap Y\ne \emptyset$. We want to show that $\diam(F^n(X))\asymp
\diam(F^n(Y))$, with a constant $C(\asymp)$ independent of $n$, 
$X$, and
$Y$. 

We can find unique tiles $X^i, Y^i\in \X^i$ for
$i=0, \dots, n$ with
\begin{align*}
  &X=X^n\sub X^{n-1} \sub  \dots \sub X^0=T
    \text{ and }\\
  &Y =Y^n\sub Y^{n-1} \sub  \dots \sub Y^0=T.
\end{align*}
Let $k\in \N_0$, $0\le k\le n$,  be the largest number with $X^k=Y^k$. If $k=n$,
there is nothing to prove, because then $X=X^n=Y^n=Y$, and so  
$\diam(F^n(X))= \diam(F^n(Y))$.

Otherwise, $0\le k<n$, and $X^{k+1}, Y^{k+1}\subset X^k=Y^k$ with
$X^{k+1}\neq Y^{k+1}$. Let $L\coloneqq\ell(F^k(X^k))$. By \ref{item:Fn3} we know that
\begin{align}
  \label{eq:ell_est1}
  L+1
  \le &  
  \ell(F^{k+1}(X^{k+1}))
  \le 
  L+N,
        \text{ and similarly, }
  \\
  \notag
  L +1
  \le&
       \ell(F^{k+1}(Y^{k+1}))
       \le
       L+N. 
\end{align}

Note that $X^i\cap Y^i\supset X^n\cap Y^n \neq \emptyset$ and
$X^i\neq Y^i$ for $i=k+1,\dots, n$. By Lemma~\ref{lem:vt}~\ref{item:vt5} the tiles  $X^i$ and $Y^i$ intersect  in a
single point, which is the unique point $v\in X^n\cap Y^n$. Moreover,   $v\in \partial X^i\cap  \partial Y^i$ for $i=
k+1,\dots, n$. 

Then  \ref{item:Fn4} implies that
\begin{align*}
  \ell(F^n(X))
  &=
  \ell(F^{n-1} (X^{n-1})) + 2
  =
    \dots
    \\
  &=
  \ell(F^{k+1} (X^{k+1})) + 2(n-k-1),
  \intertext{and similarly}
  \ell(F^n(Y))
  &=
%  \ell(F(Z^n))
%  =  
  \ell(F^{k+1}(Y^{k+1})) + 2(n-k-1).
\end{align*}
Combining this  with \eqref{eq:ell_est1}, we see  that
\begin{align*}
  \abs{\ell (F^n(X))-\ell(F^n(Y))}
  &=
  \abs{\ell(F^{k+1}(X^{k+1}))- \ell(F^{k+1}(Y^{k+1}))}\\
  &\le
  N-1.  
\end{align*}
Hence
\begin{equation*}
  \diam (F^n(X))=2^{-\ell(F^n(Y))+1}
  \asymp
  2^{-\ell(F^n(X))+1} =\diam (F^n(Y))
\end{equation*}
with $C(\asymp)=2^{N-1}$. Since $N$ is a fixed number, condition
\ref{item:qv_approx1} in Definition~\ref{def:qv_approx} follows.

To verify condition~\ref{item:qv_approx2}  in Definition~\ref{def:qv_approx}, let $n\in \N_0$ and 
$X',Y'\in F^n(\X^n)$  with
$X'\cap Y'=\emptyset$ be arbitrary. We choose points  $x\in X'$ and $y\in Y'$ 
with
\begin{equation*}
  \abs{x-y}
  =
  \dist(X', Y'),
\end{equation*}
and  consider the unique arc $[x,y]$
joining $x$ and $y$ in $\cT$.  Then the points in
$F^n(\V^n)$ that lie on $[x,y]$ decompose this arc into
non-overlapping subarcs that lie in distinct tiles of the edge-like 
decomposition $F^n(\X^n)$ of $\cT$. At least one of these
subarcs must have interior disjoint from $X'$ and $Y'$, because
$X'\cap Y'=\emptyset$. Let $\alpha $ be the first of these arcs
as we travel along $[x,y]$ from $x$ to $y$, say
$\alpha =[p,q]\sub [x,y]$.  Then there exists a tile
$Z'\in F^n(\X^n)$ such that $\alpha\sub Z'$. Obviously,
$Z'\ne X',Y'$. The two endpoints $p$ and $q$ of $\alpha$ belong
to $\partial Z'$, and so $Z'$ is an edge-tile. Moreover, one of
these endpoints also belongs to $\partial X'$, and so
$X'\cap Z'\ne \emptyset$.

The set $Z'$ is a tile of $\cT$, and so it can be written in the form $Z'=g_w(\cT)$
with $w\in \A^*$. By Lemma~\ref{lem:bdry}~\ref{item:bdry1} we have 
$\partial Z'=\partial g_w(\cT)=\{g_w(-1), g_w(1)\}=\{p,q\}$. 
It follows that
\begin{equation*}
  \abs{p-q}
  =
  \abs{g_w(-1)-g_w(1)}
  =
  2^{-\ell(w)+1}
  =
  \diam (g_w(\cT))
  =
  \diam(Z'). 
\end{equation*}
By condition~\ref{item:qv_approx1} in
Definition~\ref{def:qv_approx}, which we proved already, we also
have $ \diam(X') \asymp \diam(Z')$. Combining this with the
quasi-convexity of $\cT$ (see~\eqref{eq:cT_geod}), we arrive at
\begin{align*}
  \dist(X',Y')
  &=
    \abs{x-y}
    \asymp
    \length [x,y]
    \ge
    \length [p,q]\\
  &\ge \abs{p-q}
    =
    \diam (Z')
    \asymp
    \diam(X').
\end{align*}
Since the implicit multiplicative constants here are independent
of $n$, $X'$, and $Y'$, condition \ref{item:qv_approx2} in
Definition~\ref{def:qv_approx} follows.

To verify the remaining conditions, suppose $k,n\in \N_0$,
$X'\in F^n(\X^n)$, $Y'\in F^{n+k} (\X^{n+k})$, and $X'\cap Y'\ne \emptyset$. 
There exists a %unique
tile $Z'\in F^n(\X^n)$ with $Y'\sub Z'$. Then 
$X'\cap Z'\ne \emptyset$, and so $\diam(X')\asymp
\diam(Z')$ by what we have seen.

If $k=1$, then it follows from  \ref{item:Fn3} 
that
\begin{equation*}
  \ell(Z')+1\le \ell(Y')\le \ell(Z')+N.
\end{equation*}
Since $\diam(Y')=2^{-\ell(Y')+1}$ 
and $\diam(Z')=2^{-\ell(Z')+1}$, we conclude  that
\begin{equation*}
  \diam(Y')\asymp \diam(Z')\asymp \diam(X')  
\end{equation*}
with implicit constants independent of $n$ and the tiles.
Condition~\ref{item:qv_approx3} in Definition~\ref{def:qv_approx} follows.

For arbitrary $k\ge 0$ note that 
repeated application of   \ref{item:Fn3} leads to 
$$  \ell(Y')\ge \ell(Z')+k. $$
This gives 
\begin{align*} \diam(Y')&=2^{-\ell(Y')+1}\le 2^{-k}  2^{-\ell(Z')+1}
=2^{-k} \diam (Z')\\ &\asymp 2^{-k} \diam(X').
\end{align*} 
Again the implicit constants here are independent of $k$, $n$, and the tiles. Condition~\eqref{eq:qv_approx4p} in 
 Lemma~\ref{lem:sub_shrink} follows from this, which, in our setting,  is equivalent 
 to condition~\ref{item:qv_approx4} in
 Definition~\ref{def:qv_approx}. We have shown that
 $\{F^n(\X^n)\}$ is a quasi-visual subdivision of $\cT$ as in
 Definition~\ref{def:qvsub}. 
 
% Since $\{\X^n\}$ is a quasi-visual subdivision by
%Proposition~\ref{prop:decomp}~\ref{item:decomp5}, it follows from
%Corollary~\ref{cor:qv_tiles_shrink} that the
%diameters of tiles in $\X^n$ tends to $0$ uniformly as
%$n\to \infty$. This is also true for  the tiles in the subdivision
%$\{F^n(\X^n)\}$ of $\cT$. Indeed, it immediately follows by
%induction from \ref{item:Fn3} that
%$\ell(F^n(X)) \ge n$ for all $n\in \N_0$ and $X\in \X^n$. Hence
%\begin{align*}
%  \max&\{\diam (F^n(X)): X\in \X^n\}
%  =
%    \max\{2^{-\ell(F^n(X))+1} : X\in \X^n\}
%    \\
%  &\le 2^{-n+1} \to 0
%  \text{ as } n\to \infty. 
%\end{align*}
% 
%We can now invoke a general fact that guarantees the existence of
%a homeomorphism $F\:T\ra\cT$ that induces the given isomorphism
%of the subdivisions in the sense that
%\begin{equation*}
%  F(X)=F^n(X)
%\end{equation*}
%for all $n\in \N_0$ and $X\in \X^n$ (see
%\cite[Proposition~2.1]{BT}).  In other words, for the subdivision
%of $\cT$ we have $\{F^n(\X^n)\}=\{ F(\X^n)\}$.

\smallskip
Now define $\Y^n=F^n(\X^n)$ for $n\in \N_0$. 
Then our previous considerations show that 
$\{\Y^n\} =\{F^n(\X^n)\}$ a quasi-visual subdivision  of $\cT$ isomorphic to $\{\X^n\}$.  The statement follows.
\end{proof}

\section{Proof of Theorem~\ref{thm:CSST_qs}.}
\label{sec:proof}

 We are now ready to prove Theorem~\ref{thm:CSST_qs}.
  Assume first that the metric space $T$ is
  quasisymmetrically equivalent to the continuum self-similar tree  $\cT$. In
  particular, $T$ is homeomorphic to $\cT$. Then   $T$ is a trivalent metric tree,
  because $\cT$ is of this type. 

 The CSST $\cT$ is doubling as a subset of $\C$, and of bounded turning, because 
 it is even quasi-convex. Since a  quasisymmetry preserves the doubling property (see
  \cite[Theorem~10.18]{He}) and the bounded turning
  property (see \cite[Theorem~2.11]{TV}), the tree   $T$ is doubling and of bounded turning as well. So $T$ is a quasiconformal tree. 
  
  We know from   Proposition~\ref{prop:CSST_unif} that $\cT$ is uniformly
  branching. It follows from Lemma~\ref{lem:unif_sepa_qs} and
  Lemma~\ref{lem:unif_dense_qs} that $T$ is also uniformly
  branching. The first implication of the statement follows.

  Conversely, assume that  $T$ is a trivalent quasiconformal tree that is  
   uniformly branching. By rescaling $T$ is necessary, we may assume that $\diam(T)=1$. Then by Proposition~\ref{eq:defvn} we can choose  $\delta>0$ small enough so that the subdivision $\{\X^n\}$ induced by the sets $\V^n$ as defined in 
   \eqref{eq:defvn} has the stated  
    properties
  in this proposition.  In particular, $\{\X^n\}$ is a quasi-visual 
   subdivision of $T$.    By    Proposition~\ref{prop:exqs}  there exists a quasi-visual subdivision $\{\Y^n\}$ of $\cT$ isomorphic
   to $\{\X^n\}$.  
 By Proposition~\ref{prop:qv_f_qs} there exists a quasisymmetry 
 $F\: T\ra \cT$ that induces the isomorphism 
 between $\{\X^n\}$ and $\{\Y^n\}$; so $T$ and $\cT$ are quasisymmetrically equivalent. 
 The proof is complete.

\section{Concluding  remarks and open problems}
\label{sec:comments}

An important numerical invariant for the quasiconformal geometry
of a metric space $X$  is its  {\em conformal dimension}  $\operatorname{confdim}(X)$ 
defined as the infimum of all Hausdorff dimensions $\dim_H (Y)$ of metric
spaces $Y$ that are quasisymmetrically equivalent to $X$
 (see  \cite{MT10} for more background on this concept). 
It is known that for every quasiconformal tree $T$ we have  
$\operatorname{confdim}(T)=1$ (see  \cite[Proposition 2.4]{Kin}
and \cite[Corollary 1.4] {BM19}). In particular, for the CSST $\cT$ we have 
$\operatorname{confdim}(\cT)=1$. One can show though that the infimum underlying the definition of  the conformal dimension is not achieved here as a minimum: if  $T$ is any tree that is quasisymmetrically equivalent to $\cT$, then actually 
$\dim_H(T)>1$. This follows from \cite[Theorem 1.6]{Az15}.

The CSST is related to various other 
metric trees appearing in probabilistic models or in complex dynamics. One of these objects is the 
 {\em (Brownian) continuum random tree} (CRT) as introduced by Aldous \cite{Al91}. One can define it as a random quotient space
 of the unit interval $[0,1]$ as follows (see 
 \cite[Corollary 22]{Al93}). We consider a sample of Brownian excursion $e=(\be(t))_{0\leq t\leq 1} $ on the interval $[0,1]$. For $s,t\in [0,1]$, we set 
$$d_{\be}(s,t) = \be(s) + \be(t) - 2 \inf\{ \be(r): \min(s,t)\le r \le 
\max(s,t)\}.$$ Then $d_{\be}$ is a pseudo-metric on $[0,1]$. 
We define an equivalence relation $\sim$  on $[0,1]$ by setting $s\sim t$ if $d_{\be}(s,t) = 0$. Then $d_{\be}$ descends to a metric on  the quotient space $T_{\be}=[0,1]/\sim$. The metric space 
$T=(T_{\be}, d_{\be})$
  is almost surely a metric tree (see \cite[Sections~2 and 3]{LeGall}). Moreover, it  is known that a sample $T$ of the  CRT is almost surely homeomorphic to the CSST (see
   \cite[Corollary 1.9] {BT}).  

Roughly speaking, this means that  the topology of the CRT is completely understood. 
One can ask whether one can say more about the geometry of the CRT, in particular its quasiconformal geometry.  A natural question is whether a sample $T$ of the CRT is  almost surely quasisymmetrically equivalent to the CSST. Actually, this is not true, because almost surely  a sample $T$ of the CRT is neither doubling  nor has uniformly relatively separated branch points, which are necessary conditions for a tree to be quasisymmetrically equivalent to the CSST. So even though the CSST  serves as a natural deterministic model for the topology of the CRT, it is not so clear whether there is deterministic model for its quasiconformal geometry. One can ask though whether  the quasiconformal geometry of the CRT is uniquely determined almost surely. This leads to the following question. 

\begin{problem} Suppose $S$ and $T$ are two independent samples of the CRT. Are $S$ and $T$ almost surely quasisymmetrically equivalent?
\end{problem}

This seems to be a difficult problem, because the techniques developed in Section~\ref{sec:appr-subd} do not apply here. 

Since the CRT is almost surely homeomorphic, but not
quasisymmetrically equivalent to the CSST, one can ask if they are
equivalent with respect to another notion 
(stronger than topological equivalence, but weaker than
quasisymmetric equivalence). A closely related result was recently obtained by Lin
and Rohde. They show that almost surely the CRT can be represented by a quasiconformal tree in the plane obtained from a (random)  conformal welding (see \cite{LR19}).

An important source of trees is given by  Julia sets of 
postcritically-finite polynomials without  periodic critical points in $\C$.  For example, $P(z)=z^2+i$ is such a polynomial and one can show based on the topological characterization of the CSST given in \cite{BT} that  its Julia set $\mathcal{J}(P)$  is homeomorphic to 
the CSST. One can use Theorem~\ref{thm:CSST_qs} to show the stronger statement that $\mathcal{J}(P)$ is actually quasisymmetrically equivalent to the CSST. For some related results about quasisymmetric equivalence of Julia sets and attractors of iterated functions systems see \cite{ERS10}. 

The CSST has some interesting universality properties. For example, it is not hard to see that every trivalent tree with only finitely  many branch points admits a topological embedding 
into the CSST. One can ask similar questions for quasisymmetric embeddings (i.e., quasisymmetric homeomorphisms onto the image of the map).  

\begin{problem}
  Suppose $T$ is a trivalent quasiconformal tree with uniformly
  relatively separated branch points. Does there exist a
  quasisymmetric embedding $F\: T\ra \cT$?  
\end{problem}

There are several other  ``model trees'' whose characterization up to quasissymmetric equivalence may be interesting. For example, 
the abstract construction of the CSST as outlined in the introduction can easily be modified to obtain canonical trees $\cT_m$ with branch points of a fixed  higher order $m\in \N$, $m\ge 3$: 
instead of gluing  one line segment  of length $2^{-n}$
to the midpoint $c_s$ of a line segment $s$ of length $2^{1-n}$ obtained in the $n$-th step, we glue an  endpoint of  
$m-2$ such segments  to $c_s$. Of course, $\cT=\cT_3$.

 By a modification of this procedure, one can actually define  a
 tree $\cT_\infty$ all of whose branch points have
 %infinite degree.
 infinitely many branches.
 For the construction of  $\cT_\infty$ we again  start with a line segment $J_0$ of length $2$.  The midpoint $c$ of $J_0$ divides $J_0$ into two line segments of length $1$. We glue to $c$  one of the  endpoints of  line segments of  lengths $2^{-0}, 2^{-1}, 2^{-2}, \ldots\,.$ We equip the resulting space $J_1$ with the obvious path metric. Then $J_0\sub J_1$.  We repeat this construction for each of the line segments obtained in this way,  that is, a line segment $s$ of length $L$ is divided into two at its midpoint $c_s$ and segments of lengths $2^{-1}L, 2^{-2}L, 2^{-3}L, \ldots$ are glued in at $c_s$. Proceeding in this way, we obtain an ascending  sequence of compact geodesic metric spaces $J_0\sub J_1\sub \dots  \,. $ Then $J=\bigcup_{n\in \N_0}J_n$ carries a geodesic   metric that agrees with the metric on $J_n$ for each $n\in \N_0$. Now  $\cT_\infty$ is defined of the completion of $J$ with respect to this metric.  
 
 One can show that $\cT_\infty$ 
 is  a quasiconformal tree; in particular, it is  doubling. To guarantee this property,  it is important that length of the segments  segments glued to a midpoint $c_s$ of a segment $s$ of length $L$ as above decrease at a geometric rate.  Gluing segments of lengths $L/2, L/3, L/4, \dots$, for example,  would not result in a doubling space. 
 
The tree $\cT_\infty$ admits a topological characterization similar 
to the CSST: a metric tree $T$ is homeomorphic to $\cT_\infty$ if and only if branch points are dense in $T$ and each branch point of $T$ has infinitely many branches (this follows from results in a   previous version of \cite{BM} that is available on arXiv).

Moreover, $\cT_\infty$ has an interesting topological universality property: every metric tree $T$ admits a topological embedding 
 $F\: T\ra\cT_\infty$  (see \cite[Section 10.4]{Na}).

 It would be interesting to establish a similar universality property 
 for the quasiconformal geometry of $\cT_\infty$. The following question seems very natural in this context. 
 
 \begin{problem}
   Suppose $T$ is a quasiconformal tree with uniformly relatively
   separated branch points. Is there a quasisymmetric
   embedding $F\: T\ra \cT_\infty$?
\end{problem}


\begin{thebibliography}{BM17}

\bibitem[Al91]{Al91} D.~Aldous, {\em The continuum random tree. I.}  Ann.\ Probab.\ 19 (1991), 1--28. 

\bibitem[Al93]{Al93} D.~Aldous,
 {\em The continuum random tree. III.}  
Ann. Probab. 21 (1993),  248--289.
 
\bibitem[Az15]{Az15} J.~Azzam, {\em Hausdorff dimension of wiggly metric spaces}, 
Ark.\ Mat.\ 53 (2015),  1--36. 

\bibitem[Bo06]{Bo06} M.~Bonk, \emph{Quasiconformal geometry of fractals}, in: \emph{Proc.\
 Internat.\ Congr.\ Math. (Madrid 2006)}, Vol.\ II,  Eur.\ Math.\ Soc., 
Z\"urich, 2006,   pp.~1349--1373. 


\bibitem[BM17]{BM} M.~Bonk and D.~Meyer,
\emph{Expanding Thurston Maps},
  Amer.\ Math.\ Soc., Providence, RI, 2017. 

\bibitem[BM19]{BM19} M.~Bonk and D.~Meyer,
\emph{Quasiconformal and geodesic trees}, Fund.\ Math., to appear.

%\bibitem[BS00]{BS} M.~Bonk and O.~Schramm,   
%{\em Embeddings of Gromov hyperbolic spaces}, 
%{Geom.\ Funct.\ Anal.}
%{10} (2000),  266--306. 

\bibitem[BT19]{BT} M.~Bonk and H.~Tran,
\emph{The continuum self-similar tree},
preprint, 2020, \url{https://arxiv.org/abs/1803.09694}.

%\bibitem[BH99]{BH}
%  M.R.~Bridson and A.~Haefliger,
%  \emph{Metric Spaces of Non-Positive Curvature},
%  Springer, Berlin, 1999.

% \bibitem[BS07]{BuS} S.~Buyalo and  V.~Schroeder, {\em 
%Elements of Asymptotic Geometry},  
% Europ.\  Math.\  Soc., Z\"urich, 2007. 

\bibitem[DS97]{DS} G.\ David and S.\ Semmes,   
\emph{Fractured Fractals and Broken Dreams}, 
Oxford Lecture Ser. in Math. and its Appl.\ 7,  The Clarendon Press,
Oxford Univ.\ Press, New York, 1997. 

\bibitem[ERS10]{ERS10}
K.I.~Ero\u glu, S.~Rohde, and B.~Solomyak, 
{\em Quasisymmetric conjugacy between quadratic dynamics and iterated function systems},  Ergodic Theory Dynam.\ Systems 30 (2010), 1665--1684. 

\bibitem[Fa03]{Fa03}  K.~Falconer,  {\em Fractal Geometry}, 2nd ed., Wiley,  Hoboken, NJ, 2003. 



\bibitem[He01]{He} J.~Heinonen, \emph{Lectures on Analysis on Metric Spaces},
Springer, New York, 2001.

\bibitem[HM12]{HM} D.~Herron and D.~Meyer, 
{\em Quasicircles and bounded turning circles modulo bi-Lipschitz maps}, Rev.\ Mat.\ Iberoamericana 28 (2012), 603--630. 

\bibitem[Ki18]{Ki} J.~Kigami,
\emph{Weighted partition of a compact metrizable space, its
  hyperbolicity and Ahlfors regular conformal dimension},
preprint, 2018, \url{https://arxiv.org/abs/1806.06558}. 

\bibitem[Kin17]{Kin} K.~Kinneberg, 
\emph{Conformal dimension and boundaries of planar domains}, Trans.\ Amer.\ Math.\ Soc.\  369 (2017), 6511--6536.

\bibitem[Ku68]{Ku68} K.~Kuratowski, {\em Topology. Vol. II.}  Academic Press, New York-London; PWN---Polish Scientific Publishers, Warsaw, 1968.


\bibitem[LG06]{LeGall} J.-F. Le Gall,  {\em Random real trees}, Ann. Fac. Sci. Toulouse Math.
(6) 15 (2006), 35--62.

\bibitem[LR18]{LR19} P.~Lin and S.~Rohde,
  \emph{Conformal welding of dendrites},
  preprint, 2018.

\bibitem[MT10]{MT10}
  J.M.~Mackay and J.T.~ Tyson, 
  \emph{Conformal dimension. Theory and application.} 
  University Lecture Series 54. Amer.\ Math.\ Soc.,
  Providence, RI, 2010. 



\bibitem[Me02]{Me02} 
  D.~Meyer,
  \emph{Quasisymmetric embedding of self similar surfaces and origami with rational maps}, Ann.\ Acad.\ Sci.\ Fenn.\ Math.~27 (2002),
 461--484. 
 
 \bibitem[Me11]{Me11}
  D.~Meyer, 
  \emph{Bounded turning circles are weak-quasicircles},
  Proc.\ Amer.\ Math.\ Soc.\ 139 (2011), 1751--1761.

 
   
  \bibitem[Na92]{Na} S.B.~Nadler, Jr., {\em Continuum Theory. An Introduction.} Monographs and Textbooks in Pure and Appl. Math.\ 158. Marcel Dekker, New York, 1992. 
  
   \bibitem[Ro01]{Ro} S.~Rohde, {\em Quasicircles modulo bilipschitz maps}, Rev.\ Mat.\ Iberoamericana 17 (2001), 
   643--659. 

%
%\bibitem[Me10]{Me10}
%  D.~Meyer
%  \emph{Bounded turning circles are weak-quasicircles}
%  Proc.\ Amer.\ Math.\ Soc.\ 139 (2010), 1751--1761.

\bibitem[TV80]{TV}
  P.~Tukia and  J.~V\"ais\"al\"a, 
  {\em Quasisymmetric embeddings of metric spaces}, 
  {Ann.\ Acad.\ Sci.\ Fenn.\ Ser.\ A I Math.}
  5 (1980), 97--114.

  
%\bibitem[Ba]{Ba} R.F.~Bass, \emph{Probailistic Techniques in Analysis}, Springer, New York, 1995.
%
%\bibitem[Ci]{Ci} Z.~Ciesielski, \emph{H\"older conditions for realizations of Gaussian processes},  Trans.\  Amer.\ Math.\ Soc.\ 99 (1961),  403--413.
%
%\bibitem[IN]{IN} K.~It{\^ o} and M.~Nisio,  \emph{On the convergence of sums of independent Banach space valued random variables},  Osaka J.\  Math.\ 
%5,  35--48. 
%
%\bibitem[Ba]{Ba} F.B.~Knight, \emph{Essentials of Brownian Motion and Diffusions}, Amer\ Math.\ Soc., Providence, RI, 1981. 
% 
%\bibitem[Pin]{Pin} M.A.~Pinsky, \emph{Introduction to Fourier Analysis and Wavelets}, Brooks/Cole, Pacific Grove, CA, 2002.
%
%
%\bibitem[RW]{RW} L.C.G.~Rogers and D.~Wiliams, \emph{Diffusions, Markov Processes, and Martingales, Vol.~1}, 2nd ed., Cambridge Univ.\ Press, Cambridge, 1994. 
%


%\bibitem{} \`E. B. Vinberg, \emph{Real entire functions with
%prescribed critic values,} (Russian) Problems in group theory
%and in homological algebra, Yaroslav. Gos. Univ., Yaroslavl', 1989,
%127--138.

%\bibitem{lV50} L. I. Volkovyski, \emph{Proceedings of Steklov Math. Inst.,}
%vol. 34, Acad. Sci. USSR, Moscow, 1950. (In Russian)

%\bibitem{hW55} H. Wittich, \emph{Neuere Untersuchungen \"uber Eindeutige
%Analytische Funktionen,} Springer-Verlag, 1955.

\bibitem[Wh63]{Wh} G.T.~Whyburn, {\em Analytic Topology},   Colloquium.\ Publ., Vol.~28,  Amer.\  Math.\  Soc., Providence, RI, 1963. 



  




\end{thebibliography}
\end{document}